\documentclass[11pt]{article}
\usepackage{amscd,latexsym, amsthm,amsfonts,amssymb,amsmath,amsxtra}
\usepackage[all]{xy}
\usepackage{setspace}
\usepackage{cases}
\usepackage{graphicx,epsfig}
\usepackage{extarrows}
\usepackage{mathrsfs}

\usepackage{enumerate}
\usepackage{eucal}

\usepackage{hyperref}
\hypersetup{colorlinks = red, linkcolor  = black}
\usepackage{upgreek}
\entrymodifiers={+!!<0pt,\fontdimen22\textfont2>}

\newcommand{\ad}{{\mathrm{ad}}}

\newcommand{\cris}{{\mathrm{cris}}}

\newcommand{\dR}{{\mathrm{dR}}}

\newcommand{\Fil}{{\mathrm{Fil}}}
\newcommand{\Frob}{{\mathrm{Frob}}}

\newcommand{\Gal}{{\mathrm{Gal}}}
\newcommand{\GL}{{\mathrm{GL}}}

\newcommand{\Hom}{{\mathrm{Hom}}}

\newcommand{\id}{{\mathrm{id}}}
\renewcommand{\Im}{{\mathrm{Im}}}

\newcommand{\Isom}{{\mathrm{Isom}}}

\newcommand{\Lie}{{\mathrm{Lie}}}

\renewcommand{\mod}{\ \mathrm{mod}\ }

\newcommand{\Res}{{\mathrm{Res}}}

\newcommand{\Sh}{{\mathrm{Sh}}}

\newcommand{\Spec}{{\mathrm{Spec}}}

\newcommand{\HH}{{\mathrm{H}}}

\theoremstyle{remark}
\newtheorem{theorem}{\rm{\textbf{Theorem}}}[subsection]
\newtheorem{corollary}[theorem]{\rm{\textbf{\textbf{Corollary}}}}
\newtheorem{lemma}[theorem]{\rm{\textbf{Lemma}}}
\newtheorem{proposition}[theorem]{\rm{\textbf{\textbf{Proposition}}}}

\newtheorem{definition}[theorem]{\rm{\textbf{Definition}}}
\newtheorem{example}[theorem]{\rm{\textbf{Example}}}
\newtheorem{remark}[theorem]{\rm{\textbf{Remark}}}

\DeclareMathAlphabet{\pazocal}{OMS}{zplm}{m}{n}

\numberwithin{equation}{subsection}
\author{Qijun Yan}
\title{Ekedahl-Oort stratifications of Shimura varieties via Breuil-Kisin windows}
\date{}
\begin{document}
	\maketitle
	\section {Abstract}
	Let $ S $ be the special fibre of the good reduction of a Shimura variety of Hodge type.  
	By constructing adapted deformations for the associated $p$-divisible groups of $ S $, we manage to construct a morphism from $S$ to some quotient sheaf of the loop group associated with $S$. We show that the geometric fibres of this morphism give back the Ekedahl-Oort strata of $S$.  For any geometric point $ x $ of $ S $, we give a deformation over $ W(k(x)) $ of the $ p $-divisible group associated with $ x $ by (non-canonically) constructing a  Breuil-Kisin window (which corresponds to a $ p $-divisible group over $ W(k(x)) $ by the work of Kisin). This map in a sense gives a conceptual interpretation of Viehmann's new invariants `` truncations of level one of elements in the loop group''.
\setlength{\parindent}{0in}

	\section{Introduction}
We fix a prime number $ p>2 $ throughout this article.
\subsection{Ekedahl-Oort stratification}
We start by giving a short review of the development of the Ekedahl-Oort stratification of Shimura varieties, following \cite{GoldringKoskivirtaStrataHassinvariants}. 

The Ekedahl-Oort stratification was initially defined and studied by Ekedahl and Oort for $ \mathcal{A}_{g,\mathbb{F}_p} $, the special fibre of the moduli space $ \mathcal{A}_g $ of principally polarized abelian varieties of dimension of $ g $, which can be seen as the Shimura variety associated to the group $ \mathbf{GSp}_{2g} $. In \cite{OortAstratificationofModuliSpaces} Oort defined a stratification for 
$ \mathcal{A}_{g,\mathbb{F}_p} $ by isomorphism
classes of the BT1's (BT1 stands for truncated Barsotti Tate group of level 1). The strata are listed in loc. cit. by the so called ``elementary sequences". Moonen later on gave a group-theoretical classification of BT1's with PEL-structures in \cite{MoonenGroupschemeswithadditionalstructure}, where he describes the isomorphism classes as a certain Weyl cosets $ {}^{J}W $ of the group $ G $ of the Shimura datum associated to the PEL-structure. In a series of papers, Moonen, Wedhorn, Pink and Ziegler showed that BT1's with $G$-structures of certain type $ \chi $ give rise to an algebraic stack $ G\textsf{-Zip}^{\chi} $, which  is isomorphic to a quotient stack $[E_{\chi}\backslash G]$, where $E_{\chi}$ is some ``zip group" (see \cite{Moonen&WedhornDiscreteinvariants}, \cite{WedhornDimensionofOortStrata}, \cite{PinkWedhornZigler1}, \cite{PinkWedhornZiegler2}). 

In \cite{ViehmannWedhornEOPELtype} Viehmann and Wedhorn generalized and studied the Ekedahl-Oort stratification for the special fibre of a PEL-type Shimura variety.
In his thesis \cite{ChaoZhangEOStratification}, C. Zhang constructed  a universal $G\textsf{-Zip}$ for the special fibre  $ S $ of a general Hodge-type Shimura variety, which induces a morphism of algebraic stacks \footnote{Wortmann slightly modified Zhang's constructions in his thesis (\cite{Wortmann}) and we are here following his notations. We will also review his constructions in this article.} 
\begin{align}\label{Chao'szeta}
	\zeta: S\longrightarrow G\textsf{-Zip}^{\chi}\cong [E_{\chi}\backslash G_{\kappa}],
\end{align}
and showed that the morphism is smooth. A detailed review of the construction is given in Section \ref{SectionDefinitonofEO}. 
\begin{definition}
	An\textbf{ Ekedahl-Oort stratum} of $ S $ is defined to be a geometric fibre of the  morphism $ \zeta $ in \eqref{Chao'szeta}. 
\end{definition}

In the case of PEL type, the strata thus defined coincide with the Ekedahl-Oort strata defined in \cite{ViehmannWedhornEOPELtype}. 
The smoothness (hence openness) of $ \zeta $ provides a description of the closure relation of the strata, 
in terms of the underlying topological space of the stack $G\textsf{-Zip}^{\chi}$.

\subsection{Motivations}\label{Motivation}

Let $ G $ be a reductive group over $ \mathbb{F}_p $ and fix a Borel pair $ (B, T) $ of $ G $ (namely a maximal torus $ T $ and a Borel subgroup $ B $ of $ G $ with $ B\supset T $), all defined over $ \mathbb{F}_p $. Let  $ \mathcal{L} G$ and $ \mathcal{K}:=\mathcal{L}^{+}G $ be the loop group and the strict loop group of $ G $ respectively (Section \ref{sectionLoopSchemes}), and $ \mathcal{K}_1 $ the kernel of the reduction map $ \mathcal{K} \to G$.   Let $ k $ be an algebraically closed field of characteristic $ p $ and denote by $ W(k) $ the ring of Witt vectors of $ k $. Denote by $ \sigma: k\to k $  the absolute Frobenius of $ k $ and use the same notation for the ring endomorphism $ k[[u]]\to k[[u]] $, which is $ \sigma $ on $ k $ and fixes $ u\in k[[u]] $. 
Write $W= N_{G}T(k)/T(k) $ for the Weyl group of $ G $, with respect to $ T $. 

In \cite[Theorem 1.1]{ViehmannTrucation1} Viehmann gives a classification of the $ \mathcal{K}$-$\sigma$ conjugacy classes in $  \mathcal{K}_1(k)\backslash \mathcal{L}G(k) /\mathcal{K}_1(k) $ in terms of certain Weyl groups.  For a dominant cocharacter $ \mu: \mathbb{G}_{m,k} \to T_k $ over $ k $ (with respect to $ B_k $),  if we denote by $ C(G, \mu) $ the set of $ \mathcal{K}$-$\sigma$ conjugacy classes in $$  \mathcal{K}_1(k)\backslash \mathcal{K}(k)\mu(u)\mathcal{K}(k) /\mathcal{K}_1 (k),$$ then the result in loc. cit. implies that there is a 1-1 correspondence between $ C(G, \mu) $  and the subset $ {}^{J}W $ of the Weyl group $ W $,  where $ J$ is the type of the cocharacter $ \sigma^{-1}(\mu): \mathbb{G}_{m,k}\to T_k $.  The precise definitions of $ J $ and ${}^{J}W $  are reviewed in Section \ref{CombinatoryWeylgroups}. 

We are interested in the situation where $ G $ and $ \mu$ come from  a Shimura variety $ S $ of characteristic $ p $ (namely, the reduction modulo $ p $ of a Shimura variety). In this case,  as discussed in the previous subsection, the finite set $ {}^{J}W $ can be seen as the index set of the Ekedahl-Oort strata of this Shimura variety (see for example \cite[Section 8.2]{ViehmannWedhornEOPELtype}), and hence so can be  $ C(G, \mu)$. But note that such a bijection between $ {}^{J}W $ and $ C(G, \mu)$ is purely group-theoretic and a priory there is no direct connection from the Shimura variety in question to $ C(G, \mu)$. The main goal of this paper is to construct a map from $ S $ to a subquotient (stack, preferably) of $ \mathcal{L}G $ whose geometric points can be identified with $ C(G, \mu) $\footnote{Our cocharacter $ \mu $ is in fact a variant of $ \mu $ in \cite{ViehmannTrucation1}, but as explained in Remark \ref{Comparisionofmus} the resulting finite sets $ C(G, \mu) $ are the same.}, such that the geometric fibres of this map give back the Ekedahl-Oort strata of $ S $. This goal is basically achieved except that the target of the map we construct fails to be a stack (see Section \ref{Mainresults}). 

What makes us believe that such a connection is possible are the following observations: (1). $ p $-divisible groups over rings like $ W(k) $ can be classified in term of Breuil-Kisin modules, or equivalently Breuil-Kisin windows in our term (we give a short review of these in Section \ref{SectionreviewBKwindows}), which are modules over the power series ring $ W(k)[[u]] $ equipped with additional data; (2). Shimura varieties naturally carry 
families $p$-divisible groups (with extra structures) which determine the stratifications we want to define. This belief is also supported by other evidence; see Theorem \ref{Kisin's Classification} and the paragraph after that. Our philosophy in short is: the Frobenii of Breuil-Kisin windows will give $ C(G, \mu) $, and this turns out to be correct (see Remark \ref{RemarkIntroduction}.(1)).

In \cite{ViehmannWedhornEOPELtype} the authors determined the closure relation of Ekedahl-Oort strata via that of some truncation strata in $ \mathcal{L}G $ (\cite[Corollary 7.2]{ViehmannWedhornEOPELtype}, \cite[Corollary 4.7]{ViehmannTrucation1}). Such a criteria again depends on the bijection between $ {}^{J}W $ and $ C(G, \mu)$. One of our goals is to remedy this situation by deducing this criteria in a similar manner as the proof of \cite[Proposition 3.1.6]{ChaoZhangEOStratification}. So far we have not fully achieved this but we are confident that it can be done (see Remark \ref{RemarkIntroduction}.(2)).

\subsection{Special fibre of a Shimura variety (of Hodge type)}  
\label{SpecialfibreofShimuravarieties}
To state our main results, we need to fix some notations on the reduction of Shimura varieties (of Hodge type). All the technical terminologies and results used here are discussed or reviewed in Section \ref{GoodreductionofShimuravariety} and Section \ref{Sectionofcocharacters}.

Let ($ \textbf{G}, \textbf{X} $) be a Shimura data of Hodge type, i.e., it can be embedded into a Siegel-type Shimura data $(\textbf{GSp}, \textbf{S}^{\pm})$.  Assume that $ \textbf{G} $ has good reduction at a prime number $ p>2 $, i.e., $ \textbf{G} $  has a reductive model $\mathcal{G}$ over $\mathbb{Z}_p $, whose special fibre is denoted by $ G $. Let $ \mathsf{K}=\mathsf{K}_p\mathsf{K}^{p}\subset \textbf{G}(\mathbb{A}_f) $ be an open compact subgroup, with $ \mathsf{K}_p=\mathcal{G}(\mathbb{Z}_{p})$ and  $ \mathsf{K}^{p} \subset \mathbf{G}(\mathbb{A}_{f}^{p})$ sufficiently small, and let $\Sh_{\mathsf{K}}(\textbf{G}, \textbf{X})$ be the associated Shimura variety over the reflex field $ E $ of $(\textbf{G}, \textbf{X}) $.

Fix a place $ v $ of $ E $ above $ p $.  Denote by $ \mathcal{O}_E $ the ring of integers of $ E $. Write $ \mathcal{O}_{E,v} =W(\kappa)$  for the completion of $ \mathcal{O}_E $ at $ v $.  Kisin (\cite{KisinIntegralModels}) and Vasiu (\cite{VasiuIntegralCanonicalmodels}) have showed the existence of the integral canonical model $ \mathcal{S}_{\mathsf{K}}(\textbf{G}, \textbf{X}) $ of $\Sh_{\mathsf{K}}(\textbf{G}, \textbf{X})$ over $ \mathcal{O}_{E,v} $. Denote by $ S $ the special fibre of $ \mathcal{S}_{\mathsf{K}}(\textbf{G}, \textbf{X}) $, which is a quasi-projective and smooth scheme over $ \kappa$. 

Denote by $ [\chi]_{\mathbb{C}} $ the unique $ \textbf{G}(\mathbb{C}) $-conjugacy class of the inverse of any of the Hodge cocharacters $ \lambda: \mathbb{G}_{m}\to \textbf{G}_{\mathbb{C}} $ determined by $ (\textbf{G}, \textbf{X}) $. Denote by $ [\chi]_{\kappa} $ the reduction over $ \kappa $ of $ [\chi]_{\mathbb{C}} $ (Section \ref{Sectionofcocharacters}), which is the $ G(\kappa) $-conjugacy class of some cocharacter $ \chi: \mathbb{G}_{m,\kappa}\to G_{\kappa} $ of $ G_{\kappa} $ over $\kappa$. We fix such a cocharacter $ \chi $ and define the cocharacter
\begin{align} \label{Defofmu}
	\mu: = \Frob_{G_{\kappa}/\kappa} \circ \chi: \mathbb{G}_{m,\kappa}\to G_{\kappa}, 
\end{align}
where $\Frob_{G_{\kappa}/\kappa}: G_{\kappa}\to G_{\kappa}^{(p)}= G_{\kappa}$ is the relative Frobenius of $ G_{\kappa} $ over $ \kappa $ (see (9) in Section \ref{GeneralNotations}). The cocharacter $ \chi $ defines a parabolic subgroup $ P_{+} $ of $ G_{\kappa} $, defined over $ \kappa $ (cf. Section \ref{Sectionofcocharacters}) We can and do lift $ \chi $ to be a cocharacter $ \tilde{\chi}: \mathbb{G}_{m, W(\kappa)}\to \mathcal{G}_{W(\kappa)} $ of $ \mathcal{G}_{W(\kappa)} $ over $ W(\kappa) $, which in turn defines a parabolic subgroup $ \mathcal{P}_{+}$ of $\mathcal{G}_{W(\kappa)} $, defined over $ W(\kappa) $. Note that our ``$ \mu $'' is different from the ``$ \mu $'' in \cite{Wortmann}:
we remind the reader to be careful on this point.

\subsection{Review of Breuil-Kisin windows}\label{SectionreviewBKwindows}
Let $ k $ be an algebraically closed field extension of $ \kappa $. Breuil-Kisin modules are typically used to classify $ p $-divisible groups over a totally ramified extension $ R $ of $ W(k) $ (of arbitrary ramification index). But for our application, $ R $ is simply $ W(k) $ itself and hence no nontrivial ramification happens. 

We give below the definition of Breuil-Kisin modules and Breuil-Kisin windows, together with a classification result. In the main body of the paper we also use relative Breuil-Kisin modules developed by W. Kim (we call them Kim-Kisin modules, see Section \ref{Kisin windows and Kism-Kisin modules}) but we will not need them now since the purpose here is to provide some background knowledge so that we can present our main results in the next subsection.

Let $ \mathfrak{S} = W(k)[[u]] $ and let $ \varphi =\varphi_{\mathfrak{S}}: \mathfrak{S} \to \mathfrak{S}$ be the ring endomorphism of $ \mathfrak{S} $ which is the unique Frobenius lift $ \varphi $ on $ W(k) $ and sends $ u $ to $ u^{p} $. We write $ E(u)=u+p $, which generates the kernel of the projection $ \mathfrak{S} \to W(k)$ (as $ W(k) $-algebras) sending $ u $ to $ -p \in W(k)$. 
\begin{definition}\label{DefofBreuilKisinmodules}
	\begin{enumerate}[(1)]
		\item A \textbf{Breuil-Kisin $ \mathfrak{S} $-module} is a pair $ (\mathfrak{M}, \varphi_{\mathfrak{M}}) $, where $\mathfrak{M}  $  is a finite free $ \mathfrak{S} $-module and $\varphi_{\mathfrak{M}}:  \mathfrak{M}\to \mathfrak{M} $ is a 	$ \varphi $-linear map, such that the cokernel of the linearization 
		$$ \varphi_{\mathfrak{M}}\otimes_{\mathfrak{S}, \varphi}\id: \varphi^{*}\mathfrak{M}:= \mathfrak{M}\otimes_{\mathfrak{S}, \varphi}\mathfrak{S}\to \mathfrak{M}$$ is killed by $ E(u)$. 
		
		\item 	A \textbf{Breuil-Kisin $\mathfrak{S}$ window} is a triple $\underline{\mathcal{M}}:=(\mathcal{M}, \mathrm{Fil}\mathcal{M}, \varphi_{\mathcal{M}})$, where $ \mathcal{M} $  is a free $\mathfrak{S}$-module and $\mathrm{Fil}\mathcal{M}\subset \mathcal{M}$ a $\mathfrak{S}$-submodule of $ \mathcal{M} $ of finite type,  and where $\varphi_{\mathcal{M}}: \mathrm{Fil}\mathcal{M}\to \mathcal{M}$ is a $\varphi$-linear map 
		such that 	\begin{enumerate}[(a)]
			\item $E\cdot \mathcal{M}\subset \mathrm{Fil}\mathcal{M}$;
			\item $\mathcal{M}/\mathrm{Fil}\mathcal{M}$ is a free $W(k)$-module;
			\item the subset $\varphi_{\mathcal{M}}(\mathrm{Fil}\mathcal{M})$  generates $\mathcal{M}$ as an $\mathfrak{S}$-module. 
		\end{enumerate}
	\end{enumerate}
	
\end{definition}

A Breuil-Kisin window is also called a filtered Breuil-Kisin module in the literature, but we prefer the current terminology as it fits well with the windows theory developed by Lau and Zink.

Denote by $ \textbf{BT}(W(k)) $ the category of $ p $-divisible groups over $ W(k) $. 
\begin{theorem}[Kisin]\label{Kisin's Classification}
	There are the following equivalence  of categories:
	$$	\begin{array}{ccccc}
	\textbf{BT}(W(k)) & 	\xleftrightarrow{ \ \cong\  }	&\left\{	\begin{matrix}
	\text{Breuil-Kisin $\mathfrak{S}$-modules}
	\end{matrix} \right\}&	\xleftrightarrow{ \ \cong\  }& \left\{	\begin{matrix}
	\text{Breuil-Kisin $\mathfrak{S}$-windows}\end{matrix} \right\} \\
	& &	(\mathfrak{M},  \ \ \varphi_{\mathfrak{M}})&\longmapsto& (\varphi^*\mathfrak{M},\ \Fil\varphi^*\mathfrak{M},  \  \varphi_{\mathfrak{M}}\otimes 1), \end{array}$$
	where $\Fil\varphi^*\mathfrak{M}$ is defined in the proof of Proposition \ref{EquivalenceofModules&Windows} (we will not need it explicitly in this Introduction). The essential part of the Theorem is the first equivalence which is proved quite indirectly and hence we omit its description here. 
	
	Moreover, if $ (\mathcal{M}, \mathrm{Fil}{\mathcal{M}}, \varphi_{\mathcal{M}})$ is the associated Breuil-Kisin window of a $ p $-divisible group $ H $ over $ W(k) $, then $ (\mathcal{M}, \varphi_\mathcal{M}) \otimes \mathfrak{S}/ (u) $ is canonically identified with the (contravariant) Dieudonn\'e module  $\mathbb{D}^{*}(H)(W(k))$ of $ H $, together with its Frobenius map, and $ (\mathcal{M}, \mathrm{Fil}{\mathcal{M}}) \otimes \mathfrak{S}/ (E(u)) $ is canonically identified with $ \mathbb{D}^{*}(H)(W(k)) $ together with its Hodge filtration (cf. \cite[Corollaire 3.3.5]{BBMTheoriedeDieudonnecristalline}). 
\end{theorem}

There is also a notion of torsion Breuil-Kisin modules, and hence torsion Breuil-Kisin windows. If $\underline{\mathcal{M}} $ is the associated Breuil-Kisin window of $ H $, then the reduction modulo $ p $ of $ \underline{\mathcal{M}} $ is the associated torsion Breuil-Kisin window associated to the BT1, $ H[p] $. 	

\subsection{Main results} \label{Mainresults}

\begin{enumerate}[(1)]
	
	\item  We define a subquotient scheme $ \mathcal{D}_1 $ of $ \mathcal{L}G $  (Section \ref{specification of notations}) and construct a morphism of schemes over $ \kappa $ (Section \ref{Maintheorem1})
	\begin{align*}
		\theta: \mathrm{I}_{+}\to \mathcal{D}_{1}. 	
	\end{align*}
	More details on this morphism will be given in the next subsection.
	\item We define an fpqc quotient sheaf $ \mathcal{D}_{1}/\mathcal{K}^{\diamond} $ of $ \mathcal{D}_1 $ (obtained from an action of some group scheme $ \mathcal{K}^{\diamond} $ on $ \mathcal{D}_1 $, Section \ref{specification of notations}) such that the geometric points $ \mathcal{D}_{1}(k)/\mathcal{K}^{\diamond}(k) $ are identified with $ C(G, \mu) $, and we show that $ \theta $ induces a morphism of fpqc sheaves (Theorem \ref{MainTheorem2})
	\begin{align*}
		\eta: S\to  \mathcal{D}_{1}/\mathcal{K}^{\diamond}.
	\end{align*}
	\item We show that the geometric fibres of $  \eta$ are exactly the Ekedahl-Oort strata of $ S $ by establishing for each algebraically closed field extension $ k $ of $ \kappa $ a bijection 
	$$ \omega: \mathcal{D}_{1}(k)/\mathcal{K}^{\diamond}(k) \to E_{\chi}(k)\backslash G(k),$$ fitting into the commutative diagram
	(Proposition \ref{PropositionofEO}),
	\begin{align*}
		\xymatrix{&&\mathcal{D}_{1}(k)/\mathcal{K}^{\diamond}(k)\ar[dd]^{\omega}\\
			S(k)\ar[urr]^{\eta}\ar[drr]_{\zeta}&&\\
			&&E_{\chi}(k)\backslash G(k)}
	\end{align*}
	
\end{enumerate}

\begin{remark}\label{RemarkIntroduction}
	\begin{enumerate}[(1)]
		\item To construct the global $ \theta $, we proceed by constructing local maps from Zariski open affines of $ \mathrm{I}_+ $ to $ \mathcal{D}_1 $, and then gluing the local maps together (Section \ref{Maintheorem1}). For the local maps we use our constructions of Kim-Kisin windows (we also call them ``adapted deformations'') in Section \ref{ConstructionofKisinWindows}. But we found afterwards that if one only concerns the map $ \theta $ itself, the constructions of Kim-Kisin windows can be totally avoided (the relative crystalline Dieudonn\'e theory developed in \cite{deJongCrystalineDieudonneRigidgeometry} is still needed): this point will be reflected in the formula \eqref{defofthetaongeometricpoints} below. But what we think is more important is the conceptual interpretation of these maps: $ \theta $ is given by the Frobenius of Kim-Kisin windows and the Ekedahl-Oort stratification of $ S $ can be defined through Breuil-Kisin windows via $ \eta $.  The maps $ \theta $ and $ \eta $ conceptually explain why Shimura varieties can be related to loop groups directly and why the new invariants $ C(G, \mu) $ in \cite{ViehmannTrucation1} can be seen as the index set of Ekedahl-Oort strata of $ S $. 	
		\item We don't know yet whether $ \eta $ is smooth but it is expected to be so. Once this is proven, we can determine the closure relations of the Ekedahl-Oort strata of $ S $ as in \cite[Corollary 7.2]{ViehmannWedhornEOPELtype} for free (in fact there is a slight difference since we are using a different ``$ \mu $''). 
		
		\item The morphism $ \eta $ cannot be refined to a morphism from $ S $ to the quotient stack $[\mathcal{D}_{1}/\mathcal{K}^{\diamond}] $. Although the quotient stacks $ [\mathcal{D}_{1}/\mathcal{K}^{\diamond}]  $ and $ [E_{\chi}\backslash G] $ have the same topological spaces, they are not isomorphic as stacks for obvious reasons: the latter is algebraic while the former is not. 
	\end{enumerate}
\end{remark}


\subsection{Strategy}

We now describe the main idea of the construction of $ \theta $ and the bijection $ \omega $.  To do this, we need to be more precise and hence the text below will be slightly technical. 	

Let $ \mathcal{D} $ be the fpqc sheafification of the subfunctor of $ \mathcal{L}G $ which sends a $ \mathbb{F}_{p} $-algebra $ R $ to  $$ \mathcal{K}^{+}(R)\mu(u)\mathcal{K}^{+}(R) \subset \mathcal{L}G(R).$$ It can be shown that $ \mathcal{D} $ is represented by an infinite dimensional formally smooth scheme over $ \kappa$ (Proposition \ref{Representability of D}).
Then $ \mathcal{D}_{1} $ is defined to be the quotient of $ \mathcal{D} $ by the free action of $ \mathcal{K}_{1} $ by right multiplication. We show in Section \ref{DefinitionofK^+} that $ \mathcal{D}_{1} $ is represented by a smooth $ \kappa $-scheme of finite type.

We refer to Section \ref{TorsorsoverSandS} for the $ \mathcal{P}_{+} $-torsor $ \mathbb{I}_{+} $ over $ \mathcal{S} $ and its reduction $ \mathrm{I}_{+} $ over $ S $.  Given an element $ x\in \mathrm{I}_{+} (k)$, a lift $ \tilde{x}\in \mathbb{I}_{+}(W(k)) $ gives an isomorphism 
\begin{align*}
	\beta_{\tilde{x}}: \Lambda^{*}_{W(k)}\cong M:=\mathbb{D}^{*}(\mathcal{A}_{x}[p^{\infty}])\cong \HH^{1}_{\dR}(\mathcal{A}_{\tilde{x}}/W(k))
\end{align*}
compatible with filtrations and sending tensors $ s_{W(k)} $ to tensors $ s_{\dR} $. By transport of structure via $ \beta_{\tilde{x}} $, we obtain a (linearized) Frobenius on $ \Lambda^{*}_{W(k)} $, denoted by $ F_{\tilde{x}}^{\mathrm{lin}} : \varphi^*\Lambda^*_{W(k)}\to \Lambda^*_{W(k)}$. We define 
\begin{align*}
	\Gamma_{\tilde{x}}^{\mathrm{lin}}: \Lambda^{*}_{W(k)}\cong\varphi^*\Lambda^{*}_{W(k)}= \varphi^*\Lambda^{*}_{W(k),1}\oplus \varphi^*\Lambda^{*}_{W(k),0}\xrightarrow{1/pF_{\tilde{x}}^{\mathrm{lin}} \oplus F_{\tilde{x}}^{\mathrm{lin}} } \Lambda^{*}_{W(k)},
\end{align*} 
where $\Lambda^{*}_{W(k)}\cong\varphi^*\Lambda^{*}_{W(k)}$ is the canonical isomorphism since $ \Lambda $ is defined over $ \mathbb{Z}_p $. 
This is an isomorphism preserving tensors $ s_{W(k)} $, and hence we obtain an element in $ G(k) $, namely
\begin{align*}
	\alpha(\tilde{x}): = (\Gamma_{\tilde{x}}^{\mathrm{lin}}\mod p)^{\vee}\in G(k).
\end{align*}
Here we use the contragredient representation $(\cdot)^{\vee}: \GL(\Lambda_{\kappa}^*)\cong \GL(\Lambda_{\kappa}) $ introduced in Section \ref{Tensorsandcontragredient}. 
Finally, $ \theta(x) $ is defined to the the image in $ \mathcal{D}_1(k) $ of
\begin{align}\label{defofthetaongeometricpoints}
	\Phi_{0}(\varphi, \tilde{x}):=\alpha(\tilde{x})\mu(u) \in \mathcal{D}_1(k).
\end{align}
It turns out that $ \theta(x) $ is independent of the choice of $ \tilde{x} $. 
The unique Frobenius lift $ \varphi$ of $ W(k) $ appears in $ 	\Phi_0(\varphi, \tilde{x}) $  to keep coherence of notations in the relative setting where the Frobenius lifts are not unique.

The element $\Phi_0(\varphi, \tilde{x})$ in $ \mathcal{D}(k) $ has a realization as the reduction modulo $ p $ of the Frobenius of some Breuil-Kisin window isomorphic to the one associated to $ \mathcal{A}_{\tilde{x}} [p^{\infty}]$. In other words, $\Phi_0(\varphi, \tilde{x})$ can be realized as the Frobenius of some torsion Breuil-Kisin window isomorphic to the one associated to the BT1, $ \mathcal{A}_{\tilde{x}} [p]$. The ambiguity ``isomorphic to" here (for Frobenius maps it is the same as ``$ \mathcal{K}$-$ \sigma $ conjugate to") still exists in $ \mathcal{D}_1 (k)$, but will be killed after passing to the quotient $ \mathcal{D}_1(k)/\mathcal{K}^{\diamond}(k) $. Such a realization is the hard part of this paper. 

The  element $ \mu(u) \in \mathcal{L}G(k)$, instead of $ \chi^{(p)}(u) $ ($\chi^{(p)}$ is the ``$ \mu $'' in \cite{Wortmann}, see (9) in Section \ref{GeneralNotations} for its definition), appears in $\Phi_0(\varphi, \tilde{x})$ is due to the fact that the Frobenius lift $ \varphi: \mathfrak{S}\to \mathfrak{S} $ sends the formal variable $ u $ to $ u^p $. This also explains why the cocharacter ``$ \mu $'' in this thesis is different from the ``$ \mu $''  in \cite{Wortmann}.

For an algebraically closed field extension $ k $ of $ \kappa $, the map $ \omega $ is defined as follows.
\begin{align*}
	\begin{array}{rcl}
		\omega: \mathcal{D}_{1}(k)/\mathcal{K}^{\diamond}(k)& \longrightarrow & E_{\chi}(k)\backslash G_{\kappa}(k)\\
		h_1\mu(u)h_2& \longmapsto & E_{\chi}(k)\cdot (\sigma^{-1}(\bar{h}_2)\bar{h}_1),
	\end{array}
\end{align*}	
where for an element $h\in \mathcal{K}(k)$, we write $\bar{h}= h\mod u$.


The map $ \omega $ above (Proposition \ref{PropositionofEO}) is an analogue of the map $ \bar{\zeta} $ in \cite[Proposition 6.7]{Wortmann}. The fact that $ \omega $ is a bijection seems to be implied by the work in \cite{ViehmannTrucation1} (or combined with \cite[Proposition 6.7]{Wortmann}), but this is not completely clear to us. Due to the difference of characters in question (our $ \mu: \mathbb{G}_{m, \kappa}\to G_{\kappa} $ is not minuscule), the proof of \cite[Proposition 6.7]{Wortmann} cannot be adapted directly to our setting.

\subsection{Acknowledgement}
This paper comes out of my PhD thesis. I would like to first thank F. Andreatta (my supervisor in Milan) for initiating this work and for introducing me into this field of mathematics. I would like to specially thank U. G\"ortz for many useful discussions and helpful suggestions during my several visits to the University of Duisburg-Essen. I would like to thank T. Wedhorn for pointing out several mistakes in the thesis, as well as giving suggestions on how to fix them. Furthermore, I would like to thank B. Edixhoven, B. de Smit (my supervisor in Leiden) and C. Zhang for numerous discussions over the years.

\section{Notations and conventions}	
\subsection{General notions}\label{GeneralNotations}
\begin{enumerate}[(1)]
	\item As mentioned earlier, for a prime number $p$ in this paper, we assume that $p\geq 3$.
	
	\item All Dieudonn\'e crystals and Diudonn\'e modules are contravariant.
	
	\item If $X$ is a scheme over a ring $R$ (resp. over a scheme $ S $) and $R'$ is an $R$-algebra (resp. $S'$ is an $S$-scheme), we often denote by $X_{R'}$ (resp. $X_{S'}$) the pullback of $X$ along the ring homomorphism $R\to R'$ (resp. along the structure morphism $ S'\to S $). Sometimes we suppress the subscripts $R'$ (resp. $ S' $) if the base ring (resp. base scheme)  is clear from the context.
	
	\item Similar convention as in (3) applies for modules, $ p $-divisible groups,  stacks, etc.
	\item If $R$ is a ring of characteristic $p$ or a $\mathbb{Z}_p$-flat $p$-adic ring and $\varphi_R$ a Frobenius lift of $R$ (i.e., the reduction mod $p$ of $\varphi_R$ is the absolute Frobenius of $R/pR$), we often suppress the subscript $R$ when it is clear from the context. Here and everywhere in the sequel a $ p $-adic ring is $ p $-adically complete and separated.
	\item If $ \varphi: R \to R$ is an endomorphism of rings and $ M $ is an $ R $-module, we use the following notations interchangeably
	\begin{align*}
		M^{(\varphi)}:=\varphi^{*}M:=M\otimes_{R, \varphi}R. 
	\end{align*}	
	\item  If $R$ is a  ring with $r\in R$, and $M$ an  $R$-module, we sometimes simply denote by $M[\frac{1}{r}]$ the module $M\otimes_R R[\frac{1}{r}]$ over $R[\frac{1}{r}]$ and $f[\frac{1}{r}]: M[\frac{1}{r}]\to N[\frac{1}{r}]$ the induced homomorphism of a homomorphism $f: M\to N$ of $R$-modules.
	
	\item By a linear algebraic group over a ring $ R $ we mean an affine group scheme of finite presentation over $ R $. By a reductive group over $ R $ we mean a connected smooth affine group scheme over $ R $ such that for each geometric point $ s $ of $ \Spec R $, the pullback $ G_s $ is a connected reductive group over $ s $.  
	\item Let $ k $ be a field of characteristic $p$. For any $ k $-scheme $ S $, we denote by $\Frob_S: S\to S$ the absolute Frobenius endomorphism of $S$, and by $(\ )^{(p)}$ the pull-back along $ \Frob_S $ of a scheme or a sheaf or a morphism over $S$. For example, if $G$ is a reductive group over $k$, we denote by $G^{(p)}$ the pull-back of $G$ along $\Frob_{\Spec k}$; if $P$ is a $G$-torsor over $S$, then $P^{(p)}$ is a $G^{(p)}$-torsor over $ S $; the pull back $\lambda^{(p)}: \mathbb{G}_{m,k}\cong \mathbb{G}^{(p)}_m\to G^{(p)}$ along $\Frob_{\Spec k}$ of a cocharacter $\lambda: \mathbb{G}_{m,k}\to G$ (over $k$) is a cocharacter  of $ G^{(p)} $ (over $ k $). In particular, if $ G $ is defined over $ \mathbb{F}_p $, then $ \lambda^{(p)} $ is again a cocharacter of $ G $ since we then have canonical isomorphism $ G^{(p)} \cong G$.  When $k$ is perfect, we have an automorphism of  $X_*(G):=\Hom_k(\mathbb{G}_{m,k}, G) $, given by
	\begin{align} \label{Actionofsigmaoncocharacters}
		\sigma: X_*(G)\longrightarrow X_*(G), \ \ \lambda\longmapsto \lambda^{(p)}.
	\end{align}
	Let $ \mathrm{G} $ be reductive model over $ \mathbb{F}_p $ of $ G $, i.e., $ \mathrm{G}\otimes k= G $, and $ \sigma: k\to k(x\mapsto x^p) $ the Frobenius element in the Galois group $ \Gal(k/\mathbb{F}_p) $. Then the action $ \sigma $ in \eqref{Actionofsigmaoncocharacters} is the same as the Galois action of $\sigma\in  \Gal(k/\mathbb{F}_p)$ on  $X_*(G)$, i.e., we have 
	\begin{align}
		\lambda^{(p)}= \big(\mathbb{G}_{m,k}\xrightarrow{\id\otimes \sigma}\mathbb{G}_{m,k}\xrightarrow{\lambda} \mathrm{G}\otimes k\xrightarrow{\id\otimes \sigma^{-1}}\mathrm{G}\otimes k=G\big).
	\end{align}
	
	\item 	If $ G $ is a group scheme over a field of characteristic $ p $,  and $ R $ is a $ k $-algebra, then for any $ x\in G(R) $, by $ \sigma(x) $ we mean the image of $ x $ in $ G^{(p)}(R) $ under the relative Frobenius $ G\to G^{(p)} $. Note that there are two Frobenii $ \sigma $ and $ \varphi $ on $ \mathcal{L}G $ and on $ \mathcal{L}^{+}G $ (Section \ref{sectionLoopSchemes}). The difference between them is discussed in Section \ref{Section2Frobeni}.
	
\end{enumerate}

\subsection{Tensors and contragredient representations}\label{Tensorsandcontragredient}

We introduce below  the ``contragredient representations" following Wortmann's thesis (\cite{Wortmann}) but note that it is  used slightly differently in this paper.

Let $R$ be a ring and $M$ a finite locally free $R$-module. Denote by $ M^{*} $ the dual $ R $-module of $ M $ and by $ M^{\otimes} $ the direct sum of all $ R $-modules that arise from $ M $ by applying the operation of taking duals, tensor products, symmetric powers and exterior powers a finite number of times. An element of $ M^{\otimes} $ is called a \textbf{tensor} over $ M $. For an isomorphism $ f: M_1\cong M_2$ of finite locally free $ R $-modules, we have an induced isomorphism $ (f^{-1})^{*}: M_1^{*}\to M_2^{*} $, and hence $ f^{\otimes}: M_1^{\otimes}\to  M_2^{\otimes} $.

For any $ M $ as above, there is a canonical isomorphism of group schemes, called the\textbf{ contragredient representation } of $\GL(M)$, defined as  $$(\cdot)^{\vee}: \GL(M)\cong \GL(M^*), \ \ \ g\longmapsto g^{\vee}:=(g^{-1})^{*}.$$
Through the contragredient representation of $M $, $M^{\otimes} $ is natually identified with $ (M^{*})^{\otimes} $. 
If $s\subset M^{\otimes}$ is a set of tensors over $M$, which defines a subgroup $G\subset \GL(M)$, then they also defines a subgroup $\{g^{\vee}|g\in G\}\subset \GL(M^{*}),$ when considered as tensors over $M^{*}$. Since $ M $ is canonically identified with $ (M^{*})^{*} $, we also have the contragredient of $ \GL(M^{*}) $,
\begin{align*}
	(\cdot)^{\vee}: \GL(M^{*})\to \GL(M),
\end{align*}
which we shall use frequently in the sequel.

\section{Loop groups}\label{sectionLoopSchemes}
\subsection{Representability of loop groups}

\begin{definition}\label{DefinitionofLoopGroups}
	Let $k$ be a field and $G$ a linear algebraic group over $k$. The \textbf{(algebraic) loop group} of $G$, denoted by $\mathcal{L}G$, is the fpqc sheaf of groups whose $A$-valued points for a $k$- algebra $A$ is given by $\mathcal{L}G(A)=G(A((u)))$, where $A((u))$ is the ring of Laurent series with coefficients in $A$.
	Let $\mathcal{K}:=\mathcal{L}^+G\subset \mathcal{L}G$ be the subgroup of $\mathcal{L}G$ with $\mathcal{K}(A)=G(A[[u]])$, where $A[[u]]$ is the ring of formal power series with coefficients in $A$. It is called the \textbf{(strict) positive loop group } of $G$. We set $\mathcal{K}_1$ to be the kernel of the reduction map $\mathcal{K}\to G$.
\end{definition}

In fact, the functors $\mathcal{L}^+ (\cdot)$ and $\mathcal{L}(\cdot)$ can be defined more generally. For example, if $X$ is a scheme over $k[[u]]$, and $\mathcal{X}$ a scheme over $k((u))$, we can define the \textbf{positive loop functor} $\mathcal{L}^+X$ and the \textbf{loop functor} $\mathcal{L}\mathcal{X}$ as follows: for any $k$-algebra $A$,
$$\mathcal{L}^+X(A):=X(A[[u]]), \; \; \;\;\mathcal{L}\mathcal{X}(A):=X(A((u))).$$
These are the so called the \textbf{twisted} versions in \cite{TwistedLoopGroupsRapoport&Pappas}) to distinguish the cases where $ X $ and $\mathcal{X}$ come from $ k $-schemes. Note that if $ X $ does come from a $ k $-scheme, i.e., $ X=X_0\otimes_{k} k[[u]] $ for some $ k $-scheme $ X_0 $, then we have $ \mathcal{L}^{+}G(X)= \mathcal{L}^{+}G(X_0)$; the similar assertion holds for the functor $ \mathcal{L}(\cdot) $ as well. We shall not need such generalities in the sequel. 

In what follows in this section, we let $ G $ be a reductive group over $ k $. 

\begin{lemma}[]\label{RepStrictLoopGroups}
	The positive loop group  $\mathcal{L}^+G $ is represented by an affine group scheme (of infinite type) over $k$.  
\end{lemma}

For the proof the preceding lemma, one may see Proposition 3.2.1 in B. Levin's thesis \cite{LevinThesisGvaluedKisinModules}. The basic idea is to embed $ G $ into $ \GL_{n, k} $, and then reduce to show that $ \mathcal{L}^+\GL_{n, k}  $ is represented by an affine group scheme and that $ \mathcal{L}^+{G} \hookrightarrow  \mathcal{L}^+\GL_{n, k}  $
is a closed embedding. We refer to the discussion before Definition 1.1 in \cite{TwistedLoopGroupsRapoport&Pappas} for the equations defining this embedding.

Following the conventions in \cite{TwistedLoopGroupsRapoport&Pappas}, we call an fpqc sheaf $ Y $ over $ k $ an \textbf{ind-scheme} if it is the inductive limit of the functors associated to a directed system $\{Y_i\}$ of $k$-schemes. We say $ Y $ is \textbf{strict} if in addition, the transition morphisms are closed embeddings. A \textbf{group ind-scheme} is an ind-scheme which is a group object in the category of ind-schemes. We say $Y$ is \textbf{ind-affine} (resp. \textbf{ind-finite type, ind-projective}, etc) if each $Y_i$ can be taken to be affine (resp. ind-finite type,  projective, etc).

\begin{proposition}[{\cite[Proposition 3.2.4.]{LevinThesisGvaluedKisinModules}}]
	The loop group  $\mathcal{L}G$ is represented by a strict ind-affine  group ind-scheme over $k$.
\end{proposition}

We only explain here the basic idea of the proof. In the case $ G=\GL_{n, k} $, this is well-known (see \cite{Beauville1994}).  For each $N\geq 0$, denote by $\GL_{n,k}^{(N)}$ the subfunctor of $\mathcal{L}\GL_{n,k}$ given by
$$\GL_{n,k}^{(N)}(R)=\{M\in \GL_n(R((u)))| \text{\ both } M \text{ and } M^{-1} \text{ have at most } N  \text{ poles} \}.$$ One shows that each $\GL_{n,k}^{(N)}$ is represented by an affine scheme (not group scheme, unless $ N =0$) and the transition morphisms are closed immersions.  Clearly one has then $$\mathcal{L}\GL_{n,k}=\bigcup\limits_{N=0}^{\infty} \GL_{n,k}.$$
For the general case, again we take an embedding  $ G\hookrightarrow \GL_{n, k} $ and 
set \begin{align}\label{Embedding of loop group}
	G^{(N)}=\mathcal{L}G\cap \GL_{n,k}^{(N)}.
\end{align} Then one finishes the proof by showing that each $ G^{(N)} $ is a closed subscheme of $ \GL_{n,k}^{(N)}$.

\subsection{A subfunctor $\mathcal{D}(G, x)$ of the loop group $\mathcal{L}G$}\label{Sectionsubfunctor}
Take $x\in \mathcal{L}G(k)$ and let $\mathcal{C}(G, x)$ be the subfunctor of $\mathcal{L}G$ which associates to a $k$-algebra $R$ the subset $\mathcal{K}(R)x\mathcal{K}(R)$ of $\mathcal{L}G(R)$. Denote by $\mathcal{D}'(G, x)$  the fpqc sheafification of $\mathcal{C}(G,x)$. We show in this subsection that $\mathcal{D}'(G, x)$ is represented by a subscheme of $\mathcal{L}G$. 

Let $ G^{(i)} $ be as in \eqref{Embedding of loop group}, the we have $\mathcal{L}G=\lim\limits_{\longrightarrow}G^{(i)}$. Since the point $x$ lies in some $G^{(i)}(k)$, a natural idea is to consider the group action of schemes
\begin{align}\label{Actionof2K}
	\rho: \mathcal{K}\times \mathcal{K}\times G^{(i)}\to G^{(i)}, (g, h, z)\mapsto gzh^{-1}
\end{align}
and the orbit map $$\rho_x: \mathcal{K}\times \mathcal{K}\to G^{(i)}, (g, h)\to gxh^{-1}.$$ Indeed, if the orbit map $\rho_x$ has locally closed image, one would like to hope that the sheaf $\mathcal{D}'(G, x)$ is represented by a subscheme (maybe not reduced) of $G^{(i)}$ with underlying topological space $|\Im(\rho_x)|$. This idea turns out to be correct. However, since both $\mathcal{K}$ and $G^{(i)}$ are infinitely dimensional and it is not clear whether the morphisms $\rho, \rho_x$ are of finite type or not, many of the techniques and results of algebraic group actions on algebraic schemes (among them is Chevalley's theorem on constructible sets) do not apply directly. For this reason we need to pass to the affine Grassmannians of $G$ 

The \textbf{affine Grassmannian} of $G$, denoted by $\mathrm{Gr}$, is the fpqc sheafification of the presheaf which, to every $k$-algebra $R$, associates the quotient $\mathcal{L}G(R)/\mathcal{K}(R)$. Let $\mathcal{L}G=\lim\limits_{\longrightarrow}G^{(i)}$ be as above.
If we denote by $(G^{(i)}/\mathcal{K})^{\circ}$ the presheaf  $R\mapsto G^{(i)}(R)/\mathcal{K}(R)$ and $\mathrm{Gr}^{(i)}$ the fpqc sheafification of $(G^{(i)}/\mathcal{K})^{\circ}$, then each $\mathrm{Gr}^{(i)}$ is represented by a finite type $k$-scheme and we have $\mathrm{Gr}=\lim\limits_{\longrightarrow} \mathrm{Gr}^{(i)}$. In other words, the affine Grassmannian $\mathrm{Gr}$ is an ind-scheme of ind-finite type.  Moreover,  since $ G $ is reductive, each $\mathrm{Gr}^{(i)}$ is projective and hence $\mathrm{Gr}$ is ind-projective. In the case of $G=\GL_n$, each $\mathrm{Gr}^{(i)}$ is a closed subscheme of a finite disjoint union of usual Grassmannians. 

Consider the left action $$\theta: \mathcal{K}\times (G^{(i)}/\mathcal{K})^{\circ}\to (G^{(i)}/\mathcal{K})^{\circ}, \ (g, x\mathcal{K}(R))\mapsto gx\mathcal{K}(R)$$
of $\mathcal{K}$ on $(G^{(i)}/\mathcal{K})^{\circ}$ as functors.  Then $\theta$ extends naturally to a left action of $\mathcal{K}$ on $\mathrm{Gr}^{(i)}$, which we denote again by $\theta$. Now we have a commutative diagram of morphisms of presheaves 
\begin{equation*}
	\xymatrix{\mathcal{K}\times \mathcal{K} \times G^{(i)}\ar[d]^{\mathrm{pr}_1\times p}\ar[rr]^{\  \ \ \ \ \rho}\ar[d]^{p}&& G^{(i)}\ar[d]^{p}\\
		\mathcal{K}\times (G^{(i)}/\mathcal{K})^{\circ}\ar[d]^{\id \times i}\ar[rr]^{\ \ \ \theta}&&(G^{(i)}/\mathcal{K})^{\circ}\ar[d]^{i}\\
		\mathcal{K}\times \mathrm{Gr}^{(i)}\ar[rr]^{\ \ \ \theta}&&\mathrm{Gr}^{(i)}}
\end{equation*}
where $p: G^{(i)}\to (G^{(i)}/\mathcal{K})^{\circ}$ is the natural projection and $i: (G^{(i)}/\mathcal{K})^{\circ}\hookrightarrow \mathrm{Gr}^{(i)}$ is the natural inclusion (note that the presheaf $(G^{(i)}/\mathcal{K})^{\circ}$ is separated). Denote by  $\bar{x}$ the image of $x$ in $(G^{(i)}/\mathcal{K})^{\circ}(k)\subset \mathrm{Gr}^{(i)}(k)$ and by $\theta_{\bar{x}}: \mathcal{K}\to \mathrm{Gr}^{(i)}$ the orbit map of $\theta$ at $\bar{x}$.
\begin{lemma}\label{Locally Closed lemma}
	The topological image $|\theta_{\bar{x}}(\mathcal{K})|$ is a locally closed subset of $\mathrm{Gr}^{(i)}$.  
\end{lemma}
\begin{proof} 
	Note that we have 
	$\mathcal{K}=\lim\limits_{\longleftarrow}G_j, $ with 
	$$G_j:= \Res_{(k[u]/u^j)/k}G_{k[u]/u^j},$$
	the Weil restriction of $G_{k[u]/u^j}$ over $k$. One may check that the action of  $ \mathcal{K}$ on  $(G^{(i)}/\mathcal{K})^{\circ}$ factors through $G_{2i}$ (first do the case of $G=\GL_n$ and for the general case, embed $G$ into some $\GL_n$ as in the definition of $G^{(i)}$). Hence the action of $ \mathcal{K}$ on  $\mathrm{Gr}^{(i)}$ also factors through $G_{2i}$. Since $G_{2i}$ is again an algebraic group over $k$, one gets the result by applying Proposition 1.52. (b) in \cite{MilneAG} for the induced orbit map $G_{2i}\to \mathrm{Gr}^{(2i)}$ in question.
\end{proof}
We define $O_{\bar{x}}$ to be the unique reduced subscheme of $\mathrm{Gr}^{(i)}$ with underlying topological space $|\theta_{\bar{x}}(\mathcal{K})|$. The subscheme $O_{\bar{x}}$ is usually called the \textbf{Schubert cell} of $\bar{x}$ in $\mathrm{Gr}$. It is clear that the orbit map $\theta_{\bar{x}}$ factors through $O_{\bar{x}}$ because the Weil restriction $G_{2i}$ in the proof of Lemma \ref{Locally Closed lemma} is again a smooth linear algebraic group.
\begin{lemma}\label{Fpqc surjective}
	The orbit map $\theta_{\bar{x}}: \mathcal{K}\to O_{\bar{x}}$ is a surjective morphism of fpqc sheaves. 
\end{lemma}
\begin{proof}
	It follows from Lemma 9.27 in \cite{MilneAG} that the induced orbit map $G_{2i}\to O_{\bar{x}}$ is surjective as fpqc sheaves. But since $G$ is smooth, the projection map $\mathcal{K}(R)\to G_{2i}(R)$ is surjective for every $k$-algebra $R$. Hence $\theta_{\bar{x}}$ is a surjective map of fpqc sheaves.
\end{proof}

Let $\mathcal{D}(G, x)$ be the pull-back of $O_{\bar{x}}\subset \mathrm{Gr}^{(i)}$ along the projection map $\pi:= i\circ p: G^{(i)}\to \mathrm{Gr}^{(i)}.$ Then the orbit map $\rho_x$ factors through $\mathcal{D}(G, x)$. Our aim now is show that $\mathcal{D}'(G, x)$ is represented by the scheme $\mathcal{D}(G, x)$.

\begin{proposition}\label{Representability of D}
	The fpqc sheaf $\mathcal{D}'(G, x)$ is represented by the formally smooth subscheme $\mathcal{D}(G, x)$ of $\mathcal{L}G$. Moreover, the equation  $\mathcal{D}(G, x)(l)=\mathcal{K}(l)x\mathcal{K}(l)$ holds for any algebraically closed field extension $ l $ of $ k $. In particular, the subscheme $\mathcal{D}(G, x)\subset G^{(i)}$ has  underlying topological space $|\Im(\rho_x)|$.
\end{proposition}
\begin{proof}
	As the induced map of sheaves $\mathcal{D}'(G, x)\to \mathcal{D}(G,x)$ from the inclusion $\mathcal{C}(G,x)\subset \mathcal{D}(G, x)$ is injective, for the first part of the claim it suffices to show that the orbit map $\rho_x: \mathcal{K}\times \mathcal{K}\to \mathcal{D}(G, x)$ is surjective as fpqc sheaves.  Indeed, for any $k$-algebra $R$ and any point $y\in \mathcal{D}(G, x)(R)$, by Lemma \ref{Fpqc surjective} fpqc locally  its image $$\pi(y)=y\mathcal{K}(R)\in (G^{(i)}/\mathcal{K})^{\circ}(R)\subset \mathrm{Gr}^{(i)}$$ comes from a translation of $\bar{x}$ by an element of $\mathcal{K}(R)$. In other words, there exists a faithfully flat $R$-algebra $R'$ such that $gx\mathcal{K}(R')=y\mathcal{K}(R')\in(G^{(i)}/\mathcal{K})^{\circ}(R').$
	This implies that the orbit map $\rho_x: \mathcal{K}\times \mathcal{K}\to \mathcal{D}(G, x)$ is a surjective map of fpqc sheaves. In the case of $R=l$ for an algebraically closed field extension $l$ of $k$, the same argument shows that $\mathcal{D}(G, x)(l)=\mathcal{K}(l)x\mathcal{K}(l)$ (here we use the fact that $O_{\bar{x}}(l)=\mathcal{K}(l)\cdot \bar{x}$).
\end{proof}

\subsection{Two Frobenii on the loop group}\label{Section2Frobeni}
From now on, we let the field $ k $ be of characteristic $ p $ till the end of this section. For each $ i $, we still let $ G^{(i)} $ be as in \eqref{Embedding of loop group} and denote by $ \sigma_{G^{(i)}}: G^{(i)}\to (G^{(i)})^{(p)} $ the relative Frobenius of $ G^{(i)} $ over $ k $. Then all these $ \sigma_{G^{(i)}} $ induce a homomorphism 
$$ \sigma=\sigma_{\mathcal{L}G} : \mathcal{L}G\to (\mathcal{L}G)^{(p)}:=\varinjlim_{i}(G^{(i)})^{(p)}.$$

On the other hand, for each $ k $-algebra $ A $,  and each $ y\in \mathcal{L}G(A) $, if we see $ y $ as a morphism $ y: \Spec A((u)) \to G$, then the composition $ \sigma_{G}\circ y $  gives an element in $\mathcal{L}G^{(p)}(A)$. This induces a homomorphism $ \varphi(A):  \mathcal{L}G(A)\to \mathcal{L}G^{(p)}(A)$. Since $ \varphi(A) $ is functorial in $ A $, we obtain another homomorphism 
\begin{align*}
	\varphi:	\mathcal{L}G\to \mathcal{L}G^{(p)}\cong (\mathcal{L}G)^{(p)}. 
\end{align*}

In particular, if $ G $ is defined over $ \mathbb{F}_p $, then for each $ i $, $ (G^{(i)})^{(p)} $ is canonically isomorphic to $ G^{(i)} $.  And hence we have canonical isomorphisms
\begin{align*}
	\mathcal{L}G^{(p)}\cong (\mathcal{L}G)^{(p)}\cong  \mathcal{L}G.
\end{align*}
This is the case which we will mostly concern in the sequel. Note that there is a difference between $\varphi$ and $ \sigma $ as illustrated as follows: if we take $G=\mathbb{G}_m$ and let $R$ be a $k$-algebra, then
for any $f=\sum a_iu^i\in \mathcal{L}\mathbb{G}_m(R)$ we have
\begin{align}\label{EqtwoFrobs}
	\varphi(f)=\sum a_i^pu^{pi}, \ \ \ \ \sigma (f)=\sum a_i^pu^i.
\end{align}

In fact, given a perfect $ k $-algebra $ R $ (i.e., the absolute Frobenius of $ R $ is bijective), the homomorphism $ \sigma(R): \mathcal{L}G(R)\to (\mathcal{L}G)^{(p)}(R) $ is an isomorphism since for each $ i $, the homomorphism $G^{(i)}(R)\to (G^{(i)})^{(p)}(R)$ is an isomorphism. But this is almost never the case for  $ \varphi(R) $. 	
\begin{lemma} \label{TwoFrobniiLemma1}
	Suppose that $ G $ is defined over $ \mathbb{F}_p $ and $ R $ a perfect $ k $-algebra. Take $x\in \mathcal{L}G(R)  $, which corresponds to a morphism $ \Spec R((u)) \xrightarrow{x} G$. Then under the homomorphism $ \sigma^{-1}\circ \varphi: \mathcal{L}G(R)\to \mathcal{L}G(R) $, the image $ (\sigma^{-1}\circ \varphi)(x) $ corresponds to the morphism 
	\begin{align*}
		\Spec R((u))\xrightarrow{\pi} \Spec R((u)) \xrightarrow{x} G,
	\end{align*}
	where $ \pi:  \Spec R((u))\rightarrow\Spec R((u))$ corresponds to the $ R $-endomorphism of $ R((u)) $ which sends $ u $ to $ u^p $. 
\end{lemma} 
\begin{proof}
	The claim is clear for $ G=\GL_n $; see \eqref{EqtwoFrobs}.  For the general case, we let $ \iota: G\subset \GL_{n} $ be a closed embedding of group schemes over $ \mathbb{F}_p $.  
	
	Denote by $ \sigma_{G}: G\to G^{(p)}\cong G $ the relative Frobenius of $ G $ over $ k $.  Then we have the following commutative diagram
	\begin{align*}
		\xymatrixcolsep{3pc}\xymatrix{G\ar[d]^{\iota}\ar[r]^{\sigma_{G}}& G\ar[d]^{\iota}\\
			\GL_{n}\ar[r]^{\sigma_{\GL_{n}}} & \GL_{n}}
	\end{align*}
	Recall that $ \sigma_{G} \circ x$ corresponds to $ \varphi(x) $ in $\mathcal{L}G(R)$. Similarly, $ \iota\circ \sigma_{G}\circ x= \sigma_{\GL_{n}}\circ \iota \circ x $ corresponds to the image of $ \varphi(x) $ in $ \mathcal{L}\GL_{n}(R) $. 
	
	On the other hand, we also have commutative diagram 
	\begin{align*}
		\xymatrixcolsep{3pc}\xymatrix{\mathcal{L}G(R) \ar[r]^{\sigma_{\mathcal{L}G}^{-1}}\ar[d]_{\mathcal{L}\iota}& \mathcal{L}G(R)\ar[d]^{\mathcal{L}\iota}\\
			\mathcal{L}\GL_{n}(R)\ar[r]^{\sigma_{\mathcal{L}\GL_{n}}^{-1}} & \mathcal{L}\GL_{n}(R)}
	\end{align*}
	where $ \mathcal{L}\iota $ is the homomorphism induced from $ \iota $. Since we know that the statement is true for $ G=\GL_{n} $, $$ \mathcal{L}\iota(\sigma_{\mathcal{L}G}^{-1}\circ \varphi(x)) = \sigma_{\mathcal{L}\GL_{n}}^{-1}(\mathcal{L}\iota (\varphi(x)))$$ corresponds to the morphism 
	\begin{align*}
		\Spec R((u))\xrightarrow{\pi} \Spec R((u)) \xrightarrow{x} G\xrightarrow{\iota} \mathcal{L}\GL_{n}.
	\end{align*}
	If $  (\sigma^{-1}\circ \varphi) (x) = (\sigma_{\mathcal{L}G}^{-1}\circ \varphi) (x) $ corresponds to a morphism $ y: \Spec R((u))\to G $. Then we have $ \iota \circ y= \iota \circ (x\circ \pi) $ and hence $ y=x\circ \pi$. 
\end{proof}

\begin{lemma}\label{TwoFrobenii}
	Let $G$ be a linear algebraic group over a field $ k $ of characteristic $p>0$ and let $\mathcal{K}$ and $\mathcal{K}_1$ be as in Definition \ref{DefinitionofLoopGroups}. Then for any $k$-algebra $R$ and any element $g\in \mathcal{K}(R)$, we have $\varphi(g)\mathcal{K}_1(R)=\sigma(g)\mathcal{K}_1(R)$.
\end{lemma}
\begin{proof} We may assume $G=\GL_n$. For any $g\in \mathcal{K}(R)=G(R[[u]])$, write $g=g_0(1+uM)$ with $g_0\in G(R)$ and $M\in \mathrm{M}_n(R)$. It is easy to see that in fact  $1+uM$ lies in $\mathcal{K}_1(R)$. Hence we have $\varphi(g) \mathcal{K}_1=\sigma(g) \mathcal{K}_1(R)$ since they are both equal to $\sigma(g_0) \mathcal{K}_1(R)$.
\end{proof}

\subsection{Fpqc sheaves $\mathcal{D}(G, x)/\mathcal{K}^+$ and $\mathcal{D}_1(G, x)/\mathcal{K}^{\diamond}$} 
\label{DefinitionofK^+}Let $G, k$ be as in the preceding lemma. We have a semidirect product of affine group schemes $\mathcal{K}^{+}:=(\mathcal{K}_1\times_k \mathcal{K}_1)\rtimes_k\mathcal{K}$ induced by the right action defined for each $k$-algebra $R$ by
$$\begin{array}{ccc}
\mathcal{K}_1(R)\times \mathcal{K}_1(R)\times \mathcal{K}(R)&\longrightarrow & \mathcal{K}_1(R)\times \mathcal{K}_1(R)\\
\big((\alpha, \beta), \gamma\big)& \longmapsto & (\gamma^{-1}\alpha\gamma,  \varphi(\gamma)^{-1}\beta\varphi(\gamma) )
\end{array}$$
More explicitly, for any two elements $(\alpha, \beta, \gamma), (\alpha', \beta',  \gamma')\in  \mathcal{K}^{+}(R)$, the multiplication is given by
$$\big(\alpha',  \beta', \gamma'\big)\cdot \big(\alpha, \beta,  \gamma\big)=\big( \gamma^{-1}\alpha'\gamma\alpha, \varphi(\gamma)^{-1}\beta'\varphi(\gamma)\beta, \ \ \gamma'\gamma\big).$$
We refer to Section \ref{Sectionsubfunctor} for the notations $\mathcal{D}(G, x)$ and  $\mathcal{D}'(G, x)$ and for what follows we simply write 
\begin{align}
	\mathcal{D}:=\mathcal{D}(G, x)=\mathcal{D}'(G, x).
\end{align}
There is an action of the (infinite-dimensional) affine group scheme $\mathcal{K}^{+}$ on $\mathcal{D}$  given on local sections by
$$\begin{array}{rcl}
\mathcal{D} \times  \mathcal{K}^+ & \longrightarrow& \mathcal{D}\\
(t, (\alpha, \beta,  \gamma)) &\longmapsto & \alpha^{-1}\gamma^{-1}t \varphi(\gamma)\beta
\end{array}$$

Let $ \mathcal{D}_{1}:=\mathcal{D}/\mathcal{K}_{1}, $ (resp. $ G^{(i)}/\mathcal{K}_{1} $) be the quotient of $ \mathcal{D} $ (resp. $ G^{(i)} $) by $ \mathcal{K}_{1} $, which by definition is the fpqc sheafification of the presheaf 
\begin{align*}
	R\mapsto \mathcal{K}(R)z\mathcal{K}(R)/\mathcal{K}_{1}(R), \ \ \ 
	\big( \text{resp.}  \ R\mapsto G^{(i)}(R)/\mathcal{K}_{1}(R)\big)
\end{align*}
Since $ G^{(i)}/\mathcal{K}_{1} $ is represented by a proper $ k $-scheme of finite type, by a similar argument as in the proof of Lemma \ref{Locally Closed lemma} and Lemma \ref{Fpqc surjective}, one sees that $ \mathcal{D}_1 $ as a $ \mathcal{K}\times_k\mathcal{K} $-orbit of  $\bar{x}$ is represented by a smooth $ k $-scheme of finite type, where $ \bar{x} $ is the  $ k $-point of   $ G^{(i)}/\mathcal{K}_1$ induced by $ x $. In fact, $ \mathcal{D}_1 $ is a $ G=\mathcal{K}/\mathcal{K}_1 $-torsor in the \'etale topology over $ O_{\bar{x}'} \subset \mathrm{Gr}^{(i)}$ (here $\bar{x}'$ denotes the image of $ x $ in $\mathrm{Gr}^{(i)}(k) $), as it is the pull-back of the $ G $-torsor $ G^{(i)}/\mathcal{K}_1 $ over $ \mathrm{Gr}^{(i)} $. In particular, $\mathcal{D}_1$ is a smooth $ k $-scheme of finite type.

Write $ \mathcal{K}^{\diamond}:=\mathcal{K}_1 \rtimes_k \mathcal{K}$, seen as a quotient group of $\mathcal{K}^{+}$ modulo the second direct summand $ \mathcal{K}_1 $.
Recall that $\mathcal{K}=\lim\limits_{\longleftarrow}G_j, $ with $G_j$ the restriction of $G_{k[u]/u^j}$ over $k$. If we denote by $ \mathcal{H}_i $ the kernel of the natural reduction modulo $ u $ map $ G_j\to G $, then we have \begin{align*}
	\mathcal{K}^{\diamond}=\lim\limits_{\longleftarrow}\mathcal{K}^{\diamond}_i, \  \text{with} \  \mathcal{K}^{\diamond}_i:=\mathcal{H}_i\times_k G_i. 
\end{align*} 
Let us consider the right action induced by  $ \rho $ in \eqref{Actionof2K} 
\begin{align*}
	\rho_1: \mathcal{D}_1\times_k \mathcal{K}^{\diamond}\longrightarrow  \mathcal{D}_1, \ \ \ 
	\big(t\mathcal{K}_1, (\alpha, \gamma)\big)\longmapsto\alpha^{-1}\gamma^{-1}t \varphi(\gamma)\mathcal{K}_1.
\end{align*} 
One may check that the action of $ \mathcal{K}^{\diamond} $ on $ \mathcal{D}_1 $ factors through $ \mathcal{K}_{2i+1}^{\diamond} $ (see the hint in the proof of Lemma \ref{Locally Closed lemma}), which is an affine smooth $ k $-scheme. 

Again by a similar argument as in Lemma \ref{Locally Closed lemma}, Lemma \ref{Fpqc surjective} and Proposition \ref{Representability of D} we have the following lemma.
\begin{lemma}\begin{enumerate}[(1)]
		\item The $ \mathcal{K}^{\diamond} $-orbit $ O_{\bar{y}} $ of a point $ \bar{y}\in \mathcal{D}_{1}(k) $ is represented by a smooth $ k $-scheme of finite type.
		\item  If $ \bar{y} $ comes from an element $ y\in \mathcal{D}(k) $ the pull-back of $ O_{\bar{y}} $ under the natural projection $ \mathcal{D}\to \mathcal{D}_1 $ is the $ \mathcal{K}^{+} $-orbit $ O_y $ of $ y $ in $ \mathcal{D}$. 
		\item For any algebraically closed field extension $ l $ of $ k $, we have $ O_y(l)=y\cdot \mathcal{K}^{+}(l) $.
		
		\item For any algebraically closed field extension $ l $ of $ k $, there is a commutative diagram of bijective maps
		\begin{align*}
			\xymatrixcolsep{3pc}\xymatrix{\{\mathcal{K}^{\diamond}\text{-orbits of\ } \mathcal{D}_{1,l} \}\ar@{<->}[r]^{1-1}\ar@{<->}[d]^{1-1}&\{\mathcal{K}^{+}\text{-orbits of }\mathcal{D}_l \}\ar@{<->}[d]^{1-1}\\
				\mathcal{D}_1(l)/\mathcal{K}^{\diamond}(l)\ar@{=}[r]&\mathcal{D}(l)/\mathcal{K}^{+}(l),}
		\end{align*}
		where the top horizontal map commutes with the operation of taking closures.
	\end{enumerate}
\end{lemma} 
\begin{proof}
	Here we only show (4). Indeed there is a 1-1 correspondence 
	\begin{align}\label{Orbits correspondence}
		\xymatrixcolsep{3pc}\xymatrix{\{\mathcal{K}^{\diamond}\text{-orbits of elements in } \mathcal{D}_{1,l}(k) \}\ar@{<->}[r]^{1-1}&\{\mathcal{K}^{+}\text{-orbits of elements in }\mathcal{D}_l(k) \}.}
	\end{align}
	But the left hand side of \eqref{Orbits correspondence} is insensitive to algebraically closed field extensions. Hence every algebraic $ \mathcal{K}^{+} $-orbit of $ \mathcal{D} $ is a $ \mathcal{K}^{+} $-orbit of some $ k $-point of $ \mathcal{D} $. 
\end{proof}

	\section{Classifications of $ p $-divisible groups}

\subsection{General notions of frames and windows}
Frames and windows were introduced by Zink in \cite{ZinkWindowsforP-divisiblegroups} and greatly generalized by Lau in \cite{LauFramesandFiniteGroupSchemes}. Below we introduce these notions mainly following \cite{DieudonneCrystalsandWachModules}. 	
\begin{definition}\label{DefinitionofGeneralFrame}
	A \textbf{Frame }$\underline{S}=(S, \mathrm{Fil}S, R, \varphi, \varphi_1, \varpi)$ consists of the following data:
	\begin{enumerate}[$-$]
		\item a ring $S$ and an ideal $\mathrm{Fil}S$ of $S$ such that  $\mathrm{Fil}S+pS$ is contained in the Jacobson radical of $S$.
		\item the quotient ring $R$= $S/\mathrm{Fil}S$.
		\item a ring endomorphism $\varphi: S\to S$ whose reduction modulo $p$ is the absolute Frobenius map $S/pS\to S/pS$ (in other words, the pair $ (S, \varphi) $ is a simple frame). 
		\item a $\varphi$-linear map $\varphi_1: \mathrm{Fil}S\to S$.
		\item $\varpi$ is an element in $S$ such that $\varphi=\varpi\varphi_1$ on $\mathrm{Fil}S$.
	\end{enumerate}
\end{definition}
We say the frame $\underline{S}$ satisfies the \textbf{surjectivity condition } if the image of $\varphi_1$ generates the unit ideal of $S$.

Let $\underline{S}'=(S', \mathrm{Fil}S', R', \varphi, \varphi_1, \varpi')$ be another frame.  A \textbf{homomorphism of frames} from $\underline{S}\to \underline{S}'$ is a homomorphism of rings $f: S\to S'$ compatible with $\varphi$ and $\varphi_1$. Note that a morphism of frames here is called a \textbf{strict} homomorphism in literatures; see for example, \cite{LauFramesandFiniteGroupSchemes}.

A frame $\underline{S}$ is called a \textbf{lifting frame} if every finite projective $R$-module lifts to a finite projective $S$-module. We shall only concern lifting frames in the sequel. 

\begin{definition}
	\label{DefinitionOfWindow}A\textbf{ window }$\underline{M}=(M, \mathrm{Fil}M, \varphi_M, \varphi_{M, 1})$ over a lifting frame $\underline{S}$ consists of a finite projective $S$-module $M$, an $S$-submodule $\mathrm{Fil}M\subset M$, and $\varphi$-linear maps $\varphi_M: M\to M$ and $\varphi_{M,1}: \mathrm{Fil}M\to M$ subject to the following constraints 
	\begin{enumerate}[(1)]
		\item there exists a decomposition of $S$-modules $M=N\oplus L$ with $\mathrm{Fil}M= N\oplus (\mathrm{Fil}S)L$;
		\item if $s\in \mathrm{Fil}S$ and $m\in M$ then $\varphi_{M,1}(sm)=\varphi_1(s)\varphi_M(m)$;
		\item  for all $m\in \mathrm{Fil}M,\ \ \varphi_M(m)=\varpi\varphi_{M,1}(m)$;
		\item $\varphi_{M, 1}(\mathrm{Fil}M)+ \varphi_M(M) $ generate $M$ as an $S$-module.
	\end{enumerate}
	
	A \textbf{homomorphism of windows} is an $ S $-linear map that preserves the filtration $ \mathrm{Fil}M $ and commutes with $ \varphi_{M} $ and $ \varphi_{M,1} $. 
	
	A decomposition in $(1)$ is called a \textbf{normal decomposition} of $\underline{M}$ (or of $M$).
\end{definition}	

Let $f: \underline{S}\to \underline{S}'$ be a homomorphism of frames. The \textbf{base change of a window} $\underline{M}$ over $\underline{S}$ is defined as $\underline{M}'=(M', \mathrm{Fil}M', \varphi_{M'}, \varphi_{M', 1})$, where \begin{enumerate}[$\bullet$]
	\item $M'=M\otimes_SS'$, $\varphi_{M'}=\varphi_M\otimes_S\varphi: M'\to M'$;
	\item $\mathrm{Fil}M'$ is the submodule of $M'$ generated by $\mathrm{Fil}S'\cdot M'$ and the image of $\mathrm{Fil}M$ in $M'$;
	\item $\varphi_{M',1}$ is determined by $$\varphi_{M',1}(m\otimes x)=\varphi(x)\varphi_{M, 1}(m), \ m\otimes x\in \mathrm{Fil}M\otimes S'\subset \mathrm{Fil}M',$$
	and condition $(2)$ in Definition \ref{DefinitionOfWindow}.
\end{enumerate}
For \textbf{the dual of a window} the reader may refer to Definition 2.1.7 in \cite{DieudonneCrystalsandWachModules} or Section 2.1 in \cite{LauRelations}, or Section 2 in \cite{LauFramesandFiniteGroupSchemes}. It will appear again but we shall not use it in detail.

\begin{remark}\label{RemarkOfWindows}
	We give several remarks on definitions aboves. \begin{enumerate}[$(a)$]
		\item Once we have a normal decomposition $M=N\oplus L$ of $\underline{M}$, we also have $$\mathrm{Fil}M=N+(\mathrm{Fil}S)M.$$ 
		It follows then that any decomposition $M=N\oplus L'$ of $S$-modules is a normal decomposition of $\underline{M}$.
		\item  If $ \underline{S} $ is a lifting frame, the requirement (1) is equivalent to
		
		$(1)'$  \  $(\mathrm{Fil}S)M\subset \mathrm{Fil}M$ and $M/\mathrm{Fil}M$ is a projective $R$-module.
		\item If $ \underline{S} $ satisfies the surjectivity condition, the condition (2) implies (3).
		
		Indeed, if $1=\sum a_i\varphi_1(b_i)\in S$ with all $b_i\in \mathrm{Fil}S$, then
		$$\varphi_M(m)=\sum a_i\varphi_1(b_i)\varphi_M(m)=\sum a_i\varphi_{M,1}(b_im)=\sum a_i\varphi(b_i)\varphi_{M,1}(m)=\varpi \varphi_{M,1}(m).$$ In this case, condition $ (4) $ means that $ \varphi_{M,1}(\Fil M) $ generates $ M $ and $ \varphi_M $ is determined by $ \varphi_{M,1} $. 
		In many cases condition $(2)$ can be replaced by condition $(3)$.
	\end{enumerate}
\end{remark}

\begin{lemma}[{\cite[Lemma 2.6]{LauFramesandFiniteGroupSchemes}}]\label{UsefulLemma1} Let $\underline{S}$ be a frame as in Definition \ref{DefinitionOfWindow}.   Suppose we are given a finite projective $S$-module $M$, an $S$-submodule $\mathrm{Fil}M\subset M$ and a normal decomposition $M=N\oplus L$. Then to give a pair $(\varphi_M, \varphi_{M,1})$ such that  $\underline{M}=(M, \mathrm{Fil}M, \varphi_M, \varphi_{M, 1})$ is an $\underline{S}$ window is equivalent to give a $\varphi$-linear isomorphism $\Psi: N\oplus L\to M$ by the assignment $$\Psi(n+l)=\varphi_{M,1}(n)+\varphi_M(l)$$ for $n\in N$ and $l\in L$.
\end{lemma}
\begin{proof}
	We refer to \cite[Lemma 2.6]{LauFramesandFiniteGroupSchemes} for the whole proof but only give here the inverse of this equivalence. Given a  $\varphi$-linear isomorphism $\Psi: N\oplus L\to M$, the corresponding $\varphi_M$ and $\varphi_{M,1}$ are given as follows \begin{equation}
	\varphi_M(n+l)=\varpi \Psi(n)+ \Psi(l),\ \
	\varphi_{M,1}(n+al)=\Psi(n)+\varphi_1(a)\Psi(l)
	\end{equation}
	for all $n\in N, l\in L$. 
\end{proof}

\subsection{Frobenius lifts and the frame \underline{$\mathfrak{S}$}} \label{SettingofRings}
For the convenience of future discussions, we devote one subsection to the setting of algebras. Till the end of this section we let $ k $ be a perfect field of characteristic $ p>0 $ and denote by $ \varphi: W(k) \to W(k)$ the unique ring automorphism of $ W(k) $ inducing the absolute Frobenius of $ k $.        
\subsubsection{Frobenius lifts}
\begin{lemma} \label{Frobenius lifts}
	Let $R_0$ be a $k$ algebra which (Zariski) locally admits a finite $p$-basis (\cite[Definition 1.1.1]{deJongCrystalineDieudonneRigidgeometry}, or  \cite[D\'efinition 1.1.1]{BWDieudonneCrystallineIII}). The following holds:
	\begin{enumerate}[(1)]
		\item There exists a $p$-adic flat $W(k)$-algebra $R$ lifting $R_0$ (i.e., $R/pR\cong R_0$), which is formally smooth over $W(k)$ with respect to the $p$-adic topology. Such an $R$ is unique up to (nonunique) isomorphisms and we call it a \textbf{lift of} $ R_0 $. 
		\item There is a ring endomorphism $\varphi=\varphi_R: R\to R $ lifting the absolute Frobenius of $R_0$, which is compatible with $ \varphi: W(k)\to W(k) $. 
		\item Let $ R_0, R $ and $ \varphi_{R} $ be as above and $ A_{0} $ an \'etale $ R_0 $ algebra. Then there exists an $ R $-algebra $ A $, unique up to unique isomorphism, such that $ A $ lifts $ A_0 $ and the structure ring homomorphism $ R \to A$ lifts the structure homomorphism $R_0\to A_0$. Moreover, $\varphi_R: R\to R$ extends uniquely to a ring endomorphism of $\varphi_A: A\to A$, lifting the absolute Frobenius of $A_0$.
		\item If $\mathfrak{m}$ is a maximal idea of $R$, then $\varphi$ extends uniquely to a ring endomorphism of the $\mathfrak{m}$-adic completion $ \widehat{R}_{\mathfrak{m}} $ of $R$, lifting the absolute Frobenius of the $\mathfrak{m}$-adic completion $ \widehat{R}_{0, \mathfrak{m}} $ of $R_0$. 
	\end{enumerate}
	
\end{lemma}	

\begin{proof}
	For the proof of (1) and (2), one may refer to \cite[Remarks 1.2.3]{deJongCrystalineDieudonneRigidgeometry}, or \cite[Lemma 2.1]{KimWansuRelative}, where deformation theory developed in \cite{IllusieCotangentI} is essentially used in the proof. Or one may see \cite[1.1]{BWDieudonneCrystallineIII} for an explicit construction of the lifts. The statement (3) is the first part of \cite[Lemma 2.5]{KimWansuRelative} (take $I=(p)$ and the $ R_0 $ in loc. cit. to be our $ R $ here). 
	For (4), note first that  $ \varphi (\mathfrak{m}) \subset  \mathfrak{m} $. This follows from the fact that $ \mathfrak{m} $ contains $ p $, and the fact that the morphism $ \Spec R_0 \to \Spec R_0$ induced by the absolute Frobenius of $ R_0 $ is identity on topological spaces. Now we define $ \varphi_{\widehat{R}_{\mathfrak{m}}}: \widehat{R}_{\mathfrak{m}}  \to \widehat{R}_{\mathfrak{m}} $ by sending an element $$ (r_i)_{i} \in \varprojlim\limits_{i}R/\mathfrak{m}^{i}=\widehat{R}_{\mathfrak{m}}$$ to $ (\varphi(r_i))_{i}\in \widehat{R}_{\mathfrak{m}} $.  The element $(\varphi(r_i))_{i}$ does lie in $ \widehat{R}_{\mathfrak{m}} $ since $ \varphi(\mathfrak{m}) \subset \mathfrak{m}$. One checks that $ \varphi_{\widehat{R}_{\mathfrak{m}}} $ is the desired ring endomorphism.
\end{proof}

\begin{definition}\label{DefinitionofSimpleFrame}
	\begin{enumerate}[(1)]
		\item	A ring endomorphism $\varphi$ of $R$ in Lemma \ref{Frobenius lifts}, (2) is usually called a \textbf{Frobenius lift} of $R$ over $ W(k) $. But note that such a lift $\varphi: R\to R$ is in general NOT unique.
		
		\item	For any $k$-algebra $R_0$ (not necessarily admitting  a finite $p$-basis), we call a pair $(R, \varphi)$ satisfying conditions (1) and (2) in Lemma \ref{Frobenius lifts} a \textbf{simple frame} of $R_0$ over $ W(k) $ (compare Definition \ref{DefinitionofGeneralFrame} and \cite[(A.3)]{KisinCrystalineRepresentations}).
		
		\item	A \textbf{homomorphism of simple frames} $(R, \varphi)\to (R', \varphi)$ is a ring homomorphism $f: R\to R'$ compatible with Frobenius lifts. 
	\end{enumerate}
\end{definition}

\begin{example}\begin{enumerate}[(1)]\label{ExampleofFrames}
		\item 	Let $\tilde{R}$ be a smooth integral $\mathbb{Z}_{(p)}$-algebra of finite type and $R$ the $p$-adic completion of $\tilde{R}$. Then $R$ is a formally smooth flat $\mathbb{Z}_p$-algebra which lifts $\tilde{R}/p\tilde{R}$.
		\item  The crystalline Dieudonn\'e functor $\mathbb{D}^*$ is compatible with change of simple frames. This will be used frequently in the sequel.
		
	\end{enumerate}
	
\end{example}		

\subsubsection{The frame \underline{$\mathfrak{S}$}}	\label{SectionOfKisinFrame}	
Let $R_0$ and $k$ be as in Lemma \ref{Frobenius lifts} and $(R, \varphi)$ a simple frame of $R_0$. We associate a lifting frame
$$\underline{\mathfrak{S}}(R):=(\mathfrak{S}(R), E\cdot \mathfrak{S}(R), R, \varphi, \varphi_1, \varphi(E))$$
to $(R, \varphi)$ by setting:\begin{enumerate}[$\bullet$]
	\item $\mathfrak{S}(R)=R[[u]], \ E=E(u)=u+p\in \mathfrak{S}(R)$;
	\item $\varphi=\varphi_{\mathfrak{S}(R)}: \mathfrak{S}(R)\to \mathfrak{S}(R)$ is an extension of $\varphi_R$ by sending $u$ to $u^p$ and $\varphi_1(Ex)=\varphi(x)$.
\end{enumerate}
Every morphism of simple frames $f: (R, \varphi) \to (R', \varphi)$ induces a morphism of lifting frames $\underline{\mathfrak{S}}(f): \underline{\mathfrak{S}}(R)\to \underline{\mathfrak{S}}(R')$.
\begin{example}\begin{enumerate}[(1)]
		\item For any \'etale $R_0$-algebra $A_0$, we have a morphism of lifting frames $\underline{\mathfrak{S}}(R)\to \underline{\mathfrak{S}}(A)$ (cf. (3) in Lemma \ref{Frobenius lifts}).
		\item  If $\mathfrak{m}$ is a maximal idea of $R$, then we have a morphism of lifting frames $\underline{\mathfrak{S}}(R)\to \underline{\mathfrak{S}}(\widehat{R}_{\mathfrak{m}})$, where $\widehat{R}_{\mathfrak{m}}$ is the completion of $R$ with respect to $\mathfrak{m}$ (cf. (4) in Lemma \ref{Frobenius lifts}).
	\end{enumerate}
	From now on we use the notation $\mathfrak{S}(R)$ for $R[[u]]$ for an arbitrary ring $R$, and we often just write $\mathfrak{S}$ instead of $\mathfrak{S}(R)$ when there is no risk of confusion. 
\end{example}

\subsection{Classification of $ p $-divisible groups over $ R_0 $}\label{SectionClassificationofBT/R_0}
We continue to let $R_0$ and $k$ be as in Lemma \ref{Frobenius lifts} and $(R, \varphi)$ a simple frame of $R_0$. Let $ \underline{R} $ be a lifting frame naturally associated to the simple frame $(R, \varphi)$, defined as follows 
$$ \underline{R}: = (R, \ pR, \ \varphi, \ \varphi_1, \  p), \ \ \text{with}\ \  \varphi_1 =\frac{1}{p}\varphi. $$

We denote by $\widehat{\Omega}_{R}$ the module of $ p $-adically continuous differentials of $ R $, i.e., $$ \widehat{\Omega}_{R}: = \varprojlim\limits_{n} \Omega_{(R/p^nR)/W(k)}^{1}.$$  
It is a projective $ R $-module of finite type due to the finite $ p $-basis assumption on $ R_0 $.    
We denote by $ \mathbf{Win}(R, \nabla) $ the category of tuples $ (M, \mathrm{Fil}M, \varphi_{M}, \varphi_{M,1}, \nabla_{M}) $, where a tuple $ (M, \mathrm{Fil}M, \varphi_{M}, \varphi_{M,1}) $ is a window over the frame $ \underline{R} $, and $ \nabla_{M}: M\to M\otimes_{R}\widehat{\Omega}_{R} $ is an integrable topologically quasi-nilpotent connection over the $ p $-adically continuous derivation $ d_{R}: R\to \widehat{\Omega}_{R} $ of $ R $, with respect to which $ \varphi_{M} $ is horizontal, i.e., $\nabla_{M} \circ \varphi_{M} = (\varphi_{M}\otimes d_{\varphi})\circ \nabla_{M}$.

For any $ p $-divisible group $ H_0 $ over $ R_0 $, we denote by $\mathbb{D}^*(H_0)$ the Dieudonn\'e crystal of $ H_0 $. Denote by $ \mathbb{D}^*(H_0)(R) $ its evaluation at the canonical PD-thickening $ R\twoheadrightarrow R_0 $ and by $ \mathbb{D}^*(H_0)(R_0) $ its evaluation at the trivial PD thickening $ \id_{R_0}: R_0\to R_0 $. Write 
$$ M= \mathbb{D}^*(H_0)(R), \  M_0=\mathbb{D}^*(H_0)(R_0). $$ 
Denote by $F_M: M\to M, \ \  F_{M_0}: M_0\to M_0$
the Frobenius endomorphism of $M$ and of $M_0$ respectively. We write $ \mathrm{Fil}M $ for the preimage of the Hodge filtration of $M_0$ under the canonical projection $M\twoheadrightarrow M_0. $ Denote by $\big(\mathbf{BT}/R_0\big)$ the category of $p$-divisible groups over $R_0$. The following classification result is known.

\begin{theorem}\label{ClassificationofBTs/R_0}
	For any $ p $-divisible group $ H_{0} $ over $ R_0 $, there exists a natural connection $ \nabla_{M}: M\to M\otimes_{R}\widehat{\Omega}_{R} $ such that 
	the tuple 
	\begin{align} \label{WindowofH_0}
		\underline{M}=(M,\  \mathrm{Fil}M, \ F_{M}, \ \ F_{M}/p,\  \ \nabla_{M}) 
	\end{align}
	
	is an object in $ \mathbf{Win}(R, \nabla) $. Moreover, such an assignment gives an equivalence of categories
	$$\big(\mathbf{BT}/R_0\big)\xlongrightarrow{\cong}\mathbf{Win}(R, \nabla).$$
\end{theorem}
\begin{proof}
	This follows from the combination of  \cite[Theorem 4.1.1]{deJongCrystalineDieudonneRigidgeometry} and \cite[Theorem 2.6.4]{DieudonneCrystalsandWachModules}. Note that \cite[Theorem 2.6.4]{DieudonneCrystalsandWachModules} deals with filtered Dieudonn\'e crystals as defined in \cite[Definition 2.4.1]{DieudonneCrystalsandWachModules} but when the scheme $ T $ in loc. cit. is of characteristic $ p $, this notion is equivalent to the version without filtrations as defined in \cite[Definition 2.3.2]{deJongCrystalineDieudonneRigidgeometry}. 
\end{proof}

\subsection{Classification of $ p $-divisible groups over $ R $ (via Dieudonn\'e theory)}
We continue to let $R_0$ and $k$ be as in Lemma \ref{Frobenius lifts} and $(R, \varphi)$ a simple frame of $R_0$. Let $ \underline{R}^{0} $ be another lifting frame naturally associated to $(R, \varphi)$,  defined as follows
$$ \underline{R}^{0}: = (R,\  0, \ \varphi, \ 0, \  p). $$

The frame $ \underline{R}^{0} $ does not satisfy the surjectivity condition. We denote by $ \mathbf{Win}^{0}(R, \nabla) $ the category of tuples $ (M, \mathrm{Fil}M, \varphi_{M}, \varphi_{M,1}, \nabla_{M}) $, where a tuple $ (M, \mathrm{Fil}M, \varphi_{M}, \varphi_{M,1}) $ is a window over the frame $ \underline{R}^{0} $,  and $ \nabla_{M}: M\to M\otimes_{R}\widehat{\Omega}_{R} $ is an integrable topologically quasi-nilpotent connection over the $ p $-adically continuous derivation $ d_{R}: R\to \widehat{\Omega}_{R} $ of $ R $ such that
\begin{enumerate}[(1)]
	\item $\mathrm{Fil}M\subset M$ is a direct summand, lifting  $\ker(\varphi_{M}\otimes \varphi_{R_0})\subset M\otimes_{R}R_0 $.  
	\item  $ \varphi_{M} $ is horizontal with respect to $\nabla_{M}$ (i.e., $\nabla_{M} \circ \varphi_{M} = (\varphi_{M}\otimes d\varphi)\circ \nabla_{M}$). 
\end{enumerate} 
A morphism in $ \mathbf{Win}^{0}(R, \nabla) $ is a morphism of windows over $ \underline{R}^{0} $, which is compatible with connections. 

For each $ p $-divisible group $ H $ over $ R $, we write $ H_0=H\otimes_{R}R_0 $. Note that the Dieudonn\'e module $ M:=\mathbb{D}^*(H_0)(R) $ is equipped with the Hodge filtration $ \Fil M \subset M$ which lifts the (Hodge) filtration $\ker(\varphi_{M}\otimes \varphi_{R_0})\subset M\otimes_{R}R_0 $. If we let $ (F_M, \nabla_{M}) $ be as in Theorem \ref{ClassificationofBTs/R_0}, then we obtain a natural functor
\begin{align}\label{functorP}
	P:	\big(\mathbf{BT}/R\big)\longrightarrow\mathbf{Win}^{0}(R, \nabla).	
\end{align}
by sending a $ p $-divisible group $ H $ over $ R $ to the tuple $ (M:=\mathbb{D}^*(H_0)(R) , \mathrm{Fil}M, \varphi_{M}, 1/p\varphi_{M}, \nabla_{M}) $ described above.

A combination of Theorem \ref{ClassificationofBTs/R_0} and Grothendieck-Messing deformation theory (\cite[Chapter V, Theorem (1.6)]{MessingBT}) gives the following classification result on $ p $-divisible groups over $ R $. 
\begin{theorem}\label{ClassificationbyRmodules}
	The functor $ P $ in \eqref{functorP} is an equivalence of categories. 
\end{theorem}

\textbf{Breuil's ring $ S $}

Breuil's ring $ S $ to be defined below is closely related to the Kim-Kisin windows to be discussed in the next subsection. We will give a classification of $ p $-divisible groups over $ R $ in term of $ S $-modules with extra structures and then discuss the relation between such a classification and that given by Theorem \ref{ClassificationbyRmodules}. 

Write $ \mathfrak{S}=\mathfrak{S}(R) $ and let $ \varpi_1: \mathfrak{S}\to R $ be the $ R $-algebra homomorphism sending the formal variable $ u $ to $ p $ (here we see $ \mathfrak{S} $ as an $ R $-algebra through the embedding $ R\hookrightarrow \mathfrak{S} $). 
Let $ S $ be the $ p $-adic completion of the divided power envelope of $ \mathfrak{S} $ with respect to the kernel of $ \varpi_1 $ (namely the ideal $ E=E(u)=u+p $).  

The ring $ S $ is $ \mathbb{Z}_p $-flat and is a subring of $ S[\frac{1}{p}] $ (see \cite[Section 3.3]{KimWansuRelative} for an explicit description of $ S$). Let $ \pi_1: S\to R $ be the natural projection of $ R $-algebras sending $ u $ to $ p $ and denote its kernel by $ \Fil S $. Then $ \Fil S $ is topologically generated by the divided powers $ \{\frac{E(u)^n}{n!}\}_{n\geq 1} $ of $ E(u) $.  The Frobenius lift $ \varphi: \mathfrak{S} \to \mathfrak{S}$ extends uniquely to a Frobenius lift $ \varphi=\varphi_S: S\to S $ of $ S $, and we have $\varphi(\Fil S)\subset pS$.   The tuple $$\underline{S}:= (S, \Fil S, \varphi, \frac{1}{p}\varphi, p) $$ is a frame satisfying the surjectivity condition (in fact, $ \frac{\varphi(E)}{p} $ is unit in $ S $). In addition to the natural projection $ \pi_1: S\to R $ which sends $ u $ to $ p $, there is also another natural projection $ \pi_2: S\to R $ of $ R $-algebras which sends $ u $ to $ 0 $. The kernel of $ \pi_2 $, denoted by $ \Fil' S $,  is topologically generated by $ u $ and all $ \{\frac{u^n}{n!} \}_{n\geq 1}$. Moreover, $ \Fil 'S  $ is $ \varphi $-stable and hence $ p_2 $ induces a homomorphism of simple frames $ (S, \varphi) \to (R, \varphi)$.  Note that $ \pi_1 $ and $ \pi_2 $ are two sections of the natural embedding $ R\hookrightarrow S $ and they induce two morphisms of PD thickenings
\begin{align}\label{PDsection}
	(S\twoheadrightarrow R_0)\xrightarrow{\pi_1} (R\twoheadrightarrow R_0),\ \ \ \  (S\twoheadrightarrow R_0)\xrightarrow{\pi_2} (R\twoheadrightarrow R_0), 
\end{align}
which are two sections of the PD morphism $  (R\twoheadrightarrow R_0)\hookrightarrow (S\twoheadrightarrow R_0)$. Here  we use ($S\twoheadrightarrow R_0)\xrightarrow{\pi_1} (R\twoheadrightarrow R_0$) to denote a commutative diagram
\begin{align*}
	\xymatrixcolsep{4pc}	\xymatrixcolsep{4pc}\xymatrix{S\ar[d]^{\pi_1} \ar@{->>}[r]&R_0\ar@{=}[d]\\R \ar@{->>}[r]&R_0.}
\end{align*}
The embedding $ R\to S $ induces homomorphisms of frames
\begin{align}\label{homsofframes}
	\begin{array}{rcl}
		\underline{R}=(R, pR, \varphi, \frac{1}{p}\varphi, p)& \longrightarrow & \underline{S}=(S, \Fil S, \varphi, \frac{1}{p}\varphi, p);\\
		\underline{R}^{0}=(R, 0, \varphi, 0, p)& \longrightarrow & \underline{S}.  
	\end{array}
\end{align}

The second map in \eqref{PDsection} induce a homomorphism of frames 
\begin{align}\label{Projectionofframes}
	\underline{S}\longrightarrow \underline{R}^{0}.
\end{align}

Denote by $ \mathbf{Win}(S, \nabla^{0}) $ the category of tuples $ (\mathsf{M}, \mathrm{Fil}\mathsf{M}, \varphi_{\mathsf{M}}, \varphi_{\mathsf{M},1}, \nabla_{\mathsf{M}\otimes_{S, \pi_2}R}) $, where a tuple $ (\mathsf{M}, \mathrm{Fil}\mathsf{M}, \varphi_{\mathsf{M}}, \varphi_{\mathsf{M},1}) $ is a window over $ \underline{S} $,  and $ \nabla_{\mathsf{M}\otimes_{S, \pi_2}R}: \mathsf{M}\otimes_{S, \pi_2}R\to \mathsf{M}\otimes_{S, \pi_2}R\otimes_{R}\widehat{\Omega}_{R} $ is an integrable topologically quasi-nilpotent connection over the $ p $-adically continuous derivation $ d_{R}: R\to \widehat{\Omega}_{R} $ of $ R $, with respect to which $ \varphi_{\mathsf{M}\otimes_{S, \pi_2}R}:=\varphi_{\mathsf{M}}\otimes \varphi_R $ is horizontal.

For each object $\underline{M}= (M, \mathrm{Fil}M, \varphi_{M}, \varphi_{M,1}, \nabla_{M}) $ in $ \mathbf{Win}^{0}(R, \nabla) $, we can associate to $ \underline{M} $ an object $ (\mathsf{M}, \mathrm{Fil}\mathsf{M}, \varphi_{\mathsf{M}}, \varphi_{\mathsf{M},1}, \nabla_{\mathsf{M}\otimes_{S, \pi_2}R}) $ in $ \mathbf{Win}(S, \nabla^{0}) $  by setting  $(\mathsf{M}, \mathrm{Fil}\mathsf{M}, \varphi_{\mathsf{M}}, \varphi_{\mathsf{M},1})$ to be the base change of the window $(M, \mathrm{Fil}M, \varphi_{M}, \varphi_{M,1})$ along the homomorphism of frames $\underline{R}^{0}\to \underline{S}$ in \eqref{homsofframes}, and $\nabla_{M}=\nabla_{\mathsf{M}\otimes_{S, \pi_2}R}$. Note that by definition $ \Fil \mathsf{M} = \Fil M\otimes_RS+\Fil S \mathsf{M}$ and hence $\Fil \mathsf{M}$ is also the preimage of $ \Fil M $ under the projection $ \mathsf{M}=M\otimes_RS\xrightarrow{\id_{M}\otimes \pi_1} M $.
Such associations are functorial and hence induces a functor 
\begin{align*}
	X:  \mathbf{Win}^{0}(R, \nabla) \longrightarrow \mathbf{Win}(S, \nabla^{0}).
\end{align*} 

Write $ Q=X \circ P:	\big(\mathbf{BT}/R\big)\longrightarrow\mathbf{Win}(S, \nabla^{0}).	 $
\begin{theorem}[{\cite[Theorem 3.17]{KimWansuRelative}}]\label{ClassificationbyBreuilmodules}
	The functor $ X:  \big(\mathbf{BT}/R\big)\longrightarrow\mathbf{Win}(S, \nabla^{0})$ is an equivalence of categories.
\end{theorem}

\begin{remark} Note that we are in the simplest situation of \cite{KimWansuRelative} since in our situation here no ramification happens and hence we can identify $ R $ with $ \mathfrak{S}/(E(u)) $. To be precise, our  $ R $, $\mathfrak{S}/(E(u))$ play the role of $ R_0  $, $R$ in \cite{KimWansuRelative} respectively and $ p\subset \mathfrak{S}/(E(u)) $ plays the role of $ \varpi $ in loc. cit. Though in our case we can identify $ R $ with $ \mathfrak{S}/(E(u)) $, sometimes it is still necessary to distinguish them. In the general situation as in loc. cit.,  $ \pi_1 $ and $ \pi_2 $ have different targets.  
	
\end{remark}

It follows immediately from Theorem \ref{ClassificationbyRmodules} and Theorem \ref{ClassificationbyBreuilmodules} that the functor $X$ is also an equivalence of categories. To give an explicit inverse functor of $ X $, we need the following fact: for the $ S $-module $ \mathsf{M} $ in an object  $ \underline{\mathsf{M}}$ of $\mathbf{Win}(S, \nabla^{0})  $, we have a canonical isomorphism of $ R $-modules
\begin{align}\label{MysteriousIsomorphism}
	\mathsf{M}\otimes_{S, \pi_1}R \cong \mathsf{M}\otimes_{S, \pi_2}R.
\end{align} 
This is due to the crystalline interpretation of $ \mathsf{M} $ via Theorem \ref{ClassificationbyBreuilmodules}. Indeed, we may assume $ \mathsf{M}=\mathbb{D}^{*}(H_0)(S) $ for some $ p $-divisible group $ H_0 $ over $ R_0 $. Then we have a canonical isomorphism
$$\mathsf{M}\otimes_{S, \pi_1}R \cong \mathbb{D}^{*}(H_0)(R)\cong \mathsf{M}\otimes_{S, \pi_2}R.$$ 

Define a functor 
\begin{align}
	Y: \mathbf{Win}(S, \nabla^{0}) \longrightarrow \mathbf{Win}^{0}(R, \nabla)
\end{align}
$$(\mathsf{M}, \mathrm{Fil}\mathsf{M}, \varphi_{\mathsf{M}}, \varphi_{\mathsf{M},1}, \nabla_{\mathsf{M}\otimes_{S, \pi_2}R}) \longmapsto (M:=\mathsf{M}\otimes_{S, \pi_2}R, \Fil M, \varphi_{\mathsf{M}}\otimes \varphi_{R}, \frac{1}{p} (\varphi_{\mathsf{M}}\otimes \varphi_{R}), \nabla_{M} ),$$
where $ \nabla_{\mathsf{M}\otimes_{S, \pi_2}R}$ is sent to itself (since we set $M:=\mathsf{M}\otimes_{S, \pi_2}R$), and $ \Fil M $ is the image of $\Fil  \mathsf{M} $ under the projection $$ \mathsf{M} \rightarrow \mathsf{M}\otimes_{S, \pi_1}R \cong \mathsf{M}\otimes_{S, \pi_2}R=M. $$ 

It is trivial to check that $ X $ and $ Y $ are inverse to each other.

\subsection{Kim-Kisin  windows and Kim-Kisin modules}
\label{Kisin windows and Kism-Kisin modules}

Breuil-Kisin modules, as defined in \cite[(2.2.1)]{KisinCrystalineRepresentations}, play a vital rule in the development of integral $ p $-adic Hodge theory. They are typically used to classify $ p $-divisible groups over a totally ramified extension $ R $ of $ W(k) $ (of arbitrary finite ramification index). Such a classification was conjectured  in a precise form by Breuil (see \cite{BreuilConjecture}) and was first proved by Kisin in \cite{KisinCrystalineRepresentations}. A similar classification result was generalized by Brinon and Trihan (\cite{BrinonTrihanReprese}) to the case where $ R $ is a $ p $-adic discrete valuation ring with imperfect residue field admitting a finite $ p $-basis. The case where $ R $ is a regular local ring with perfect residue field is studied in a series of paper by Cais, Lau,Vasiu and Zink (\cite{VasiuZinkBreuil'scalssificationof}, \cite{LauRelations}, \cite{DieudonneCrystalsandWachModules}), using the theory of displays and windows. W. Kim in \cite{KimWansuRelative} generalized the classification results aforementioned to a relative setting (e.g., $ R $ is a $ p $-adic ring with $ R/(p) $ locally admitting a finite $ p $-basis), where he essentially used the method developed in \cite{CarusoLiuQuasi-Semistable}. We call the relative version of Breuil-Kisin modules Kim-Kisin modules (cf. \cite[Definition 6.1]{KimWansuRelative}). In the following we are dealing with the simplest relative situation in the sense that no nontrivial ramification occurs.

We retain the notations in Section \ref{SectionOfKisinFrame} and let $ \underline{\mathfrak{S}}:=\underline{\mathfrak{S}}(R) $ be the lifting frame associated to a simple frame $ (R, \varphi) $. 

\begin{definition}  A \textbf{Kim-Kisin $\mathfrak{S}$ window}, or simply a \textbf{Kim-Kisin window} $$\underline{\mathcal{M}}=(\mathcal{M}, \mathrm{Fil}\mathcal{M}, \varphi_{\mathcal{M}}, \varphi_{\mathcal{M},1}, \nabla_M)$$ is a 5-tuple where \begin{enumerate}[(1)]
		\item $(\mathcal{M}, \mathrm{Fil}\mathcal{M}, \varphi_{\mathcal{M}}, \varphi_{\mathcal{M},1})$ is an $\underline{\mathfrak{S}}$-window.
		\item $M$ is defined to be $\mathcal{M}\otimes_{\mathfrak{S}}\mathfrak{S}/(u)$ and $\nabla_{M}: M\to M\otimes_{R}\widehat{\Omega}_{R}$ is an integrable topologically quasi-nilpotent connection over the $ p $-adically continuous derivation $ d_{R}: R\to \widehat{\Omega}_{R} $, with respect to which the $\varphi$-linear endomorphism  $F_M:=\varphi_\mathcal{M}\otimes_{\mathfrak{S}}\varphi_R$ of $M$ is horizontal.
	\end{enumerate}
\end{definition}		
\begin{remark}The following remarks will be needed in the sequel.
	\label{RemarkFrobeniusDecomposition}\begin{enumerate}[$(a)$]
		\item For a Kim-Kisin $\mathfrak{S}$ window $\underline{\mathcal{M}}$, we have $\varphi_{\mathcal{M},1}=\frac{1}{\varphi(E)}\varphi_\mathcal{M}$ on $\mathrm{Fil}\mathcal{M}$ since $\varphi(E)\in \mathfrak{S}$ is not a zero divisor, and hence  the map $\varphi_{\mathcal{M}, 1}$ is determined by $\varphi_\mathcal{M}$; see (c) in Remark \ref{RemarkOfWindows}.
		
		\item For any normal decomposition $\mathcal{M}=\mathcal{N}\oplus \mathcal{L}$ of $\underline{\mathcal{M}}$, by definition the $\varphi$-linear map $\varphi_\mathcal{M}$ has the decomposition \begin{equation}\label{Decomposition of Frobenius}
		\varphi_\mathcal{M}=\big(\mathcal{N}\oplus \mathcal{L}\xrightarrow{E\cdot\id_\mathcal{N}\oplus \id_\mathcal{L}}\mathcal{N}\oplus \mathcal{L}\xrightarrow{\Gamma}\mathcal{M}\big),
		\end{equation}
		where $$\Gamma=\frac{1}{E}\varphi_\mathcal{M}|_\mathcal{M}\oplus \varphi_\mathcal{M}|_\mathcal{L}$$ is a $\varphi$-linear isomorphism (Lemma \ref{UsefulLemma1}).  The linearization of \eqref{Decomposition of Frobenius} is the following
		\begin{equation}\label{LinearizedDecomposition of Frobenius}
		\varphi_\mathcal{M}^{\mathrm{lin}}=\big(\mathcal{N}^{(\varphi)}\oplus \mathcal{L}^{(\varphi)}\xrightarrow{\varphi(E)\cdot\id\oplus \id}\mathcal{N}^{(\varphi)}\oplus \mathcal{L}^{(\varphi)}\xrightarrow{\Gamma^{\mathrm{lin}}}\mathcal{M}\big),
		\end{equation}
		where $\varphi_\mathcal{M}^{\mathrm{lin}}$, $\Gamma^{\mathrm{lin}}$ are the linearizations of $\varphi_\mathcal{M}$ and $\Gamma$ respectively, and where $\Gamma^{\mathrm{lin}}$ is an $\mathfrak{S}$-isomorphism.
	\end{enumerate}
\end{remark}

\begin{definition}
	A\textbf{ Kim-Kisin $\mathfrak{S}$ module} $\underline{\mathfrak{M}}$ is a triple $(\mathfrak{M}, \varphi_{\mathfrak{M}}, \nabla_{M})$ where  \begin{enumerate}[(1)]
		\item $\mathfrak{M}$ is a finite projective $\mathfrak{S}$-module;
		\item $\varphi_{\mathfrak{M}}: \mathfrak{M}\to \mathfrak{M}$ is a $\varphi$-linear map such that the cokernel of the linearization $(1\otimes\varphi_{\mathfrak{M}}): \varphi^*\mathfrak{M}\to \mathfrak{M}$ is annihilated by $E\in \mathfrak{S}$.
		\item $ M $ is defined as 
		$$M:=\mathfrak{M}\otimes_{\mathfrak{S}, \varphi}\mathfrak{S}/(u)=\varphi^*\mathfrak{M}\otimes_{\mathfrak{S}, \varpi_2}R$$ and $\nabla_{M}: M\to M\otimes_{R}\widehat{\Omega}_{R}$ is an integrable topologically quasi-nilpotent connection which commutes with the $\varphi $-linear endomorphism of $M$, $$F_{M}:=(\varphi_{\mathfrak{M}}\otimes1)\otimes_{\mathfrak{S}, \varpi_2}\varphi_{R}.$$
	\end{enumerate}
\end{definition}
Denote by $\mathbf{Win}(\mathfrak{S}, \nabla^0)$ the category of Kim-Kisin-windows and by $\mathbf{Mod}(\mathfrak{S}, \nabla^0)$ the category of Kim-Kisin $\mathfrak{S}$-modules, both with respect to $(R, \varphi_R)$. Here we emphasize the simple frame $(R, \varphi_R)$, especially the Frobenius lift $\varphi_R$, because later we are going to study the effect of different Frobeni on Kim-Kisin windows. 	


\begin{proposition}\label{EquivalenceofModules&Windows}
	There is an equivalence of categories
	\begin{align*}\mathbf{Win}(\mathfrak{S}, \nabla^0)&\longrightarrow \mathbf{Mod}(\mathfrak{S}, \nabla^0), \\
		(\mathcal{M}, \mathrm{Fil}\mathcal{M}, \varphi_{\mathcal{M}}, \varphi_{\mathcal{M},1}, \nabla_{M})&\longmapsto (\mathrm{Fil}\mathcal{M}, E\cdot\varphi_{\mathcal{M},1}, \nabla_{M}),
	\end{align*}  where the second  $\nabla_{M}$ is justified via the isomorphism $\varphi_{\mathcal{M},1}\otimes 1: \varphi^*\mathrm{Fil}\mathcal{M}\to \mathcal{M}$. This equivalence preserves exactness and duality.
\end{proposition}
\begin{proof}
	This follows from an exactly same argument as in \cite[lemma　 2.1.15　]{DieudonneCrystalsandWachModules}, though connections are not concerned in loc. cit.  We sketch here how to obtain the inverse of the equivalence of categories. 
	
	For any Kim-Kisin module $(\mathfrak{M}, \varphi_{\mathfrak{M}}, \nabla_{M})$. There is a unique $ \mathfrak{S} $-linear map $ \psi: \mathfrak{M}\to \varphi^{*}\mathfrak{M} $ such that 
	\begin{align*}
		\varphi_{\mathfrak{M}}\circ \psi= E(u)\cdot \id_{\mathfrak{M}}, \ \ \  \psi\circ \varphi_{\mathfrak{M}}= E(u)\cdot \id_{\varphi*\mathfrak{M}}. 
	\end{align*}
	We obtain a Kim-Kisin window $\underline{\mathcal{M}}=(\mathcal{M}, \mathrm{Fil}\mathcal{M}, \varphi_{\mathcal{M}}, \varphi_{\mathcal{M},1}, \nabla_{M})$ by setting 
	\begin{align*}
		\mathcal{M}:=\varphi^{*}\mathfrak{M}, \ \ \  \mathrm{Fil}\mathcal{M}:&=\psi(\mathfrak{M}), \ \ \ 	\varphi_{\mathcal{M}}:=\varphi_{\mathfrak{M}}\otimes\id_{\mathfrak{S}}; \\
		\varphi_{\mathcal{M}, 1}(\psi(x))&=x\otimes 1, \ \ \ \text{ for any\ } x\in  \mathfrak{M}.
	\end{align*}
\end{proof}

\begin{theorem}[{\cite[Corollary 6.7]{ KimWansuRelative}}]\label{Kim's theorem}
	There is an equivalence of categories
	$$\big(\mathbf{BT}/R\big)\to\mathbf{Mod}(\mathfrak{S}, \nabla^0).$$
	It is compatible with duality. Moreover, if $(R, \varphi) \to (R', \varphi)$ is morphism of simple frames, then the equivalence commutes with base change of frames along $\underline{\mathfrak{S}}(R)\to \underline{\mathfrak{S}}(R')$.
\end{theorem}

Then it follows immediately that:
\begin{corollary} \label{TheoremCategoryEquivalence}
	There is a category equivalence
	$$Z: \big(\mathbf{BT}/R\big)\to\mathbf{Win}(\mathfrak{S}, \nabla^0).$$
	It is compatible with duality. Moreover, if $(R, \varphi) \to (R', \varphi)$ is morphism of simple frames, then the equivalence commutes with base change of frames from $\underline{\mathfrak{S}}(R)\to \underline{\mathfrak{S}}(R')$. \label{CorollaryBaseChangeofKisinWindow}
	
\end{corollary}
\begin{remark}\label{RemarkKim'sClassification}	
	\begin{enumerate}[(1)]
		\item In Corollary (\ref{TheoremCategoryEquivalence}) there is no obvious map from either of the left and the right hand sides to the other.
		\item The classification of $ p $-divisible groups over $ R $ in terms of Kim-Kisin $ \mathfrak{S} $-modules gives rise to a classification of ($ p $-power order) finite flat group schemes over $ R $ in terms of torsion Kim-Kisin $ \mathfrak{S} $-modules (by no means trivial, see \cite[9.8]{KimWansuRelative}). We will not explicitly use it. But note that taking the reduction modulo $ p $ of the Kim-Kisin module of a $ p $-divisible group $ H $ over $ R $ one obtains exactly the Kim-Kisin module associated with the $ p $-kernel of $ H $. We will take such operations frequently in the future without further explanation.  
	\end{enumerate}
\end{remark}

The following lemma describes the basic relations between Kim-Kisin modules and its crystalline Dieudonn\'e modules.
\begin{lemma}\label{propertiesofKimKisinwindows}
	Let $(\mathcal{M}, \mathrm{Fil}\mathcal{M}, \varphi_{\mathcal{M}}, \varphi_{\mathcal{M},1}, \nabla_{M})$ be the Kim-Kisin window associated to a $ p $-divisible group $ H $ over $ R $. We have the following canonical identifications. \begin{enumerate}[(1)]
		\item	$ (\mathcal{M}, \varphi_\mathcal{M}) \otimes_{\mathfrak{S}, \varpi_2}R $ is canonically identified with the (contravariant) Dieudonn\'e module  $\mathbb{D}^{*}(H)$ of $ H $, together with its Frobenius map.
		\item  $ (\mathcal{M}, \mathrm{Fil}{\mathcal{M}}) \otimes_{\mathfrak{S}, \varpi_1} R$ is canonically identified with $ \mathbb{D}^{*}(H) $ together with its Hodge filtration (cf. \cite[Corollaire 3.3.5]{BBMTheoriedeDieudonnecristalline}).
	\end{enumerate}

	The combination of Proposition \ref{EquivalenceofModules&Windows}, Theorem \ref{Kim's theorem}, and Theorem \ref{ClassificationofBTs/R_0} gives below a commutative diagram.
	\begin{align}\label{DiagramofDeformations}
		\xymatrixcolsep{4pc}\xymatrix{\big(\textbf{BT}/R\big)\ar[r]^{\cong}\ar[d]^{\mod p}&\textbf{Mod}\big(\mathfrak{S}, \nabla^{0}\big)\ar[r]^{\cong}&\textbf{Win}\big(\mathfrak{S}, \nabla^{0}\big)\ar[d]^{\mod u}\\
			\big(\textbf{BT}/R_0\big)\ar[rr]^{\cong}&&\textbf{Win}\big(R, \nabla \big)}
	\end{align}
	where we simply use ``$\cong$'' to denote an equivalence of categories, and where the right vertical functor ``mod $u $'' denotes the base change of windows along the homomorphism of lifting frames
	$$ \underline{\mathfrak{S}}\xrightarrow{\mod u}\underline{R}. $$ 
\end{lemma}
\begin{proof}
	This follows from the discussions between Remark 3.13 and Lemma 3.14 in \cite[Corollary 6.7]{ KimWansuRelative}, \cite[Corollary 6.7]{ KimWansuRelative} and \cite[Corollary 6.7]{ KimWansuRelative} together with the erratum \cite{Erratum2RelativeClassification}.  
\end{proof}

\subsection{A comparision of classification results}
In the remaining of this section, we compare the classification result given in Theorem \ref{ClassificationbyRmodules} and that given in Corollary \ref{TheoremCategoryEquivalence} by establishing an explicit functor from  $\mathbf{Win}(\mathfrak{S}, \nabla^0)$ to  $\mathbf{Win}^{0}(R, \nabla)$.

Let $ \varpi_2 : \mathfrak{S}\to R$ be the homomorphism of $ R $-algebras sending $ u $ to $ 0 $. Clearly we have a homomorphism of simple frames $ (\mathfrak{S}, \varphi) \to (S, \varphi)$ induced by the embedding $ \mathfrak{S}\to S $, and $  (\mathfrak{S}, \varphi) \to (S, \varphi)$ induced by $ \varpi_2 $. The compositions of the embedding $ \mathfrak{S}\to S $ with $ \pi_1 $ and $ \pi_2 $ are equal to $ \varpi_1 $ and $ \varpi_2 $ respectively.

We now define a functor from $ \mathbf{Win}(\mathfrak{S}, \nabla^0)$ to $\mathbf{Win}(S, \nabla^{0}) $ as follows. 
$$ W: \mathbf{Win}(\mathfrak{S}, \nabla^0)\longrightarrow  \mathbf{Win}(S, \nabla^{0}) $$
$$(\mathcal{M}, \mathrm{Fil}\mathcal{M}, \varphi_{\mathcal{M}}, \varphi_{\mathcal{M},1}, \nabla_{\mathcal{M}\otimes_{\mathfrak{S}, \varpi_2}R}) \longmapsto (\mathsf{M}:=\mathcal{M}\otimes_{\mathfrak{S}}S, \Fil \mathsf{M},\varphi_{\mathsf{M}}:= \varphi_{\mathcal{M}}\otimes \varphi_{S}, \frac{1}{p} \varphi_{\mathsf{M}}, \nabla_{\mathcal{M}\otimes_{\mathfrak{S}, \varpi_2}R}),$$
where the connection $\nabla_{\mathcal{M}\otimes_{\mathfrak{S}, \varpi_2}R}$ on the right hand side makes sense since by definition of $ \mathsf{M} $ we have $ \mathsf{M}\otimes_{S, \pi_2}R= \mathcal{M}\otimes_{\mathfrak{S}, \varpi_2}R $.
\begin{theorem}
	The functor $ W $ is well defined and is an equivalence of categories. Moreover, we have the following commutative diagram
	\begin{align}\label{CommutativediagramofFrakSwindowandSwindow}
		\xymatrixcolsep{4pc}\xymatrix{\big(\textbf{BT}/R\big)\ar[dr]^{Q}\ar[r]^{Z}&\mathbf{Win}(\mathfrak{S}, \nabla^0)\ar[d]^{W}\\
			& \mathbf{Win}(S, \nabla^{0})}
	\end{align}	
	
\end{theorem}
\begin{proof}
	To see that $ W $ is well defined, we need to check that $ \varphi_{M}:=\frac{1}{p}\varphi_{M} $ makes sense. To see this, let $ \mathcal{M}=\mathcal{N}\oplus \mathcal{L}  $ be a normal decomposition of $ \mathcal{M} $, then we have $ \Fil \mathsf{M}=\mathcal{N}\otimes_{\mathfrak{S}}S +\Fil S \cdot \mathcal{M}$. 
	Since $ \varphi_{\mathcal{M}}(\mathcal{N}) \subset \varphi(E)\mathcal{M}$, we have 
	\begin{align*}
		\varphi_{\mathsf{M}}(\mathcal{N}\otimes_{\mathfrak{S}}S)\subset \varphi(E)\mathsf{M}=\frac{\varphi(E)}{p}\cdot p\mathsf{M}=p\mathsf{M}.
	\end{align*}
	Here we use the fact that $ \frac{\varphi(E)}{p} $ is a unit in $ S $. Now it follows from the inclusion $ \varphi(\Fil S)\subset pS $ that the functor $ W $ is well defined.
	
	The fact that $ W $ is an equivalence of categories is due to \cite[Proposition 6.6]{KimWansuRelative}  and \cite[Corollary 6.7]{KimWansuRelative}. Indeed, using the notations in loc. cit., the equivalence between  $\mathrm{Mod}_{\mathfrak{S}}(\varphi,\nabla)$ and $\mathrm{MF}_{S}(\varphi,\nabla)$ induces trivially an equivalence  $W^{0}: \mathrm{Mod}_{\mathfrak{S}}(\varphi,\nabla^{0})\to \mathrm{MF}_{S}(\varphi,\nabla^{0})$), which translated into our language via the equivalence in Proposition \ref{EquivalenceofModules&Windows} is the $ W $ defined  above.
	
	The commutativity of \eqref{CommutativediagramofFrakSwindowandSwindow} is due to the way in which the equivalence in Corollary \ref{TheoremCategoryEquivalence} is deduced (see the remark below). 
\end{proof}
\begin{remark}
	In fact, Theorem \ref{Kim's theorem} is deduced in \cite{KimWansuRelative} by first proving Theorem \ref{ClassificationbyBreuilmodules} and then showing the equivalence 
	$$ \mathbf{Mod}(\mathfrak{S}, \nabla^0)\longrightarrow  \mathbf{Win}(\mathfrak{S}, \nabla^0)\xlongrightarrow{W}  \mathbf{Win}(S, \nabla^{0}). $$ 
\end{remark}

Now a direct computation gives the following corollary.

\begin{corollary}\label{Corrolaryconcerningdeformation}
	The composition of $ W $ and $ Y $ gives an equivalence of categories
	\begin{align}
		\Pr:	\mathbf{Win}(\mathfrak{S}, \nabla^0)\longrightarrow  \mathbf{Win}^{0}(R, \nabla)
	\end{align}
	$$(\mathcal{M}, \mathrm{Fil}\mathcal{M}, \varphi_{\mathcal{M}}, \varphi_{\mathcal{M},1}, \nabla_{\mathcal{M}\otimes_{S, \varpi_2}R}) \longmapsto (M:=\mathcal{M}\otimes_{\mathfrak{S}, \varpi_2}R, \Fil M, \varphi_{M}:=\varphi_{\mathcal{M}}\otimes \varphi_{R}, \frac{1}{p} \varphi_{M}, \nabla_{M} ),$$
	where $ \Fil M $ is the image of $ \Fil \mathcal{M} $ under the projection
	$$ \mathcal{M} \to \mathcal{M}\otimes_{\mathfrak{S}, \varpi_1}R\cong \mathcal{M}\otimes_{\mathfrak{S}, \varpi_2}R=M.$$
	Here the isomorphism  $\mathcal{M}\otimes_{\mathfrak{S}, \varpi_1}R\cong \mathcal{M}\otimes_{\mathfrak{S}, \varpi_2}R$ is due to \eqref{MysteriousIsomorphism} and is thus also canonical. 
\end{corollary}

\begin{remark}
	It is not clear to us how to give a direct construction of the inverse functor of $ \Pr $, even when $ R $ is $ W(k) $ itself.
\end{remark}

	\section{Adapted deformations of $p$-divisible groups}\label{SectionofConstructionofKisinWindows}
We continue to follow notations in Sections \ref{SettingofRings},  \ref{SectionClassificationofBT/R_0}, and \ref{Kisin windows and Kism-Kisin modules}. Let $(R, \varphi)$ be a simple frame of a $k$-algebra $R_0$, which locally admits a finite $ p $-basis.
We fix a $ p $-divisible group $ H_0$ over $ R_0 $.

\subsection{Adapted deformations}
\label{ConstructionofKisinWindows}
We have seen in Theorem \ref{ClassificationofBTs/R_0} that $ H_0 $ corresponds to an $ \underline{R}$-window, together with a natural connection, namely 
$$\underline{M}= (M:=\mathbb{D}^{*}(H_0)(R),\  \mathrm{Fil}M, \ \varphi_{M}, \ \varphi_{M}/p,\  \ \nabla_{M}). $$
We let $ \lambda: \mathbb{G}_{m, R}\to \GL(M) $ be a cocharacter over $ R $ of weights $ 0 $ and $ 1 $, such that the reduction modulo $ p $ of $ \lambda $, denoted by $ \lambda_{0}: \mathbb{G}_{m, R_0}\to \GL(M_0) $, induces the Hodge filtration of $ M_0:=\mathbb{D}^{*}(H_0)(R_0) $ as the weight $ 1 $ submodule of $ M_0 $.  Let $ M=N\oplus L $ be the splitting induced by $ \lambda $, where $ N $ and $ L $ are submodules of $ M $ with weights $ 1 $ and $ 0 $ respectively. 

Write $ \mathfrak{S}=\mathfrak{S}(R) $. In the following we shall construct a Kim-Kisin window $ \underline{\mathcal{M}} $ over $ \underline{\mathfrak{S}} $. The construction is analogous to the construction of Kisin's ``$ G_{W} $-adapted deformations'' in \cite[Prop. (1.1.13)]{KisinModpPoints}. We first give each of the datum in $ \underline{\mathcal{M}} $ as below.

\begin{enumerate}[$\bullet$]
	\item $\mathcal{M}=M\otimes_R\mathfrak{S}$, $\mathcal{N}=N\otimes_R\mathfrak{S}$, $\mathcal{L}=L\otimes_R\mathfrak{S}$;
	\item	$\mathrm{Fil}\mathcal{M}=\mathcal{N}\oplus E\cdot \mathcal{L}$;
	
	\item $\varphi_{\mathcal{M}}: \mathcal{M}\to \mathcal{M}$ is the following composition 	\begin{align*}
		\mathcal{M}\to \mathcal{M}^{(\varphi)}&=\mathcal{N}^{(\varphi)}\oplus \mathcal{L}^{(\varphi)}\xrightarrow{\varphi(E)\cdot\id_{\mathcal{N}^{(\varphi)}}\oplus \id_{\mathcal{L}^{(\varphi)}}}   \mathcal{N}^{(\varphi)}\oplus \mathcal{L}^{(\varphi)}=M^{(\varphi)}\otimes_R\mathfrak{S}\\
		&M^{(\varphi)}\otimes_R\mathfrak{S}\xrightarrow{\Gamma^{\mathrm{lin}}\otimes \id_{\mathfrak{S}}}M\otimes_R\mathfrak{S}=\mathcal{M}.
	\end{align*} Here if we denote by $F^{\mathrm{lin}}: M^{(\varphi)}\to M$ the linearization of the Frobenius $F: M\to M$ then $ \Gamma^{\mathrm{lin}} $ is defined as \begin{equation}\label{DefinitionOfGammaLin}
	\Gamma^{\mathrm{lin}}=\frac{1}{p}F^{\mathrm{lin}}|_{N^{(\varphi)}}\oplus F^{\mathrm{lin}}|_{L^{(\varphi)}}: N^{(\varphi)}\oplus L^{(\varphi)}\to M.
	\end{equation}
	\item $\varphi_{\mathcal{M},1}=\frac{1}{\varphi(E)}\cdot\varphi_{\mathcal{M}}$;
	\item $\nabla_{M}: M\to M\otimes_R\hat{\Omega}_{R} $ is already given as part of the data in $ \underline{M} $.  
\end{enumerate}
\begin{lemma}
	The tuple $\underline{\mathcal{M}}=(\mathcal{M}, \mathrm{Fil}\mathcal{M}, \varphi_\mathcal{M}, \varphi_{\mathcal{M},1},  \nabla_{M})$ is indeed a Kim-Kisin window.
\end{lemma}
\begin{proof}
	The nontrivial part is to verify condition (4) in Definition \ref{DefinitionOfWindow}, i.e., we need to check that the the linearization 
	\[ \varphi_{\mathcal{M},1}^{\mathrm{lin}}: \mathrm{Fil}\mathcal{M}^{(\varphi)}\to \mathcal{M}  \]
	is surjective (equivalently an isomorphism). A simple calculation shows that one only needs to show that the map $\Gamma^{\mathrm{lin}}$ in \eqref{DefinitionOfGammaLin} is an isomorphism. To see this we apply Lemma \ref{LemmaYaonie} below by letting $P=N^{(\varphi)}$ and $Q=L^{(\varphi)}$.
\end{proof}

The following (probably well-known) lemma will be used in the proof of Lemma  \ref{LemmaYaonie}. 
\begin{lemma}\label{Ausefullemma}
	Let $(A, \varphi)$ be a simple frame of a $k$-algebra $ A_0 $. Then for any homomorphism $ f: A_0 \to B_0 $ of $k$-algebras, there exists a (canonical) homomorphism of simple frames $(A, \varphi)\to (W(B_0), \varphi)  $ which lifts $ f $.  
\end{lemma}
\begin{proof}
	By Proposition 2, a) of \cite[Chap. IX, Sec. 1]{BourbakiComAlg}, the ring homomorphism (namely the ghost map) $$\Phi_{A}: W(A)\to \prod_{i=0}^{\infty}A$$ in loc. cit. is injective and by c) in the same proposition the image of the ring homomorphism $$\alpha:=\prod_{i=0}^{\infty}\varphi^{i}: A\longrightarrow \prod_{i=0}^{\infty}A, \ a\longmapsto (a, \varphi(a), \varphi(\varphi(a)), \cdots)$$ 
	lies in the image of $\Phi_A$, and hence the composition $\beta:= \Phi_A^{-1}|_{\Phi_A(W(A))}\circ \alpha$ gives a section (as a ring homomorphism) of the canonical projection $W(A)\to A$. Moreover, by formula (25) in loc. cit. $\beta$ is compatible with Frobenius lifts. Now the composition map $$(A, \varphi)\xrightarrow{\beta} (W(A), \varphi)\to (W(A_0), \varphi)\to( W(B_0), \varphi)$$ gives the desired homomorphism of simple frames.
\end{proof}
When $ R_0 $ is a perfect field, the following lemma is a classical result. In the general situation as below, it is also known but for lack of good references, we give a proof.
\begin{lemma}\label{LemmaYaonie}
	Let $ M_0=N_0\oplus L_0 $ be the normal decomposition induced by $ \lambda_{R_0} $ and let $ M^{(\varphi)}=P\oplus Q $ be any decomposition whose reduction modulo $ p $ is identified with the decomposition $ M_0^{(\varphi)}=N_0^{(\varphi)}\oplus L_0^{(\varphi)} $. Then the $ R $-linear homomorphism 
	\begin{equation}
	\Gamma^{\mathrm{lin}}:=	\frac{1}{p}F^{\mathrm{lin}}|_{P}\oplus F^{\mathrm{lin}}|_{Q}: M^{(\varphi)}\to M
	\end{equation}
	is an isomorphism.
\end{lemma}

\begin{proof}

	To see $\Gamma^{\mathrm{lin}}$ is an isomorphism we need only to show that  $\Gamma^{\mathrm{lin}}\otimes\id_{R_{\mathfrak{m}}}$ is an isomorphism for each maximal ideal $\mathfrak{m}$ of $R$ ($\mathfrak{m}$ necessarily contains $ p $ since $ R $ is $ p $-adically complete).  
	Denote by $\widehat{R}_{\mathfrak{m}}, \widehat{R}_{0,\mathfrak{m}} $ the  $\mathfrak{m}$-adic completion of $R_{\mathfrak{m}}$, $R_{0,\mathfrak{m}}$ respectively, and $k(\mathfrak{m}) $ their residue field. If we let $\overline{k(\mathfrak{m}})$ be an algebraic closure of $k(\mathfrak{m})$, then by Lemma \ref{Ausefullemma}, there exists a homomorphism of simple frames
	$(\widehat{R}_\mathfrak{m}, \varphi)\rightarrow  (W(\overline{k(\mathfrak{m}})), \varphi)$, lifting the composition $\widehat{R}_{0,\mathfrak{m}} \to k(\mathfrak{m}) \to \overline{k(\mathfrak{m})}$. Since the Dieudonn\'e functor $\mathbb{D}^*$ commutes with homomorphisms of simple frames, using the fact that the assertion holds for the classical case where $R_0$ is a perfect field (see for example \cite[Lemma 2.2.6]{ChaoZhangEOStratification} for a proof), one sees that $\Gamma^{\mathrm{lin}}\otimes\id_{W(\overline{k(\mathfrak{m})})}$ is an isomorphism. In particular, $\Gamma^{\mathrm{lin}}\otimes\id_{\overline{k(\mathfrak{m})}}$ is an isomorphism, and hence by faithfully flat descent, so is $\Gamma^{\mathrm{lin}}\otimes\id_{k(\mathfrak{m})}$. Now by Nakayama's lemma, $\Gamma^{\mathrm{lin}}$ is an isomorphism. 
\end{proof}

It is easy to see that by taking the reduction modulo $ p $ of frames 
$ \underline{\mathfrak{S}} \to \underline{R}$ one recovers the associated window 
$ \underline{M} $ in \eqref{WindowofH_0} of $ H_0 $. Hence by the commutativity of diagram \eqref{DiagramofDeformations} the $ p $-divisible group corresponding to Kim-Kisin window $ \underline{\mathcal{M}} $ constructed above is a deformation over 
$ R $ of $ H_0 $. We call $ \underline{\mathcal{M}} $ the adapted deformation of $ H_0 $ with respect to the pair $ (\varphi_R, \lambda) $. Here we emphasize the Frobenius lift $ \varphi_{R} $ because later we will study the effect of Frobenius lifts on 
$ \underline{\mathcal{M}}$. To emphasize the pair $ (\varphi_{R}, \lambda) $ we may as well write 
\begin{align}\label{ConventionofFrobenius}
	\Phi(\varphi_{R}, \lambda):=\Phi_{\mathcal{M}}(\varphi_{R}, \lambda) :=	\varphi_{\mathcal{M}}^{\mathrm{lin}}.
\end{align}

\begin{remark}  We give a few remarks concerning the construction above, which will become useful in the sequel. 
	\begin{enumerate}[(a)]  \label{RemarkofKisinWindows}
		\item There is another way to see the construction of $\varphi$-linear maps $\varphi_\mathcal{M}$ and $\varphi_{\mathcal{M},1}$. Indeed the $ \varphi$-linear map $$\Gamma:=\frac{1}{p}F|_{N}\oplus F|_{L}: M\to M$$ is a $\varphi$-linear isomorphism since its linearization is \eqref{DefinitionOfGammaLin}.  Hence the base change $\Psi:=\Gamma\otimes_R\varphi_{\mathfrak{S}}$ is also a $\varphi$-linear isomorphism. It is direct to check that the pair $(\varphi_\mathcal{M}, \varphi_{\mathcal{M},1})$ constructed above is exactly the one induced by $\Psi$ via Lemma \ref{UsefulLemma1}.
		If we denote by $ \lambda^{(\varphi)}: \mathbb{G}_{m, R} \to \GL(M^{(\varphi)})$ the pull back along $ \varphi_R$ of $ \lambda $. Then it is clear that 
		\begin{align}
			\Phi(\varphi, \lambda) = \lambda^{(\varphi)}(E)\circ \Psi.
		\end{align} 				
		\item It is direct to check that we have 
		\begin{align}
			\Phi(\varphi, \lambda)=
			\big(\varphi^*\mathcal{N}\oplus \varphi^*\mathcal{L}&\xrightarrow{f} \varphi^*\mathcal{N} \oplus \varphi^*\mathcal{L}\xrightarrow{F^{\mathrm{lin}}\otimes\id_{\mathfrak{S}}}M\otimes_R \mathfrak{S}\big),\ \  \text{with}\\
			f=&\frac{\varphi(E)}{p}\id|_{\varphi^*\mathcal{N}}\oplus \id|_{\varphi^*\mathcal{L}}
		\end{align}
		Strictly speaking, such a decomposition only makes sense after base change to $ \mathfrak{S}[\frac{\varphi(E)}{p}] \subset \mathfrak{S}[\frac{1}{p}]$. 
		
	\end{enumerate}
	
\end{remark}

\subsection{Functoriality of adapted deformations}\label{FunctorialityofKim-Kisinwindows}
Let $ R'_0 $ be another $ k $-algebra, locally admitting a finite $ p $-basis and $ R_0\to R'_0 $ a homomorphism of $ k $-algebras. Let $ (R', \varphi_{R'} ) $ a simple frame of $ R'_0 $ and $\alpha: (R, \varphi_{R})\to (R', \varphi_{R'}) $ a homomorphism of simple frames over $ W(k) $, which induces the structure homomorphism $ R_0\to R'_0 $. By Example \ref{ExampleofFrames} (2) we have 
\begin{align*}
	\mathbb{D}^{*}(H_0\otimes R'_0)(R')\cong M\otimes_{R}R',
\end{align*}
and hence one checks easily that the adapted deformation of $ H_0\otimes R_0' $ w.r.t. the pair $ (\varphi_{R'}, \lambda_{R'}) $ is the base change of $ \underline{\mathcal{M}} $ along the homomorphism $ \underline{\mathfrak{S}}\to \underline{\mathfrak{S}(R')} $ induced from  $ \alpha $. In particular, we have 
\begin{align*}
	\Phi(\varphi_{R'}, \lambda_{R'})=\Phi(\varphi_{R}, \lambda)\otimes R'.
\end{align*}

\subsection{Comparisons of different Frobenii and different decompositions} \label{SectionGluing}
Notations are as in Section \ref{ConstructionofKisinWindows}. 
Denote by $ (\mathcal{M}_0,\varphi_{\mathcal{M}_0}^{\mathrm{lin}}) $ the reduction modulo $ p $ of the pair $ (\mathcal{M}, \varphi_{\mathcal{M}}^{\mathrm{lin}})$. Note that by our convention in \eqref{ConventionofFrobenius} we write $  \varphi_{\mathcal{M}}^{\mathrm{lin}}=\Phi(\varphi_{R}, \lambda) $. Similarly we write $  \varphi_{\mathcal{M}_0}^{\mathrm{lin}}=\Phi_0(\varphi_{R}, \lambda) $. We are mainly interested in the Frobenius $ \Phi_0(\varphi_{R}, \lambda) $ and for our purpose we shall study how the pair $(\varphi_{R}, \lambda)$ should affect $ \Phi_0(\varphi_{R}, \lambda) $.

\begin{lemma}\label{Gluinglemma2}
	Let $ \varphi_1, \varphi_2 $ be two Frobenius lifts of $ R $ and $ \lambda_1, \lambda_2 $ two cocharacters of $ \GL(M) $ over $ R $ of weights $ 0,1 $ such that the weight submodules of $ M $ are free $ R $-modules and that the reductions modulo $ p $ of $ \lambda_1, \lambda_2 $ induce the Hodge filtration of $ M_0 $ as weight $ 1 $-submodule. Suppose that 
	$\lambda_1=\lambda_2 \mod p$. Then there exists a unique automorphism $ \eta: \mathcal{M}_0^{(\varphi)}\to  \mathcal{M}_0^{(\varphi)}$ whose reduction modulo $ u $ is identity, such that 
	\begin{align*}
		\Phi_0(\varphi_1, \lambda_1)=\Phi_0(\varphi_2, \lambda_2)\circ \eta.
	\end{align*}
\end{lemma}

\begin{proof} The uniqueness of $\eta$ follows from the fact that $\Phi_0(\varphi_1, \lambda_1),\  \Phi_0(\varphi_2, \lambda_2)$ become isomorphisms after inverting $ u\in \mathfrak{S}(R_0) $. We show below the existence of $ \eta $. Write
	$$c=\varphi_1(E)=\varphi_2(E)=u^p+p.$$ 
	
	Since $ \Phi(\varphi_1, \lambda_1)$ and $\Phi(\varphi_2, \lambda_2) $ become isomorphisms after inverting $c$ there exists a unique isomorphism $$\tilde{\eta}: (\varphi_1^*\mathcal{M})[\frac{1}{c}]\to  (\varphi_2^*\mathcal{M})[\frac{1}{c}] \ \ \text{such that}\ \   \Phi(\varphi_1, \lambda_1)[\frac{1}{c}]= \Phi(\varphi_2, \lambda_2)[\frac{1}{c}]\circ \tilde{\eta}.$$ 
	
	To show the existence of $ \eta $ it is sufficient to show that the reduction modulo $ p $ of $\tilde{\eta}$, which is a priori only an automorphism of $\mathcal{M}^{(\varphi)}_0[\frac{1}{u}]$, comes from an automorphism of $\mathcal{M}_0^{(\varphi)}$. We show this in two steps: in the first step we let $ \lambda_1=\lambda_2 $ while in the second step we let $ \varphi_1=\varphi_2 $.
	
	\textbf{Step 1}: Assume $ \lambda=\lambda_1=\lambda_2 $. Let $ M=N\oplus L $ be the decomposition induced by $ \lambda $. 
	
	Recall that from Remark \ref{RemarkofKisinWindows}, (b), for each $ i=1,2 $ we have 
	\begin{align*}
		\Phi(\varphi_i, \lambda)= (F_i^{\mathrm{lin}}\otimes\id_{\mathfrak{S}})\circ f_i,
	\end{align*}
	where $F_i^{\mathrm{lin}}: \varphi^*M\to M $ is the linearization of the Frobenius $F_i: M\to M$ of $ M $ with respect to the Frobenius lift $ \varphi_i$, and $ f_i $ is defined as  
	$$f_i=\frac{c}{p}\id|_{\varphi_i^*\mathcal{N}}\oplus \id|_{\varphi_i^*\mathcal{L}}.$$
	By Dieudonn\'e theory there exists a unique isomorphism $\chi: \varphi_1^*M\to \varphi_2^*M$ of $ R $-modules such that $$F_1^{\mathrm{lin}}=F_2^{\mathrm{lin}}\circ \chi \ \ \text{with} \ \  \chi\otimes\id_{R_0} =\id_{M_0^{(\varphi)}}.$$ 	
	Hence we have the following commutative diagram, where we omit $ [\frac{1}{c}] $ everywhere.
	\begin{align}\label{DiagramofFabrizio'sMatrixComputation}
		\xymatrix{\varphi_1^*\mathcal{M}\ar[d] ^{\tilde{\eta}}\ar@{=}[r]&\varphi_1^*\mathcal{N}\oplus \varphi_1^* \mathcal{L}\ar[r]^{f_1}&\varphi_1^*\mathcal{N}\oplus \varphi_1^* \mathcal{L}\ar@{=}[r]&\varphi_1^*M\otimes_R\mathfrak{S}\ar[d]^{\chi\otimes\id_{\mathfrak{S}}}\\
			\varphi_2^*\mathcal{M}\ar@{=}[r]&\varphi_2^*\mathcal{N}\oplus \varphi_2^* \mathcal{L}\ar[r]^{f_2 }&\varphi_2^*\mathcal{N}\oplus \varphi_2^* \mathcal{L}\ar@{=}[r]&\varphi_2^*M\otimes_R\mathfrak{S}}
	\end{align}

	Let $m_1, \cdots, m_r\in N$ be an $ R $- basis of $N$, and $m_{r+1}, \cdots, m_{r+s}\in L$  an $ R $-basis of $L$. Let 
	\begin{align*}
		&\mathfrak{B}^{(\varphi_i)}:=(\varphi_i^{*}m_1, \ldots, \varphi_{i}^*m_{r+s}), \\
		&\mathfrak{B}^{(\varphi_i)}\otimes 1:=(\varphi_i^{*}m_1\otimes1, \ldots, \varphi_{i}^*m_{r+s}\otimes 1), 
	\end{align*}
	be the induced basis for $ \varphi_i^*M $ and for $ \varphi_i^*\mathcal{M} $ respectively. Then $f_1, f_2$ have the same matrix representation
	\[X=\left(	
	\begin{array}{cc}
	\frac{c}{p}\mathbf{I}_r&0\\
	0 & \mathbf{I}_s
	\end{array}
	\right)\]
	in the sense that $ f_i $ sends the basis $ \mathfrak{B}^{(\varphi_i)}\otimes 1 $ to $(\mathfrak{B}^{(\varphi_i)}\otimes 1)X$. Let 
	\begin{align}\label{Matrix Y}
		Y=\left(	
		\begin{array}{cc}
			\mathbf{I}_r+pA&pB\\
			pC & \mathbf{I}_s+pD
		\end{array}
		\right)
	\end{align}
	be the matrix representation of $ \chi $ under the basis $ \mathfrak{B}^{(\varphi_1)} $ and $ \mathfrak{B}^{(\varphi_2)}$, where $A, B, C, D$ are matrices with entries in $R$.
	A direct computation shows that $\tilde{\eta}$ has matrix representation
	\[Z=\left(	
	\begin{array}{cc}
	\mathbf{I}_r+pA&\frac{p^2}{c}B\\
	cC & \mathbf{I}_s+pD
	\end{array}\right)\]
	If we denote by $ Z_0 $ the reduction modulo $ p $ of $ Z $, then it is clear that $ Z_0 $ lies in the kernel of the reduction map $$\GL_{r+s}(\mathfrak{S}(R_0))\xrightarrow{\mod u} \GL_{r+s}(R_0).$$

	\textbf{Step 2}. Assume $ \varphi=\varphi_1=\varphi_2 $. Let $$ M=N_1\oplus L_1, \ \  M=N_2\oplus L_2 $$ be the normal decompositions of $ M $ induced by $ \lambda_1 $ and $ \lambda_2 $ respectively. Write $$ \mathcal{N}_i=N_i\otimes\mathfrak{S},\ \  \mathcal{L}_i=L_i\otimes\mathfrak{S}. $$ 
	Then we have the following commutative diagram 
	\begin{align}\label{DiagramofFabrizio2}
		\xymatrix{\varphi^*\mathcal{M}\ar[d] ^{\tilde{\eta}}\ar@{=}[r]&\varphi^*\mathcal{N}_1\oplus \varphi^* \mathcal{L}_1\ar[r]^{f_1}&\varphi^*\mathcal{N}_1\oplus \varphi^* \mathcal{L}_1\ar@{=}[r]&\varphi^*M\otimes_R\mathfrak{S}\ar@{=}[d]\\
			\varphi^*\mathcal{M}\ar@{=}[r]&\varphi^*\mathcal{N}_2\oplus \varphi^* \mathcal{L}_2\ar[r]^{f_2 }&\varphi^*\mathcal{N}_2\oplus \varphi^* \mathcal{L}_2\ar@{=}[r]&\varphi^*M\otimes_R\mathfrak{S}}
	\end{align}	
	Here $ f_i $ by definition is given by 
	$$f_i=\frac{c}{p}\id|_{\varphi^*\mathcal{N}_i}\oplus \id|_{\varphi^*\mathcal{L}_i}.$$
	Let $m_{i,1}, \cdots, m_{i,r}\in N_i$ be an $ R $- basis of $N_i$, and $m_{i,r+1}, \cdots, m_{i, r+s}\in L_i$  an $ R $-basis of $L_i$ such that 
	\begin{align*}
		m_{1,j}=m_{2,j} \mod p, \ \ \ j=1,\ldots, r+s.
	\end{align*} Here we use the assumption $ \lambda_1=\lambda_2 \mod p $. We let  
	\begin{align*}
		&\mathfrak{B}^{(\varphi)}_1:=(\varphi^{*}m_{1,1}, \ldots, \varphi^*m_{1, r+s}), \\
		&\mathfrak{B}_1^{(\varphi)}\otimes 1:=(\varphi^{*}m_{1,1}\otimes1, \ldots, \varphi^*m_{1,r+s}\otimes 1),
	\end{align*}
	be the induced basis for $ \varphi^*M $ and for $ \varphi^*\mathcal{M} $ which are on the top of diagram \eqref{DiagramofFabrizio2}. In a same way we can define $\mathfrak{B}_2^{(\varphi)}, \ \mathfrak{B}_2^{(\varphi)}\otimes 1 $, and take them as the basis for $ \varphi^*M $ and for $ \varphi^*\mathcal{M} $ which are on the bottom of diagram \eqref{DiagramofFabrizio2}, respectively. Now the right vertical identity homomorphism has the same matrix representation as $ Y $ in \eqref{Matrix Y}. Then we finish the proof by following a similar matrix computation as in Step 1.

\end{proof}

\section{ Shimura varieties of Hodge type}
\subsection{Shimura data and Shimura varieties over $\mathbb{C}$}

\begin{definition}\label{DefinitionofShimuradatum}
	Let $\textbf{G}$ be a connected reductive group over $\mathbb{Q}$ and $\textbf{X}$ a conjugacy class of homomorphism of algebraic groups over $\mathbb{R}$
	$$h: \mathbb{S}=\Res_{\mathbb{C}/\mathbb{R}}\mathbb{G}_m\to \textbf{G}_{\mathbb{R}}.$$
	
	\begin{enumerate}[(1)]
		\item 	A pair $(\textbf{G, X})$ as above is called a \textbf{Shimura datum} if it satisfies the following conditions:
		\begin{enumerate}[(a)]
			\item Let $\mathfrak{g}$ denote the Lie algebra of $\textbf{G}_{\mathbb{R}}$. Then the composite
			$$\mathbb{S}\to \textbf{G}_{\mathbb{R}}\to \textbf{G}_{\mathbb{R}}^{\ad}\to \GL(\mathfrak{g})$$ defines a Hodge structure of type $(-1, 1), (0,0), (1, -1)$. 
			
			\item $h(i)$ is a Cartan involution of $\textbf{G}_{\mathbb{R}}^{\ad}$. This means that we require the real form of $\textbf{G}^{\ad}$ defined by the involution $g\mapsto h(i){\bar{g}}h(i)^{-1}$ to be compact.
			\item $\textbf{G}$ has no factor defined over $\mathbb{Q}$ whose real points form a compact group.
			
		\end{enumerate}
		
		\item 	A morphism $i: \textbf{(G}_1, \textbf{X}_1)\to (\textbf{G}_2, \textbf{X}_2)$ of Shimura data is a map of groups $\textbf{G}_1\to \textbf{G}_2$, which induces a map $\textbf{X}_1\to \textbf{X}_2$.		
	\end{enumerate}
	
\end{definition}	

The requirement (a) in Definition \ref{DefinitionofShimuradatum} means that under the action of $\mathbb{C}^{\times}$ on $\mathfrak{g}_{\mathbb{C}}=\mathfrak{g}\otimes_{\mathbb{R}}\mathbb{C}$ by conjugation through $ h_{0} $, we have a decomposition
$$\mathfrak{g}_{\mathbb{C}}= V^{-1,1}\oplus V^{0,0}\oplus V^{1,-1},$$
where $z\in \mathbb{C}^{\times}$ on $V^{p, q}$ via multiplication by  $z^{-p}\bar{z}^{-q}$. This implies that for any $h_0\in \textbf{X}$ the stabilizer $\mathsf{K}_{\infty}\subset \textbf{G}(\mathbb{R})$ of $h_0$ is compact modulo its center, and $\textbf{G}(\mathbb{R})/ \mathsf{K}_{\infty}\cong \textbf{X}$ has a complex structure.

Let $\mathbb{A}_{f}$ denote the finite adeles over $\mathbb{Q}$, and $\mathbb{A}_f^{p}\subset\mathbb{A}_f$ the subgroup of adeles with trivial component at a prime $p$. Let
$\mathsf{K}=\mathsf{K}_p\mathsf{K}^p\subset \textbf{G}(\mathbb{A}_f)$
be a compact open subgroup, where $\mathsf{K}_p\subset \textbf{G}(\mathbb{Q}_p)$, and $\mathsf{K}^p\subset \textbf{G}(\mathbb{A}_f^p)$ are compact open subgroups. A theorem of Baily-Borel asserts that when $\mathsf{K}^{p}$ is small enough 
\begin{align}\label{DoublequotientofcomplexShimuravariety}	\Sh_\mathsf{K}(\textbf{G}, \textbf{X})_{\mathbb{C}}:=\textbf{G}(\mathbb{Q})\backslash \textbf{X}\times \textbf{G}(\mathbb{A}_f)/\mathsf{K}
\end{align}
has a natural structure of an algebraic variety over $\mathbb{C}$.
We will always assume in the following that $\mathsf{K}^p$ is small enough. 
\subsection{Shimura varieties over number fields}	
Let $ (\textbf{G}, \textbf{X}) $ be a Shimura datum and $E=E(\textbf{G}, \textbf{X})$ the reflex field of $ (\textbf{G}, \textbf{X}) $ (see \cite[Definition 12.2]{Milne_IntroShimura}). 
\begin{remark}
	Any subfield $F$ of $\bar{\mathbb{Q}}$ over which the group $\textbf{G}$ splits contains $E$. This is because if $T$ is a split maximal torus of $\textbf{G}$ over $F$, then the set $W\backslash\Hom(\mathbb{G}_m, T)$ does not change if we pass from $k$ to $\bar{\mathbb{Q}}$. It follows that $E$ is a finite field extension of $\mathbb{Q}$; i.e., $E$ is a number field. 
\end{remark}	
Results of Shimura, Deligne, Milne and others imply that, up to isomorphism, $\Sh_\mathsf{K}(\textbf{G},\textbf{ X})_{\mathbb{C}}$ has a unique quasi-projective smooth model $\Sh_{\mathsf{K}}(\textbf{G}, \textbf{X})$  over the number field $E$. 

A morphism $j: (\textbf{G}_1, \textbf{X}_1)\to (\textbf{G}_2, \textbf{X}_2)$ of Shimura data induces a morphism of schemes
\begin{equation}\label{Morphism of Shimura Data}
\Sh_{\mathsf{K}_1}(\textbf{G}_1, \textbf{X}_1)\to \Sh_{\mathsf{K}_2}(\textbf{G}_2, \textbf{X}_2),
\end{equation}
provided that the compact open subgroups are chosen so that $\mathsf{K}_1$ maps into $\mathsf{K}_2$.
This map is defined over the composite of the reflex fields $E(\textbf{G}_1, \textbf{X}_1)$ and $E(\textbf{G}_2, \textbf{X}_2)$. If the morphism $j$ is an embedding, i.e., the homomorphism $\textbf{G}_1\to \textbf{G}_2$ associated to $j$ is a closed embedding, then for any $\mathsf{K}_1$, the subgroup $\mathsf{K}_2$ can always be chosen so that  (\ref{Morphism of Shimura Data}) is a closed immersion of schemes.

Sometimes people also consider the pro-varieties
$$\Sh_{\mathsf{K}_p}(\textbf{G}, \textbf{X}):= \varprojlim\Sh_\mathsf{K}(\textbf{G}, \textbf{X}),$$
where $\mathsf{K}$ runs through compact open subgroups as above with a fixed factor $\mathsf{K}_p$ at $p$, and 
$$\Sh(\textbf{G}, \textbf{X}):= \varprojlim\Sh_{\mathsf{K}}(\textbf{G}, \textbf{X}),$$ where $\mathsf{K}$ runs through all compact open subgroup of $\textbf{G}(\mathbb{A}_f)$.	But we shall not consider such towers by fixing soon a Shimura datum $ (\textbf{G}, \textbf{X}) $ and a compact open subgroup $ \mathsf{K}\subset \mathbb{A}_{f} $. 

\subsection{Shimura datum of Hodge type} 

Suppose that  $V$ is a finite-dimensional $\mathbb{Q}$-vector space with a perfect alternating pairing $\psi$ and write $\textbf{GSp}=\textbf{GSp}(V, \psi)$ the corresponding group of symplectic similitudes. Then we get a Shimura datum $(\mathbf{GSp}, \textbf{S}^{\pm})$ with $\textbf{S}^{\pm}$ the Siegel double space, which is defined to be the set of maps $\mathbb{S}\to \textbf{G}_{\mathbb{R}}$ such that
\begin{enumerate}[(1)]
	\item The $\mathbb{C}^{\times}$ action on $V_{\mathbb{R}}$ gives rise to a Hodge structure of type $(-1, 0)$ and $(0, -1)$.
	\item $(x, y)\mapsto \psi(x, h(i)y)$ is (positive or negative) definite on $V_{\mathbb{R}}$.
\end{enumerate}
\begin{definition}
	A Shimura datum $(\textbf{G}, \textbf{X})$ is said to be \textbf{of Hodge type} if there is an embedding of Shimura data $i: (\textbf{G} , \textbf{X})\hookrightarrow (\textbf{GSp}, \textbf{S}^{\pm})$ for some $(\textbf{GSp}, \textbf{S}^{\pm})$ as above. 
\end{definition}
The reflex field of  $(\mathbf{GSp}, \textbf{S}^{\pm})$ is just $\mathbb{Q}$. We fix a $\mathbb{Z}$-lattice $V_{\mathbb{Z}}$ of $V$  and write $V_{\hat{\mathbb{Z}}}= V_{\mathbb{Z}}\otimes \hat{\mathbb{Z}}$ with $\hat{\mathbb{Z}}$ the profinite completion of $\mathbb{Z}$.  If $V_{\hat{\mathbb{Z}}}$ is stable by $\mathsf{K}\subset \textbf{G}(\mathbb{A}_f)$, then for $\mathsf{K}^p$ small enough $\Sh_\mathsf{K}(\mathbf{GSp}, \textbf{S}^{\pm})$ can be seen as the moduli space of abelian varieties.

\subsection{Good reduction of Shimura varieties of Hodge type}\label{GoodreductionofShimuravariety}

We will mainly study the special fibre of the integral model of some Shimura variety of Hodge type. Let us fix some notations and some basic assumptions, following \cite[4.1]{WedhornGeneralizedHass} and \cite[(1.3.3)]{KisinModpPoints}.  

We fix a Shimura datum $(\textbf{G} , \textbf{X})$ of Hodge type and let $i: (\textbf{G} , \textbf{X})\hookrightarrow (\textbf{GSp}, \textbf{S}^{\pm})$ be an embedding of Shimura data (we do not fix this embedding for the moment). Fix a prime $ p>3 $ and let $ \mathsf{K}=\mathsf{K}_p\mathsf{K}^{p}\subset \textbf{G}(\mathbb{A}_f) $ be an open compact subgroup such that $ \mathsf{K}_p \subset  \mathbf{G}(\mathbb{Q}_p)$ is a hyperspecial subgroup and such that $ \mathsf{K}^{p} \subset \mathbf{G}(\mathbb{A}_{f}^{p})$ is sufficiently small. The condition that $ \mathsf{K}_p $ is hyperspecial means that there is a reductive group $ \mathcal{G} $ over $\mathbb{Z}_{(p)} $, which we fix from now on, such that $ \mathsf{K}_p=\mathcal{G}(\mathbb{Z}_p) $. Then by \cite[(2.3.1), (2.3.2)]{KisinIntegralModels} there is an $ \mathbb{Z}_{(p)} $-lattice $ \Lambda $ of $ V $ such that the embedding $ i $ is induced by some embedding $ i: \mathcal{G}\to \GL(\Lambda) $. By Zarhin's trick we may assume after changing $ (V, \psi ) $ and $ \lambda $ that $ \psi $ induces a perfect $ \mathbb{Z}_{(p)} $-pairing on $ \Lambda $ (cf. \cite[(1.3.3)]{KisinIntegralModels}), still denoted by $ \psi $. Now we have an embedding $ \iota: \mathcal{G} \to \textbf{GSp}(\Lambda, \psi)$ of reductive group schemes over $ \mathbb{Z}_{(p)} $. The condition that $ \mathsf{K}^{p}$ is sufficiently small guarantees that the double quotient in \eqref{DoublequotientofcomplexShimuravariety} has a structure of smooth quasi-projective complex variety and hence admits a canonical model  $\Sh_{\mathsf{K}}(\textbf{G}, \textbf{X})$ over the reflex field $ E $ of $(\textbf{G}, \textbf{X}) $. 

From now on we fix this embedding $ \iota$ (and hence $ i $). By \cite[Proposition (1.3.2)]{KisinIntegralModels}, $ \mathcal{G} $ is then the schematic stabilizer of a set of tensors $ s\subset \Lambda^{\otimes} $. From the discussion in Section \ref{Tensorsandcontragredient}, we may identify $ \Lambda^{\otimes}$ with $(\Lambda^{*})^{\otimes} $, and hence we have
\begin{align}
	\mathcal{G}=\{g\in \GL(\Lambda)\big| g^{\vee}(s)=s \text{ pointwise}\}.
\end{align}

Set $ \tilde{\mathsf{K}}_{p}=\textbf{GSp}(\mathbb{Z}_p) $. By \cite[(2.1.2)]{KisinIntegralModels} there exists an open compact subgroup $ \tilde{\mathsf{K}}^{p}\subset \textbf{GSp}(\mathbb{A}_f) $ containing $\mathsf{K}^{p}  $ such that $ \iota $ induces an embedding of Shimura varieties over $ E $
\[\varepsilon^{0}:  \Sh_{\mathsf{K}}(\textbf{G}, \textbf{X})\hookrightarrow \Sh_{\tilde{\mathsf{K}}}(\textbf{GSp}, \textbf{S}^{\pm}).\] Moreover, if $\tilde{\mathsf{K}}^{p} $ is sufficiently small, $ \Sh_{\tilde{\mathsf{K}}}(\textbf{GSp}, \textbf{S}^{\pm}) $ has a quasi-projective smooth integral canonical model over $ \mathbb{Z}_{(p)} $, denoted by
\[ \tilde{\mathcal{S}}: = \tilde{\mathcal{S}}_{\tilde{\mathsf{K}}}(\textbf{GSp}, \textbf{S}^{\pm}), \] which has an explicit moduli interpretation 
(\cite[(1.3.4)]{KisinIntegralModels}) and admits a universal abelian scheme $ \mathcal{A}\to \tilde{\mathcal{S}} $. We always assume that $\tilde{\mathsf{K}}^{p} $ is sufficiently small for what follows.

Fix a place $ v $ of the reflex field $ E $ of $( \textbf{G}, \textbf{X}) $ above $ p $. Denote by $ \mathcal{O}_E $ the ring of integers of $ E $ and $ \mathcal{O}_{E, (v)} $ its localization at $ v $. Write $ \kappa:=k(v) $ the residue field of $ \mathcal{O}_{E, (v)} $ and $ \mathcal{O}_{E,v} $ the completion of $ \mathcal{O}_E $ at $ v $. The existence of the hyperspecial subgroup $ \mathsf{K}_{p} $ implies that $ E $ is unramified at $ p $ (\cite[Corollary 4.7]{ShimuraVarietiesandMotives}), and hence we have 
\begin{align*}
	\mathcal{O}_{E, v}=W(\kappa), 
\end{align*}
we shall use these two notations interchangeably. Kisin showed the existence of the integral canonical model $$ \mathcal{S}: =\mathcal{S}_{\mathsf{K}}(\textbf{G}, \textbf{X}) $$ of $\Sh_{\mathsf{K}}(\textbf{G}, \textbf{X})$ over $ \mathcal{O}_{E, (v)} $ by taking $ \mathcal{S} $ to be the normalization of the closure of $ \Sh_{\mathsf{K}}(\textbf{G}, \textbf{X}) $ in $ \tilde{\mathcal{S}}\otimes_{\mathbb{Z}_{(p)}}\mathcal{O}_{E, (v)} $ and proving that it is smooth over $ \mathcal{O}_{E, (v)} $ (\cite[(1.3.4), (1.3.5)]{KisinModpPoints}). In particular, we obtain a finite morphism of schemes over $ \mathcal{O}_{E, (v)} $
\begin{align}
	\varepsilon: \mathcal{S}\to \tilde{\mathcal{S}}. 
\end{align}
We call the pull-back to $ \mathcal{S} $ of the universal abelian scheme $ \mathcal{A}$ the \textbf{universal abelian scheme} of $ \mathcal{S} $, still denoted by $ \mathcal{A} $. Denote by $$ S:=S_{\mathsf{K}},\ \  \tilde{S}:=\tilde{S}_{\tilde{\mathsf{K}}}, \  \ \varepsilon_{0}: S\to \tilde{S},\ \  \mathcal{A}_0$$  
the pull-back to  $ \Spec \kappa $ of
$ \mathcal{S}, \  \tilde{\mathcal{S}}, \ \varepsilon $ and $ \mathcal{A} $ respectively. We also call $ \mathcal{A}_0 $ the \textbf{universal abelian scheme} of $ S $. In particular, $ S $ is a quasi-projective, smooth scheme over $ \kappa$. 
\subsection{Reduction of Hodge cocharacters}\label{Sectionofcocharacters}

We retain notations from Section \ref{GoodreductionofShimuravariety}. Then for any $h\in \textbf{X}$, there is an associated Hodge cocharacter $$\nu_h: \mathbb{G}_m\to \textbf{G}_{\mathbb{C}}$$ by the formula $\nu_h(z):=h_{\mathbb{C}}(z, 1)$. More precisely, for any $\mathbb{C}$-algebra $R$, we have $R\otimes_{\mathbb{R}}\mathbb{C}=R\times c^*(R)$ where $c$ denotes complex conjugation.  Then on $R$-points the cocharacter $\nu_h$ is given by
$$R^{\times}\hookrightarrow (R\times c^*(R))^{\times}= (R\otimes_{\mathbb{R}}\mathbb{C})^{\times}=\mathbb{S}(R)\xrightarrow{h}\textbf{G}_{\mathbb{C}}(R).$$
Denote by $[\chi]_{\mathbb{C}}$ the unique $\textbf{G}(\mathbb{C})$-conjugacy class in $\Hom_{\mathbb{C}}(\mathbb{G}_{m, \mathbb{C}}, \textbf{G}_{\mathbb{C}})$ which contains all the $\nu_h$. Every element $h\in \textbf{X}$ defines a Hodge decomposition $V_{\mathbb{C}}=V^{(-1, 0)}\oplus V^{(0,-1)}$ via the embedding $\textbf{X}\hookrightarrow \textbf{S}^{\pm}$. By definition $\nu_{h}(z)$ acts on $V^{(-1, 0)}$ through multiplication by $z$ and on $V^{(0, -1)}$ as the identity. Consequently $V_{\mathbb{C}}$ has only weight $0$ and $-1$ with respect to any cocharacter $\lambda\in [\chi]_{\mathbb{C}}$.

Denote by $ G $ the special fibre of $ \mathcal{G} $, i.e., \begin{align}\label{DefofGSpecialfibre}
	G=\mathcal{G}_{\mathbb{F}_p}=\{g\in \GL(\Lambda_{\mathbb{F}_p})| g^{\vee}(s_{\mathbb{F}_p})=s_{\mathbb{F}_p} \text{ pointwise}\}.
\end{align}
It is a reductive group over $ \mathbb{F}_p $. We describe in the following how to obtain a $ G(\kappa) $-conjugacy class of cocharacters  of $ G_{\kappa} $ over $ \kappa $ from the Shimura datum $ (\textbf{G}, \textbf{X}) $, as it plays an important role for the study of reductions of Shimura varieties.	

As a preparation for what follows, we let $\mathsf{Z}: = \textbf{Hom}_{\mathbb{Z}_{(p)}} (\mathbb{G}_{m,\mathbb{Z}_{(p)} }, \mathcal{G})$ be the fppf sheaf of cocharacters, and $ \mathsf{Ch}: =\mathcal{G}\backslash\mathsf{Z} $ the quotient sheaf of $ \mathsf{Z} $ by the adjoint action of $ \mathcal{G} $. Then by \cite[Chapter XI, Corollary 4.2]{SGA3}, the sheaf $ \mathsf{Z}$ is represented by a smooth separated scheme over $ \mathbb{Z}_{(p)} $, and it is shown in \cite[Proposition 2.2.2]{ChaoZhangEOStratification} that $ \mathsf{Ch} $  is represented by a disjoint union of finite  \'etale (and hence proper) scheme over $ \mathbb{Z}_{(p)} $. 

Note that the reductive group $\mathcal{G}_{\mathbb{Z}_p}$ is quasi-split over $\mathbb{Z}_p$ and splits over a finite \'etale extension of $\mathbb{Z}_p$, as is the case for any reductive group over $\mathbb{Z}_p$ (\cite[Corollary 5.2.14]{ConradReductiveGroupSchemes}). Identifying the conjugacy class $[\chi]_{\bar{\mathbb{Q}}}$ with a point in $\mathsf{Ch}(\bar{\mathbb{Q}})$, then this point has a unique lift $c$ in $\mathsf{Ch}(E)$. Let $c_v$ be the pull back of $c$ in $\mathsf{Ch}(E_v)$. Then by the properness of $ \mathsf{Ch} $, $ c_v $ lifts uniquely to a point $\tilde{c}_v\in \mathsf{Ch}(\mathcal{O}_{E, v})$. Let $\bar{c}_v$ be the reduction of $\tilde{c}_v$ in $\mathsf{Ch}(k)$.
As $G_{\kappa}$ is quasi-split over $\kappa$, it follows from \cite[Lemma 1.1.3]{Kottwitz1TwistedOrbitalIntegrals} that there exists a cocharacter $\chi: \mathbb{G}_m\to G_{\kappa}$ such that  the $G(\kappa)$-conjugacy class of $\chi $, denoted by $[\chi]_{\kappa}$, is identified with the point $\bar{c}_v\in \mathsf{Ch}(\kappa)$. It is clear that $ [\chi]_{\kappa} $ is uniquely determined by the Shimura datum $ (\textbf{G}, \textbf{X}) $.

For any $\lambda\in [\chi]_{\kappa}$, $\Lie G_{\kappa}$ has only weights $0$ and $1$ with respect to the adjoint action of $\mathbb{G}_m$ through $\lambda$. If $\lambda$, via the embedding $ i: G\to \textbf{GSp}(\Lambda_{\kappa}) $,  induces a weight decomposition $\Lambda_{\kappa}=\Lambda_{\kappa}^{-1}\oplus \Lambda_{\kappa}^{0}$, then the cocharacter $(\ )^{\vee}\circ \lambda$  induces a weight decomposition
\begin{equation}\label{weightdecomposition}
\Lambda_{\kappa}^*=\Lambda_{\kappa}^{*,1}\oplus \Lambda_{\kappa}^{*,0}, \ \text{with} \ \ \Lambda_{\kappa}^{*,1}=(\Lambda_{\kappa}^{-1})^* \ \text{and}\  \Lambda_{\kappa}^{*,0}=(\Lambda_{\kappa}^0)^*,
\end{equation}
through the contragredient representation $(\cdot)^{\vee}: \GL(M)\cong \GL(M^*)$ (cf. Section \ref{Tensorsandcontragredient}).

We fix a cocharacter $ \chi=\chi_{\kappa}: \mathbb{G}_{m, \kappa}\to G_{\kappa} $ over $ \kappa  $ in $ [\chi]_{\kappa} $ and define the cocharacter 
\begin{align}\label{Defofmu}
	\mu: = \Frob_{G_{\kappa}/\kappa} \circ \chi: \mathbb{G}_{m,\kappa}\to G_{\kappa},
\end{align}
where $\Frob_{G_{\kappa}/\kappa}: G_{\kappa}\to G_{\kappa}^{(p)}= G_{\kappa}$ is the relative Frobenius of $ G_{\kappa} $ over $ \kappa $, and where we identify $ G_{\kappa} $ with $G_{\kappa}^{(p)}$ since it is defined over $ \mathbb{F}_p $. The cocharacter $ \chi $ determines a couple of subgroups of $ G_{\kappa} $. We introduce here all these subgroups though not all of them will be used immediately. 
\label{Parabolics induced by a cocharacter}

Denote by $P_+$ (resp.$P_{-}$) the parabolic subgroup of $G_{\kappa}$ which is characterized by the property that $\Lie(P_+)$ (resp. $\Lie(P_-)$) is the sum of the non-negative (resp. non-positive) weight spaces with respect to the adjoint operation of $\chi$ on $\Lie(G_{\kappa})$. Let $U_+$ (resp. $U_-$) be the unipotent radical of $P_+$ (resp. $P_-$) and let $M$ be the common Levi subgroup of $P_+$ and $P_-$, which is also the centralizer of $\chi$ in $G_{\kappa}$. 	

Since the scheme $ \mathsf{Z} $ is smooth, $ \chi_{\kappa} $ lifts to a cocharacter
\begin{align}
	\tilde{\chi}: \mathbb{G}_{m, W(\kappa)} \to \mathcal{G}_{W(\kappa)}
\end{align}  of $ \mathcal{G}_{W(\kappa)} $, which we fix from now on. We define a cocharacter $ \tilde{\mu}$ over $ W(\kappa) $ of $\mathcal{G}_{W(\kappa)}$ as 
\begin{align} \label{DefofTildeMu}
	\tilde{\mu}:={\tilde{\chi}}^{(p)}:  \mathbb{G}_{m, W(\kappa)}\cong \mathbb{G}_{m, W(\kappa)}^{(p)}\to  \mathcal{G}_{W(\kappa)}^{(p)}\cong \mathcal{G}_{W(\kappa)},
\end{align}
where we identify $ \mathcal{G}_{W(\kappa)}^{(p)} $ with $ \mathcal{G}_{W(\kappa)} $ since the latter is defined over $ \mathbb{Z}_p $ and similarly for $ \mathbb{G}_{m, W(\kappa)} $.

As in practice it is often convenient to work over the completion $ \mathcal{O}_{E,v}=W(\kappa) $, the ``integral model $ \mathcal{S} $'' appearing in the sequel will always mean the pull back $ \mathcal{S}\otimes_{\mathcal{O}_{(v)}}\mathcal{O}_v $ unless otherwise stated.

	\section{Construction of a morphism from $\mathrm{I}_+$ to  $\mathcal{D}_1$ }

\subsection{Recall of notations}\label{specification of notations}
We retain notations from Section \ref{GoodreductionofShimuravariety} and Section \ref{Sectionofcocharacters}. 
We specify from now on the field $ k $ and the group $ G $ in Section \ref{sectionLoopSchemes} to be $ \kappa $ and  $\mathcal{G}_{\kappa} $ respectively. Now we can form the ind-scheme $ \mathcal{L}G $ over $ \kappa $ and the following infinite dimensional group schemes, all over $ \kappa $:  $$  \mathcal{K}=\mathcal{L}^{+}G, \ \mathcal{K}_1,\  \mathcal{K}^{+}, \ \mathcal{K}^{\diamond}. $$ Also we can form the schemes 
\begin{align*}
	\mathcal{D}:=\mathcal{D}(G, \mu(u)),\ \ \  \mathcal{D}_1:=\mathcal{D}_1(G, \mu(u)),
\end{align*}
where $ \mathcal{D} $ is an infinite dimensional subscheme of $ \mathcal{L}G $, $ \mathcal{D}_1 $ is a smooth scheme of finite type over $ \kappa $, and $ \mathcal{D} $ is a $ \mathcal{K}_1 $-torsor over $ \mathcal{D}_1 $. 

For any $ \kappa $-algebra $ A $, we continue to denote by $ \varphi : \mathfrak{S}(A)\to \mathfrak{S}(A)$ the absolute Frobenius of $ \mathfrak{S}(A) $ while we use $ \sigma: \mathfrak{S}(A)\to \mathfrak{S}(A) $ to denote the ring endomorphism of $ \mathfrak{S}(A) $, which is the absolute Frobenius on $ A $ and sends the variable $ u $ to itself. 
We have seen in Section \ref{TwoFrobenii} that there are two natural Frobenii $ \sigma $ and $ \varphi $ on $ \mathcal{L}G $ and on $ \mathcal{L}^{+} G$. Seeing $ G $ as a subgroup of $ \GL(\Lambda_{\kappa}) $, here we write down explicitly $ \sigma $ and $ \varphi $. For any $ \kappa $-algebra $ A $, and any element $ g\in \mathcal{L}^+G(A) $, seeing as an isomorphism $ g: \Lambda_{\mathfrak{S}(A)} \cong \Lambda_{\mathfrak{S}(A)}$ preserving tensors $ s_{\mathfrak{S}(A)} $, we have 
\begin{align}\label{ExpliteFromular4TwoFrobinii}
	\sigma(g)=\varepsilon_1\circ \sigma^{*}g \circ \varepsilon_1,\ \ \  \varphi(g)=\varepsilon_2\circ \varphi^{*}g \circ \varepsilon_2, 
\end{align}
where $ \sigma^*g $ is the pull back along $ \sigma: \mathfrak{S}(A) \to \mathfrak{S}(A)$ of $ g $ and similarly for $ \varphi^*g $, and where 
\begin{align*}
	\varepsilon_1:	\Lambda_{\mathfrak{S}(A)}^{(\sigma)}\cong \Lambda_{\mathfrak{S}(A)}, \ \ \  \varepsilon_2:	\Lambda_{\mathfrak{S}(A)}^{(\varphi)}\cong \Lambda_{\mathfrak{S}(A)} 
\end{align*} are canonical isomorphisms since $ \Lambda $ is defined over $ \mathbb{Z}_p $. We have similar descriptions for those $ g\in \mathcal{L}G(A) $. 

Note that we use $\sigma, \varphi: G(A)\to G(A)$ interchangeably for the relative (= absolute) Frobenius of $ G(A) $. 

\subsection{Some torsors over  $ \mathcal{S} $ and $ S $}\label{TorsorsoverSandS}
Write $\mathcal{V}=\HH^1_{\dR}(\mathcal{A}/\mathcal{S})$ and $ \mathcal{C}\subset \mathcal{V} $ the Hodge filtration of $ \mathcal{V} $. Recall that $ \mathcal{V} $ comes with a finite set of tensors $ s_{\dR}\subset \mathcal{V}^{\otimes} $. 
Recall also that we have fixed a cocharacter $ \tilde{\chi}: \mathbb{G}_{m, W(\kappa)}\to \mathcal{G}_{W(\kappa)} $ over $ W(\kappa) $ in Section \ref{Sectionofcocharacters} which  is a lift of $ \chi_{\kappa} $.
The cocharacter $()^{\vee}\circ \tilde{\chi}$ induces a $\mathbb{Z}$-grading of $ W(\kappa) $-modules:
$\Lambda^*_{W(\kappa)}=(\Lambda^*)_{\tilde{\chi}}^0\oplus(\Lambda^*)^1_{\tilde{\chi}}$.   From this grading we obtain a descending filtration
\begin{align}\label{Filtration}
	&\Fil_{\tilde{\chi}}^0:=\Lambda^*_{W(\kappa)}\supset \Fil_{\tilde{\chi}}^1:=(\Lambda^*)^1_{\tilde{\chi}}\supset\Fil_{\tilde{\chi}}^2:=0.
\end{align}

We can then define two schemes over $ \mathcal{S} $ from this filtration,
\begin{align*}
	&\mathbb{I}:=\Isom_{\mathcal{S}}\big([\Lambda^*_{W(\kappa)}, s_{W(\kappa)}]\otimes O_{\mathcal{S}}, [\mathcal{V}, s_{\dR}]\big),\\
	&\mathbb{I}_+:=\Isom_{\mathcal{S}}\big([\Lambda^*_{W(\kappa)}, s_{W(\kappa)}, \Fil_{\tilde{\chi}}^{\bullet} ]\otimes O_{\mathcal{S}}, [\mathcal{V}, s_{\dR}, \mathcal{V}\supset\mathcal{C}]\big).
\end{align*}
Let $\mathcal{P}_+$ be the stabilizer of $\Fil_{\chi}^{\bullet}$ in $\mathcal{G}$. It is a parabolic subgroup of $ \mathcal{G} $ over $ W(\kappa) $, and has a similar description as $ P_+ $ (cf. Section \ref{Sectionofcocharacters}).

\begin{lemma} \label{Torsoroverintegralmodel}
	Let $A$ be a $ p $-adic $W(\kappa)$-algebra and write $X=\Spec A$. Then for any morphism of schemes $ X\to \mathcal{S}$ over $W(\kappa)$, the fullback $\mathbb{I}_{+, X}$ is a $\mathcal{P}_+$-torsor over $X$. Moreover, for any closed point $x$ of $X$, if we denote by $\hat{x}: \Spec \widehat{\mathcal{O}}_{X, x}\to \mathcal{S}$ the induced morphism, then the pullback to $\Spec \widehat{\mathcal{O}}_{X, x}$ of $ \mathbb{I}_+ $, denoted by $\mathbb{I}_{+, \hat{x}}$, is a trivial $\mathcal{P}_+$-torsor over $\Spec \widehat{\mathcal{O}}_{X, x}$. 
\end{lemma}

\begin{proof}
	Note that if $G$ is a smooth group scheme over a scheme $S$, then a scheme $Y$ of finite presentation over $S$ is an fppf $G$-torsor over $ S $ if and only if the action of $G$ on $Y$ is free and transitive and $Y$ is faithfully flat over $S$. Indeed if $Y\to S$ is faithfully flat, then it is an fppf cover of $S$ and the scheme $Y\times_SY$ is a trivial $G\times_SY$-torsor over $Y$ since $Y\times_SY$ has a $Y$-point given by the diagonal map $Y\to Y\times_SY$ and the action of $G\times_SY$ on $Y\times_SY$ is free and transitive.
	
	For a closed point $ x $ of $ X $, we also denote by $x:\Spec \mathcal{O}_{X, x}\to \mathcal{S}$ the induced morphism. By definition of $\mathbb{I}_+$, the action of $\mathcal{P}_+$ on $ \mathbb{I}_+ $ is free and transitive. Hence due to the discussion above,  for the first assertion in the lemma it suffices to show that for each closed point $ x $ of $ X $, the pullback $\mathbb{I}_{+, x}$ is a $\mathcal{P}_+$-torsor over $ \Spec  \mathcal{O}_{X, x}$. Since $A$ is $ p $-adically complete, the image  in $ \mathcal{S} $ of $ x $, denoted by $ s $,  necessarily lies in the special fibre $ S $.
	
	On the other hand, by \cite[Lemma 2.3.2,  2)]{ChaoZhangEOStratification}, if $s'$ is any closed point of $ \mathcal{S} $ which is a specialization of $ s $ ($s'$ exists since $ \mathcal{S} $ is (locally) noetherian, see \cite[Tag 01OU]{stacks-project}), the fibre $ \mathbb{I}_{ +,\widehat{s'}}$ is a trivial $ \mathcal{P}_+ $-torsor over $ \widehat{\mathcal{O}}_{\mathcal{S}, s'} $. But since the canonical morphism $ \Spec \widehat{\mathcal{O}}_{\mathcal{S}, s}  \to \mathcal{S}$ factors through $\Spec \widehat{\mathcal{O}}_{\mathcal{S}, s'}\to \mathcal{S}$,  the fibre $ \mathbb{I}_{ +,\hat{s}}$ is also a trivial $ \mathcal{P}_+ $-torsor over $ \widehat{\mathcal{O}}_{\mathcal{S}, s} $. In particular, $\mathbb{I}_{ +,\hat{s}}$ is faithfully flat over $ \widehat{\mathcal{O}}_{\mathcal{S}, s} $. Now by faithfully flat descent for $\Spec  \widehat{\mathcal{O}}_{\mathcal{S}, s} \to \Spec \mathcal{O}_{\mathcal{S}, s}$, the fibre $ \mathbb{I}_{+,s}$ is also faithfully flat over $ \Spec \mathcal{O}_{\mathcal{S}, s} $. Now the first assertion follows from the fact that $x:\Spec \mathcal{O}_{X, x}\to \mathcal{S}$ factors through the canonical embedding $\Spec \mathcal{O}_{\mathcal{S}, s} \to \mathcal{S}$ while the second assertion follows from the fact that $ \hat{x}: \Spec \widehat{\mathcal{O}}_{X, x}\to \mathcal{S}  $ factors through $\Spec \widehat{\mathcal{O}}_{\mathcal{S}, s}\to \mathcal{S}$. 
\end{proof}	

\begin{remark}
	A similar statement holds true for $\mathbb{I}  $ (with respect to $\mathcal{G}$). The lemma above is in a sense an ad hoc one, since we believe that the following slightly stronger statement holds true: $\mathbb{I}_+$ is a $\mathcal{P}_+$-torsor over $ \mathcal{S} $ and $\mathbb{I}  $ is a $\mathcal{G}$-torsor over $\mathcal{S}$. In fact, the stronger assertion here concerning $\mathbb{I}$ follows readily from the combination of \cite[Lemma 2.3.2,  1)]{ChaoZhangEOStratification} and (3.2.1) in \cite{KimRapoportUniformisation}. We hope to give a whole proof of this stronger version in the future. 
\end{remark}



Write $ \mathcal{V}_0=\HH^1_{\dR}(\mathcal{A}_0/S)$ and $\mathcal{C}_0 \subset \mathcal{V}_0$ the Hodge filtrations of $ \mathcal{V}_0 $. By functoriality of the de Rham cohomology one knows that $  \mathcal{C}_0 \subset \mathcal{V}_0$ is the pullback to $ S $ of $ \mathcal{C}\subset\mathcal{V}  $. Denote by $\bar{s}_{\dR}\subset \mathcal{V}_{0}^{\otimes}$ the reduction of the set of tensors $s_{\dR}\subset \mathcal{V}^{\otimes}$, and by $\bar{s}$ the base change to $(\Lambda^*_{\kappa})^{\otimes}$ of $s\subset (\Lambda^*)^{\otimes}$. Now we define  the following schemes over $ S $,
\begin{align}\label{Torsor}
	&\mathrm{I}:=\Isom_S\big([\Lambda_{\kappa}^*, \bar{s}]\otimes O_{S}, [\mathcal{V}_0, \bar{s}_{\dR})\big],\\
	&\mathrm{I}_+:=\Isom_S\big([\Lambda_{\kappa}^*, \bar{s}, \Fil_{\chi, \kappa}^{\bullet} ]\otimes O_{S}, [\mathcal{V}_0, \bar{s}_{\dR}, \mathcal{V}_0\supset\mathcal{C}_0]\big),
\end{align}
where the $\Fil_{\chi, \kappa}^{\bullet}$ denotes the base change to $ S $ of the filtration $\Fil_{\tilde{\chi}}^{\bullet}$ in \eqref{Filtration}. Clearly $ \mathrm{I}, \mathrm{I}_{+} $ are base changes to $ S $ of $ \mathbb{I}, \mathbb{I}_{+}$ respectively. By Theorem 2.4.1, 2) in \cite{ChaoZhangEOStratification}, $ \mathrm{I} $ is a $ G $-torsor over $ S $ and $ \mathrm{I}_{+} $ is a $ P_{+} $-torsor over $ S $.

\subsection{Frobenius invariance of tensors $s_{\dR}$}\label{SectionFrobeniusInvariance}
Let $ R_0 $ be a $ \kappa $-algebra as in Lemma \ref{Frobenius lifts}, i.e., $ R_0 $ locally admits a finite $ p $-basis (e.g., $ R_0 $ is a smooth $ \kappa $-algebra of finite type), and $ (R, \varphi) $ a simple frame of $ R_0 $. Assume from now on that $ R $ is noetherian. Let $ x: \Spec R\to \mathcal{S} $ be a morphism of $ W(\kappa) $-schemes. Write $H=\mathcal{A}_{x}[p^{\infty}]$ and  $ H_{0}= H\otimes_RR_0$.

By Theorem \ref{ClassificationbyRmodules}, associated to the $ p $-divisible group $ H $ is a tuple $ (M, N, F=F_M, \nabla) $, where \begin{enumerate}[$\bullet$]
	\item $M=\mathbb{D}^{*}(H_0) (R)\cong \HH^{1}_{\dR}(\mathcal{A}_{x}/R)=\mathcal{V}_{x}$;
	\item $ F=F_M:M\to M $ is the Frobenius endomorphism of $ M $ (c.f. Section \ref{ClassificationofBTs/R_0});
	\item $ N\subset M $ is the Hodge filtration of $ M $, i.e., $ N $ is the preimage of $ \mathcal{C}_x $ in $ M $ under the canonical isomorphism $ M \cong \mathcal{V}_{x}$;
	\item  $ \nabla=\nabla_{M}: M\to M\otimes_{R}\hat{\Omega}_{R} $ is an integrable topologically nilpotent connection over $ d_{R} : R\to \hat{\Omega}_{R}$, whose existence is stated in Theorem \ref{ClassificationofBTs/R_0},  and with respect to which $ F $ is horizontal. 
\end{enumerate} 


By identifying the filtered modules $ M\supset N $ with $ \mathcal{V}_{x} \supset \mathcal{C}_{x}$ we also obtain a set of tensors $ s_{\dR, R}\subset M^{\otimes}$ over $ M$. 	
\begin{lemma}\label{Frobenius invariance of tensors}
	The Frobenius endomorphism $F_M: M\to M$, after inverting $ p $, preserves the tensors $s_{\dR,R}$ pointwise.
\end{lemma}
\begin{proof}
	This essentially follows from the discussion in \cite[1.5.4]{KisinIntegralModels}. Indeed, for any maximal idea $ \mathfrak{m} $ of $ R $,	by Lemma \ref{Frobenius lifts}, (4) the Frobenius lift $ \varphi: R\to R $ induces a simple frame ($ \hat{R}_{\mathfrak{m}}, \varphi )$ of $  \hat{R}_{\mathfrak{m}} $, compatible with $ (R, \varphi_R) $. Note that $\hat{R}_{\mathfrak{m}}$ is necessarily $ p $-adically complete ( since $  \mathfrak{m}  $ contains $ p $). Hence it suffices to show the lemma after base change to $\hat{R}_{\mathfrak{m}}$ for all $ \mathfrak{m} $. In particular, we may assume that $ R $ is a local ring. 
	
	Let $ s'\in \mathcal{S} $ be the image of the closed point of $ \Spec R $, which necessarily lies in the special fibre $ S\subset \mathcal{S} $. As in the proof of Lemma \ref{Torsoroverintegralmodel}, there exists a closed point $s\in  S $ such that $ s $ is a specialization in $ \mathcal{S} $ of $ s' $. Then the morphism $ x: \Spec R\to \mathcal{S} $ facts through the canonical embedding $ s: \Spec\hat{ \mathcal{O}}_{\mathcal{S}, s} \to \mathcal{S}$.  We know from the proof of Corollary 2.3.9 in \cite{KisinIntegralModels} (the second paragraph) that $ \hat{ \mathcal{O}}_{\mathcal{S}, s} $ is isomorphic to the deformation ring $ R_{G_W} $ of the $ p $-divisible group $H_{0,s}:= \mathcal{A}_s[p^{\infty}]\otimes_{\hat{ \mathcal{O}}_{\mathcal{S}, s}} k(s)$ (equipped with Tate tensors) over $ k(s) $ induced by the embedding $ s $. The ring $ R_{G_W} $ is defined in the the proof of Proposition 3.5 in loc. cit. but note that there is a notation difference since this ring is irrelevant of our $ R $ in question: in order to avoid confusion we will denote it by $ A $. 
	
	We may identify $\hat{ \mathcal{O}}_{\mathcal{S}, s}$ with $ A $. The ring $ A $ is in fact a quotient of Faltings's versal deformation ring $ B $ of $ H_{0,s} $, which is isomorphic to a ring of power series over $W(k(s))$ (Section 1.5 of loc. cit.). One may arrange so that $ A=B/I $, with $ I\subset B $ an ideal generated by some formal variables of $ B $. Then if we denote by $ \varphi_B: B\to B $ the Frobenius lift of $ B $ obtained by sending all formal variables of $ B $ to their $ p $-th powers, $ \varphi_{B} $ induces a Frobenius lift $ \varphi_{A}: A\to A $, i.e., we obtain a homomorphism of simple frames $ (B, \varphi_B) \to (A, \varphi_A)$.  Write $ p: B\to A $ for the canonical projection. Let $ (N, F_N, \nabla_{N}) $ be the Dieudonn\'e module of the tautological $ p $-divisible group $ H_{B} $ over $ B $ as in \cite[Section 1.5]{KisinIntegralModels} (denoted by $ (M, \varphi, \nabla) $ in loc.cit.). Then the triple
	\begin{align*}
		(L:=N\otimes_{B}A, F_L: =F_N\otimes \varphi_{A}, \nabla_{N}\otimes\id_A)
	\end{align*}
	corresponds to the Dieudonn\'e module of $ p^*H_{B} $ over $ A $. Let $ g: A\to R $ be the induced homomorphism and write $ f =g\circ p $. Since $ \varphi_R\circ f $ and $ f\circ \varphi_{B} $ become the same after modulo $ p $, we obtain from the definition of Dieudonn\'e crystals a canonical isomorphism 
	\begin{align*}
		\varepsilon: \varphi_{R}^* f^* N\cong \mathbb{D}^*(\varphi_{R}^* f^* H_{B}\otimes_{R}R_0)= \mathbb{D}^*(f^*\varphi_{B}^*  H_{B}\otimes_{R}R_0)\cong f^*\varphi_{B}^*  N.
	\end{align*}
	
	Now we have the following identifications: 
	\begin{align*}
		M=N\otimes_{B}R,\ \  \nabla_{M}=\nabla_{N}\otimes_{B}\id_R,\ \  F_M^{\mathrm{lin}}=(F_N^{\mathrm{lin}} \otimes\id_R)\circ \varepsilon=(F_L^{\mathrm{lin}} \otimes\id_R)\circ \varepsilon.
	\end{align*}

	Note that the tensors $ s_{\dR, R} $ are obtained from the tensors $ s_{\dR} $ over $ N $ by base change, i.e., $  s_{\dR, R}= s_{\dR}\otimes_B1$. By \cite[1.5.4]{KisinIntegralModels}, one finds
	\begin{align*}
		\varepsilon \big( s_{\dR}\otimes_B1\otimes_{R, \varphi_{R}} 1\big)=  s_{\dR}\otimes_{B, \varphi_{B}}1\otimes_{B} 1=  s_{\dR, A}\otimes_{A, \varphi_{A}}1\otimes_{A} 1\in \big( \varphi_{A}^{*}L\otimes_{A}R\big)^{\otimes}
	\end{align*}
	On the other hand, we know that $F_L^{\mathrm{lin}} \otimes\id_R  $ sends the tensors $ s_{\dR, A}\otimes_{A, \varphi_{A}}1 $ to $s_{\dR, A}$ (see  \cite[Corollary 4.8]{Wortmann} for a clear statement though the proof is essentially in \cite{KisinIntegralModels}). Hence $ F_M ^{\mathrm{lin}}$ sends the tensors $ s_{\dR, R} \otimes_{R,\varphi_{R}}1$ to $ s_{\dR, R} $. In other words, $ s_{\dR, R} $ is $ F_M $-invariant.     
\end{proof}

Each element $ x\in \mathbb{I}(R)$ by definition gives an isomorphism $\beta_{x}: \Lambda^{*}_{R}\cong M$, sending the tensors $ s_{R} $ to the tensors $ s_{\dR, R} $ pointwise. 
We can transform the Frobenius map $F$ to $\Lambda_R^*$ via $\beta_{R}$ by defining $$F_{x}:=\beta_{x}^{-1}\circ F\circ \beta_{x}.$$
Note that we have a canonical isomorphism $\varepsilon: (\varphi^* \Lambda_{R}^*, \ \varphi^{*}s_{R})\longrightarrow (\Lambda_{R}^*, s_{R}).$ Define
\begin{align}\label{DefoflinearFrob}
	F_{x}^{\mathrm{lin}}:=\big(V_R^{*}\xrightarrow{\varepsilon^{-1}} \varphi^*V_R^{*}=V_R^{*}\otimes_{R, \varphi} R\xrightarrow{F_x\otimes \id} V_R^{*}\big),
\end{align}
which by Lemma  \ref{Frobenius invariance of tensors} preserves the tensors $ s_{R} $, i.e., we have $ (F_{x}^{\mathrm{lin}})^{\vee}\in \mathcal{G}(R[\frac{1}{p}]).$  
For the next lemma, one may recall the definition of the cocharacter $ \tilde{\mu}: \mathbb{G}_{m, W(\kappa)}\to \mathcal{G}_{W(\kappa)} $  in \eqref{DefofTildeMu}. 
\begin{lemma}\label{FrobeniuselementLocus}
	For any $x\in \mathbb{I}_+(R)$ we have 
	$(F_{x}^{\mathrm{lin}})^{\vee}\in\mathcal{G}(R)\tilde{\mu}(p)$ and for any $x\in \mathbb{I}(R)$ we have 
	$(F_{x}^{\mathrm{lin}})^{\vee}\in\mathcal{G}(R)\tilde{\mu}(p)\mathcal{G}(R)$. 
\end{lemma} 
\begin{proof}
	We show only the first part of the claim since it implies the second part. 
	
	For any $x\in \mathbb{I}_+(R)$, the weight decomposition $\Lambda_R^*=\Lambda_R^{*,1}\oplus \Lambda_R^{*,0}$ (cf. \eqref{weightdecomposition}) induced by the cocharacter $\tilde{\chi}_R: \mathbb{G}_{m, R}\to \mathcal{G}_R$ induces via $\beta_{x}$ a decomposition $M=N\oplus L$. It suffices to show the following:
	\begin{align}\label{equation decompostion}
		F_{x}^{\mathrm{lin}}=  \Gamma^{\mathrm{lin}}_{x} \circ \big(\tilde{\mu}(p)\big)^{\vee}, \ \ \ \text{with}\ \  \Gamma^{\mathrm{lin}}_{x}:=\beta_{x}^{-1}\circ \Gamma^{\mathrm{lin}}\circ \varphi^{*}(\beta_{x}) \circ \varepsilon,
	\end{align}
	where $\Gamma^{\mathrm{lin}}$ is defined as in \eqref{DefinitionOfGammaLin}. Indeed this can been seen from the following commutative diagram 
	\begin{align*}
		\xymatrixcolsep{4pc}\xymatrix{V_R^{*}\ar[d]_{\varepsilon}\ar[rr]^{\big(\tilde{\mu}(p)\big)^{\vee}}&&V_R^*\ar[d]^{\varepsilon}\ar[rr]^{\Gamma_{x}^{\mathrm{lin}}}&&V_R^*\ar[dd]^{\beta_{x}}\\	\varphi^* V_R^*\ar[rr]^{\varphi^{*}\big(\tilde{\chi}(p)^{\vee}\big)}\ar[d]_{\varphi^*(\beta)}&&\varphi^*V_R^*\ar[d]_{\varphi^*(\beta)}&&\\
			N^{(\varphi)}\oplus L^{(\varphi)}\ar[rr]^{p\cdot\id \oplus \id }&&N^{(\varphi)}\oplus L^{(\varphi)}\ar[rr]^{\Gamma^{\mathrm{lin}}}&&M}
	\end{align*}
\end{proof}

\subsection{The morphism $ \theta:\mathrm{I}_+\to \mathcal{D}_1$}
As suggested by the title, in this subsection we construct a morphism of $ \kappa $-schemes from $ \mathrm{I}_{+}$ to $ \mathcal{D}_{1}$. We also give an explicit description of $ \theta $ on geometric points, where we skip the notion of Kim-Kisin windows and adapted deformations.

\subsubsection{Main Construction $ \theta:\mathrm{I}_+\to \mathcal{D}_1$ }\label{Maintheorem1}

\textbf{Affine open coverings for $\mathbb{I}_+$.}

Take an affine open covering $ \{\Spec\tilde{R}_i\to \mathcal{S}\}_{i\in I} $ of the integral model $ \mathcal{S} $ over $W(\kappa) $ such that 
\begin{align}
	\text{	$\mathcal{C}_{\tilde{R}_i}\subset \mathcal{V}_{\tilde{R}_i} $ are free modules over $ \tilde{R}_{i} $ for each $ i $}.  
\end{align}\label{Covering}
Since $ \mathcal{S} $ is quasi-projective (hence separated), the intersection of 	$\Spec \tilde{R}_i$ and $\Spec \tilde{R}_j$ is again affine, and hence is the spectrum of some $ W(\kappa)$-algebra, say $ \tilde{R}_{ij} $. For each $ i, j $, since the torsors $ \mathbb{I}_{+, \tilde{R}_i}, \mathbb{I}_{+, \tilde{R}_{ij}} $ are affine over $ \tilde{R}_i $ and over $ \tilde{R}_{ij} $ respectively, we can write
\begin{align}
	\mathbb{I}_{+, \tilde{R}_i} =\Spec \tilde{A_i}, \ \mathbb{I}_{+, \tilde{R}_{ij}} =\Spec \tilde{A}_{ij}
\end{align} 
for some $ W(\kappa) $-algebras $ \tilde{A}_i $ and $ \tilde{A}_{ij} $. Denote by $ R_{i,0}, A_{i,0}, R_{ij,0}, A_{ij,0} $ the reduction modulo $ p $ of $ \tilde{R}_{i}, \tilde{A}_{i}, \tilde{R}_{ij}, \tilde{A}_{ij} $ respectively. 
Clearly, we obtain affine open coverings 
\begin{align}
	\{\Spec R_{i,0} \to S\}_{i\in I},\ \  \{\Spec A_{i,0} \to \mathrm{I}_+\}_{i\in I}
\end{align} for $ S $ and for $ \mathrm{I}_+ $ respectively, with 
\begin{align*}
	\Spec R_{i,0}\cap\Spec R_{j,0}=\Spec R_{ij,0}, \ \ \ \Spec A_{i,0}\cap\Spec A_{j,0}=\Spec A_{ij,0}.
\end{align*}

In the following we do the construction in four steps.

\textbf{Step 1: Preparations}

Write $ H_{i,0}:=\mathcal{A}[p^{\infty}]_{A_{i,0}} $ and let $ A_i $, $ A_{ij} $ be the $ p $-adic completions of $ \tilde{A}_i, \tilde{A}_{ij} $ respectively. The induced morphism $ x_{i}:\Spec A_{i} \to \mathbb{I}_+ $ gives an isomorphism of free  $ A_i $-modules, 
$$ \beta_{x_i}: \Lambda^*_{A_i}\cong M_i,$$
which respects the filtrations and sends the tensors $ s_{A_i} $ to the tensors $ s_{\dR, A_i} $ pointwise. Here as usual $ M_i $ is defined as
\begin{align*}
	M_i:=\HH^1_{\dR}(\mathcal{A}_{A_i}/A_i)\cong \mathbb{D}^{*}(H_{i,0})(A_i).
\end{align*} 
We write 
$$\mathfrak{S}_i=\mathfrak{S}(A_i),\ \ \  \mathfrak{S}_{ij}=\mathfrak{S}(A_{ij}).$$
Denote by $ M_{i,0}, \mathfrak{S}_{i,0}, \mathfrak{S}_{ij,0} $, $ x_{i,0} $ and $ \beta_{x_{i,0}} $ the reduction modulo $ p $ of $ M_i,  \mathfrak{S}_i, \mathfrak{S}_{ij}, x_{i}$ and $ \beta_{x_i} $ respectively.
For each $ i\in I $ we choose a Frobenius lift $ \varphi_{i}$ for $ A_i $ and denote by  $ \chi_{x_{i}}: \mathbb{G}_{m, A_i}\to \GL(M_i) $ the cocharacter of  $\GL(M_i)$ over $ A_i $ induced by $\chi_{A_i}: \mathbb{G}_{m,A_i}\to \mathcal{G}_{A_i} $ via $ \beta_{x_i}$.  Consider the adapted deformation $ \underline{\mathcal{M}_i}=(\mathcal{M}_i, \Phi(\varphi_i, \chi_{x_i}), \ldots) $ of $H_{i,0} $ w.r.t. the pair $ (\varphi_{i}, \chi_{x_i}) $\footnote{Here in the Kim-Kisin window $ \underline{\mathcal{M}_i} $ we omit all terms except the underlying $ \mathfrak{S}(A_i) $-module $ \mathcal{M}_i $ and its Frobenius $ \Phi(\varphi_i, \chi_{x_{i}}) : \mathcal{M}_i^{(\varphi_i)}\to \mathcal{M}_i$.}. We refer to Section \ref{SectionofConstructionofKisinWindows} for details on the Kim-Kisin window $\underline{ \mathcal{M}_i} $ but recall that we use 
\begin{align}\label{Phi_0}
	\Phi_0(\varphi_i, \chi_{x_{i}}): \mathcal{M}_{i,0}^{(\varphi)}\to \mathcal{M}_{i,0} 
\end{align} 
to denote the reduction modulo $ p $ of the Frobenius $ \Phi(\varphi_i, \chi_{x_{i}}) : \mathcal{M}_i^{(\varphi_i)}\to \mathcal{M}_i$, where 
\begin{align*}
	\mathcal{M}_i:= M_i\otimes_{A_i}\mathfrak{S}_i,\ \ \  \mathcal{M}_{i,0}:= M_{i,0}\otimes_{A_{i,0}}\mathfrak{S}_{i,0}.
\end{align*}   

\textbf{Step 2: Construction of local morphisms $ \theta_i : \Spec A_{i,0}\to \mathcal{D}_{1}$.}

For each $ i\in I $, we construct a morphism of $ \kappa $-schemes $ \theta_{i}: \Spec A_{i,0}\to \mathcal{D}_1 $, equivalently, an element $ \theta_{i}(x_{i,0}) \in \mathcal{D}_{1}(A_{i,0})$, in the following way. Define 
\begin{align}\label{longformula}
	\Phi_0(\varphi_i, x_{i}): =\big((\beta_{x_{i,0}}\otimes\id_{\mathfrak{S}_{i,0}})^{-1}\circ \Phi_0(\varphi_{i}, \chi_{x_i})\circ \varphi(\beta_{x_{i,0}}\otimes\id_{\mathfrak{S}_{i,0}})\big)^{\vee}
\end{align}
which after inverting $ u\in  \mathfrak{S}_{i,0}$ lies in $\GL\big(\Lambda_{\mathfrak{S}_{i,0}}[\frac{1}{u}]\big)$, where we use the contragredient representation $$(\cdot)^{\vee}: \GL\big(\Lambda^{*}_{\mathfrak{S}_{i,0}}[\frac{1}{u}]\big)\longrightarrow \GL\big(\Lambda_{\mathfrak{S}_{i,0}}[\frac{1}{u}]\big).$$
By Lemma \ref{Varphipreservestensors} below, 	$\Phi_0(\varphi_i, x_{i})$ lies in $ \mathcal{D}(A_{i,0}) $. Now we define 
\begin{align}\label{The morphism theta}
	\theta_{i}(x_{i,0}): =  \overline{\Phi_0(\varphi_i, x_{i})},
\end{align}
where $ \overline{\Phi_0(\varphi_i, x_{i})} $ denotes the image in $ \mathcal{D}_1(A_{i,0}) $ of $\Phi_0(\varphi_i, x_{i})\in \mathcal{D}(A_{i,0})$. By Lemma \ref{Gluinglemma2}, the element $ \theta_{i}(x_{i,0}) $ is independent of the choice of the Frobenius lift $ \varphi_i $ of $ A_i $.

\begin{lemma}\label{Varphipreservestensors}
	$\Phi_0(\varphi_i, x_{i})$ lies in $ \mathcal{D}(A_{i,0}) $.
\end{lemma}

\begin{proof} We show this by giving a formula for $\Phi_0(\varphi_i, x_{i})$. We use the  notation $ \Gamma_{x_i}^{\mathrm{lin}} $ as in \eqref{equation decompostion}. Define
	\begin{align*}
		\alpha(x_i): = (\Gamma_{x_i}^{\mathrm{lin}}\mod p)^{\vee}\in G(A_{i,0}).
	\end{align*}
	Then by Remark \ref{RemarkofKisinWindows}. (a) we have 
	\begin{align}\label{KeyFormula}
		\Phi_0(\varphi_i, x_{i})=\alpha(x_i)\mu(u)\in \mathcal{D}(A_{i,0}).
	\end{align}
\end{proof}
\textbf{Step 3: Gluing of the local morphisms}

We shall show that for all $ i, j\in I $, we have
\begin{align*}
	\theta_i(x_{i,0})|_{A_{ij,0}}=\theta_j(x_{j,0})|_{A_{ij,0}}.
\end{align*}
By the way we construct the morphisms $ \theta_i $ and $ \theta_j $, we need only to show that the two homomorphisms 
$$\Phi_0(\varphi_{i}, \chi_{x_{i}})\otimes_{\mathfrak{S}_{i,0}} \id_{\mathfrak{S}_{ij,0}},\  \Phi_0(\varphi_{j}, \chi_{x_{j}})\otimes_{\mathfrak{S}_{j,0}} \id_{\mathfrak{S}_{ij,0}}:\mathcal{M}_{ij,0}^{(\varphi)}\to \mathcal{M}_{ij,0} $$
differ by an automorphism of $ \mathcal{M}_{ij,0}^{(\varphi)} $ whose reduction modulo $ u $ is the identity map, where  
\begin{align*}
	\mathcal{M}_{ij,0}:=\mathcal{M}_{i,0}\otimes_{\mathfrak{S}_{i,0}}{\mathfrak{S}_{ij,0}}=\mathcal{M}_{j,0}\otimes_{\mathfrak{S}_{j,0}} {\mathfrak{S}_{ij,0}}.
\end{align*}
This is in fact clear from the functoriality of adapted deformations. Indeed, if we still denote by $ \varphi_i $ the Frobenius lift of $ A_{ij} $ induced by $ \varphi_i $. Then by Section \ref{FunctorialityofKim-Kisinwindows}
we have \begin{align*}
	\Phi(\varphi_{i}, \chi_{x_{i}})\otimes_{\mathfrak{S}_{i}} \id_{\mathfrak{S}_{ij}}:&=\Phi(\varphi_{i}, \chi_{x_i,A_{ij}})\\
	\Phi(\varphi_{j}, \chi_{x_{j}})\otimes_{\mathfrak{S}_{j}} \id_{\mathfrak{S}_{ij}}:&=\Phi(\varphi_{j}, \chi_{x_j,A_{ij}})
\end{align*}
But since we have $  \beta_{x_i,A_{ij}}= \beta_{x_j,A_{ij}}$, and hence $  \chi_{x_i,A_{ij}}= \chi_{x_j,A_{ij}}$, the conclusion follows from Lemma \ref{Gluinglemma2}. Finally we obtain a morphism of $ \kappa $-schemes 
$$\theta: \mathrm{I}_+\longrightarrow \mathcal{D}_1.$$ 
\textbf{Step 4: $ \theta $ is independent of the affine open covering of $ \mathcal{S}$}

This follows from a standard argument. 
Given another affine open covering of $ \mathcal{S} $ satisfying the requirements \eqref{Covering}, which by the construction above, gives a morphism $ \theta_1:  \mathrm{I}_+\rightarrow \mathcal{D}_1$. Note that the union of these two affine open coverings is again an affine open covering of $ \mathcal{S} $, satisfying $ (1), (2) $ in Section \ref{Covering}, and hence induces another morphism $ \theta_2:  \mathrm{I}_+\rightarrow \mathcal{D}_1$. But we have then 
$ \theta=\theta_2,  \theta_1=\theta_2, $
and therefore $ \theta=\theta_1 $. This finishes the proof.


\begin{remark}
	\begin{enumerate}[(1)]
		\item 	We remark here that to give the morphism $ \theta: \mathrm{I}_{+}\to \mathcal{D}_{1} $ one may simply use abstract formulas for $ \Phi_0(\varphi_i, x_i)$ as in \eqref{KeyFormula}, without relating to adapted deformations (or Kim-Kisin windows), as it seems to be much more direct. In fact, one shall see this more clearly in Section \ref{DescriptionofThetaongemetricpoints} below. Nonetheless, the connection with  adapted deformations provides a  conceptual interpretation of the abstract objects $ \Phi_0(\varphi_i, x_i) $, and hence of the morphism $\theta: \mathrm{I}_{+}\to \mathcal{D}_{1}$. 
		\item Consider the morphism $x: \Spec A\to \mathbb{I}_{+}$ as in Step 1 (here we omit the subscript $ i $), that gives a $ p $-divisible group $ H=\mathcal{A}_x[p^{\infty}] $ over $ A $. A natural question one may ask is: is the adapted deformation $ \underline{\mathcal{M}} $ ``the" Kim-Kisin window corresponding to $ H $? The answer is ``yes'', up to an isomorphism in $ \mathbf{Win}(\mathfrak{S}, \nabla^0) $. This follows from Corollary \ref{CorollaryBaseChangeofKisinWindow}. But note that the construction of $ \underline{\mathcal{M}} $ really depends on the point $ x\in \mathbb{I}_+(A) $ (not only the induced point in $\mathcal{S}(A)$), and hence is not canonical. As a complement of the remark in (1), it is necessary to point out that it is the explicit constructions of the adapted deformations that allows us to do the computation in Lemma \ref{Gluinglemma2} so that we can eventually  construct the map $ \theta: \mathrm{I}_{+}\to \mathcal{D}_{1} $.
	\end{enumerate}
\end{remark}

\subsubsection{Description of $ \theta: \mathrm{I}_{+}\to \mathcal{D}_{1} $ on geometric points}\label{DescriptionofThetaongemetricpoints}
Let $ k $ be an algebraically closed field extension of $ \kappa $. Write $ \mathfrak{S}=\mathfrak{S}(W(k))$ and $\mathfrak{S}_0$ for its reduction modulo $ p $. Given an element $ x\in \mathrm{I}_{+} (k)$, a lift $ \tilde{x}\in \mathbb{I}_{+}(W(k)) $ of $ x $ gives an isomorphism 
\begin{align}\label{Tildeofx}
	\beta_{\tilde{x}}: \Lambda^{*}_{W(k)}\cong M:=\mathbb{D}^{*}(\mathcal{A}_{x}[p^{\infty}])(W(k))\cong \HH^{1}_{\cris}(\mathcal{A}_{\tilde{x}}/W(k))
\end{align}
compatible with torsors and filtrations. Similarly we have $ \beta_{x} = \beta_{\tilde{x}}\otimes\id_{k}$. As in \eqref{longformula} we can then form
\begin{align}\label{longformulaongeometricpoints}
	\Phi_0(\varphi, \tilde{x}): =\big((\beta_x\otimes\id_{\mathfrak{S}_{0}})^{-1}\circ \Phi_0(\varphi, \chi_{\tilde{x}})\circ \varphi(\beta_x\otimes\id_{\mathfrak{S}_0})\big)^{\vee}\in \mathcal{D}(k),
\end{align}
where $ \varphi : W(k) \to W(k)$ is the unique Frobenius lift of $ W(k) $. 
\begin{lemma}
	The following holds: 
	\begin{align}\label{MapThetaonpoints}
		\theta(k)(x):=\overline{\Phi_{0}(\varphi, \tilde{x})},
	\end{align}
	where $\overline{\Phi_{0}(\varphi, \tilde{x})}$ denotes the image in $ \mathcal{D}_{1}(k) $ of $ \Phi_{0}(\varphi, \tilde{x}) $. 
\end{lemma}
\begin{proof}
	The corresponding morphism $ x: \Spec k\to \mathrm{I}_{+} $ factors through some $ \Spec A_{i,0} $ in the covering $\{\Spec A_{i,0} \to \mathrm{I}_{+} \}$ of $ \mathrm{I}_{+} $ (cf. \eqref{Covering}).  We may drop the subscript $ i $ in $ A_{i,0} $ and write $ A_0=A_{i,0} $. We fix a Frobenius lift $ \varphi_{A}: A\to A $ of $ A $. Then there is a unique homomorphism of simple frames $ (A, \varphi_{A})\to (W(k), \varphi) $ over $ W(\kappa) $ lifting the ring homomorphism $ A_{0}\to k $ corresponding to $ x $. Denote by $ \breve{x}: \Spec W(k)\to \mathbb{I}_{+} $ the corresponding lift of $ x $. Let $ \mathbf{\tilde{x}}: \Spec A\to \mathbb{I}_{+}  $ be the structure morphism. By our construction of $ \theta $ in Theorem \ref{Maintheorem1}, $ \theta (x)$ is defined to be the image in $  \mathcal{D}_{1}(k)$ of 
	$$ \Phi_{0}(\varphi_{A}, \mathbf{\tilde{x}}) \otimes_{\mathfrak{S}(A_{0})}\mathfrak{S}(k)\in \mathcal{D}(k).$$ 
	
	But by the functoriality of the formation of $ \Phi_{0}(\varphi_{A}, \chi_{\mathbf{\tilde{x}}}) $ we have 
	\begin{align*}
		\Phi_{0}(\varphi_{A}, \mathbf{\tilde{x}}) \otimes_{\mathfrak{S}(A_{0})}\mathfrak{S}(k)&=\Phi_{0}(\varphi, \breve{x}).
	\end{align*}
	Now by Lemma \ref{Gluinglemma2}, the two elements in $ \mathcal{D}(k) $, namely $\Phi_{0}(\varphi, \breve{x})$ and $\Phi_{0}(\varphi, \tilde{x})$, have the same image in $ \mathcal{D}_1(k) $. 
\end{proof}

We can also define $ \Gamma_{\tilde{x}}^{\mathrm{lin}} $ and $ \alpha(\tilde{x}) $ as in Lemma \ref{Varphipreservestensors} and have the following formula
\begin{align*}
	\Phi_{0}(\varphi, \tilde{x})=\alpha(\tilde{x})\mu(u)\in \mathcal{D}(k). 
\end{align*}

\subsection{The morphism $ \eta: S \to \mathcal{D}_1 /\mathcal{K}^{\diamond}$ } \label{ConstructionofMorphismeta}

Our main aim in this subsection is to show that the composition morphism of fpqc sheaves $ \mathrm{I}_{+}\xrightarrow{\theta} \mathcal{D}_1 \to \mathcal{D}_{1}/\mathcal{K}^{\diamond}$ is $ P_{+} $-invariant, and hence induces a morphism of fpqc sheaves $\eta: S\to \mathcal{D}_{1}/\mathcal{K}^{\diamond}$. To do this, we first investigate the behavior of $  \theta(k): \mathrm{I}_{+}(k)\to \mathcal{D}_{1}(k) $ under the $ P_{+}(k) $-action on $ \mathrm{I}_{+}(k) $, where as in Section \ref{DescriptionofThetaongemetricpoints}, $ k $ is an algebraically closed field extension of $ \kappa $.  

The following is a continuation of our discussion in Section \ref{DescriptionofThetaongemetricpoints}. We will need some group theoretic results first. 

Let $ \mathcal{U}_{+} $ be the unipotent radical of $ \mathcal{P}_{+} $ and fix  a maximal torus $ \mathcal{T} $ of $ \mathcal{G}_{W(k)} $ such that the cocharacter $ \tilde{\chi}_{W(k)}: \mathbb{G}_{m, W(k)} \to \mathcal{G}_{W(k)}$ factors through $ \mathcal{T} $. Let $ \mathcal{M} $ be the centralizer of $ \tilde{\chi}_{W(k)} $ in $ \mathcal{P}_+ $, which is a Levi subgroup of $ \mathcal{P}_{+} $. Note that $ \mathcal{T} $ necessarily splits over $ W(k) $ since $ W(k) $ is strictly Henselian.  
Let $\Phi$ be the set of roots of $\mathcal{G}_{W(k)}$ with respect to $\mathcal{T}$.  Recall that for every root $\alpha\in \Phi,$ there is an embedding of algebraic groups
$\mathcal{U}_{\alpha}: \mathbb{G}_{a, W(k)}\to G_{W(k)}$, unique up to an element in $W(k)^{\times}$ (here $W(k)^{\times}$ acts on $\mathbb{G}_{a, W(k)}$ by multiplication), such that if $A$ is a $W(k)$-algebra we have
$t\mathcal{U}_{\alpha}(x)t^{-1}=\mathcal{U}_{\alpha}\big(\alpha(t)x\big)$ for all $x\in \mathbb{G}_a(A)=A$ and all $t\in \mathcal{T}(A)$. We will also use the notation $ \mathcal{U}_{\alpha} $ for the image of $  \mathcal{U}_{\alpha}  $. 
Then we have the following isomorphisms of closed subschemes of $\mathcal{G}_{W(k)}$,
\begin{equation}\label{EqrootsIsomorphism}
\prod_{\langle\alpha, \chi\rangle>0}\mathcal{U}_{\alpha}\xlongrightarrow{\sim} \mathcal{U}_{+, W(k)}, \ \ \ \  \ \ \ \  \prod_{\langle\alpha, \chi\rangle<0}\mathcal{U}_{\alpha}\xlongrightarrow{\sim} \mathcal{U}_{-, W(k)},
\end{equation}
sending an element $ (u_{\alpha})_{\alpha}$ in $\prod_{\langle\alpha, \chi\rangle>0 }\mathcal{U}_{\alpha}  $ to the product $\prod_{\langle\alpha, \chi\rangle>0}u_{\alpha}\in \mathcal{U}_{+, W(k)}$, where $\alpha\in \Phi$ runs over all possible roots of $\mathcal{G}_{W(k)}$, and where the products can be taken in any fixed order; the second isomorphism in \eqref{EqrootsIsomorphism} is defined in a similar way. In particular, for each $\alpha_1, \alpha_2\in \Phi$ with $\langle \alpha_1, \tilde{\chi}_{W(k)} \rangle>0$, $\langle \alpha_2, \tilde{\chi}_{W(k)} \rangle<0$ we have
\begin{align}\label{Eqk}
	\tilde{\chi}(p)\mathcal{U}_{\alpha_1}(W(k))\tilde{\chi}(p)^{-1}\subset K_1(k), \ \ \ \ 	\tilde{\chi}(p)^{-1}\mathcal{U}_{\alpha_2}(W(k))\tilde{\chi}(p)\subset K_1(k),
\end{align}
where $ K_1(k) $ denotes the kernel of the reduction modulo $ p $ map $ \mathcal{G}(W(k))\xrightarrow{\mod p} G(k) $. 

Now a combination of \eqref{EqrootsIsomorphism} and \eqref{Eqk} gives the inclusions
\begin{align}\label{Eqxiaoyue}
	\tilde{\chi}(p)\mathcal{U}_{+}(W(k))\tilde{\chi}(p)^{-1}\subset K_1(k), \ \ \ \ 	\tilde{\chi}(p)^{-1}\mathcal{U}_{-}(W(k))\tilde{\chi}(p)\subset K_1(k).
\end{align}

\begin{lemma}\label{Computationfor c}
	Let $ p_{+}\in  P_+(k) $ be such that $ (x, p_{+})\in (\mathrm{I}_{+}\times_{S}P_{+})(k) $.  Then there exists an element $ c\in \mathcal{K}^{\diamond}(k) $ such that $ \theta(x\cdot p_{+}) =  \theta(x)\cdot c$. \footnote{We intentionally mention the element $ c $ as it will serve for Theorem \ref{MainTheorem2} below.} In particular, the map $  \mathrm{I}_{+}(k)\xrightarrow{\theta(k)} \mathcal{D}_{1}(k) \to  \mathcal{D}_{1}(k)/\mathcal{K}^{\diamond}(k) $ is $ P_{+} $-invariant, and hence induces a map
	\begin{align*}
		\eta(k): S(k)\to \mathcal{D}_{1}(k)/\mathcal{K}^{\diamond}(k). 
	\end{align*}
\end{lemma}

\begin{proof}
	Note that we have a decomposition of algebraic groups over $ \kappa $, $ P_{+}=U_{+}\rtimes M$, and hence we may uniquely write $p_{+}=u_{+}m$, with $ u_+\in U_+(k) $, $ m\in M(k) $.  We shall find such a $ c $ in the statement of the lemma by calculation.  By \ref{equation decompostion}, if we write $g_{\tilde{x}}=\big(F_{\tilde{x}}^{\mathrm{lin}}\big)^{\vee}$ and  $\Delta_{\tilde{x}}=\big(\Gamma^{\mathrm{lin}}_{\tilde{x}}\big)^{\vee}$, then we have $g_{\tilde{x}}= \Delta_{\tilde{x}}\tilde{\mu}(p)\in \mathcal{G}(W(k)[\frac{1}{p}])$, where $ \Delta_{\tilde{x}} $ lies in $ \mathcal{G}(W(k)) $ and the reduction modulo $ p $ of $  \Delta_{\tilde{x}} $ is $ \alpha(\tilde{x}) $.

	Let $ \tilde{p}_{+}= \tilde{u}_+\tilde{m}\in (\mathcal{U}_{+}\rtimes \mathcal{M})(W(k))$ be a lift of $ p_{+}$, such that $ (\tilde{x}, \tilde{p}_{+}) $ lies in $ (\mathbb{I}_{+}\times_{\mathcal{S}}\mathcal{P}_{+})(W(k)) $.
	Then we have 
	\begin{align*}
		g_{\tilde{x}\tilde{p}_{+}}&=(\tilde{p}_{+})^{-1}g_{\tilde{x}}\sigma(\tilde{p}_{+})&\\
		&=(\tilde{p}_{+})^{-1}\Delta_{\tilde{x}}\tilde{\mu}(p)\sigma(\tilde{u}_{+})\sigma(\tilde{m})&\\
		&=(\tilde{p}_{+})^{-1}\Delta_{\tilde{x}}\sigma(\tilde{m})\Omega\tilde{\mu}(p), \ \ \ \text{where}\\
		\Omega: &= \tilde{\mu}(p)\sigma(\tilde{m})^{-1}\sigma(\tilde{u}_{+})\sigma(\tilde{m})\tilde{\mu}(p)^{-1}.
	\end{align*}

	It follows from \eqref{Eqxiaoyue} that $ \Omega $ actually lies in $ K_1(k) $. Hence we have 
	\begin{align*}
		\theta(xp_{+})=p_{+}^{-1}\alpha(\tilde{x})\sigma(m)\mu(u) \in \mathcal{D}_{1}(k).
	\end{align*} 
	To continue our calculation, we need to use the equicharacteristic analogue of \eqref{Eqxiaoyue}. To be precise, we have 
	\begin{align*}
		\mu(u)u_{+}\mu(u)^{-1}\in \mathcal{K}_{1}(k) \ \text{for every element} \  u_{+} \in U_{+}(k) \ \ \ \ (*).
	\end{align*}
	
	We have then
	\begin{align*}
		\theta(xp_{+})&=p_{+}^{-1}\alpha(\tilde{x})\sigma(m)\mu(u)\\
		&=\alpha(\tilde{x})\mu(u)\sigma(u_{+})^{-1}\cdot(1, \ p_{+})\\
		&= \alpha(\tilde{x})y\alpha(\tilde{x})^{-1}\theta(x)\cdot(1, \ p_{+})\\
		&=\theta(x)\cdot \big(\alpha(\tilde{x})y^{-1}\alpha(\tilde{x})^{-1}, \ p_{+}\big),
	\end{align*}
	where we let $y:=\mu(u)\sigma(u_{+})^{-1}\mu(u)^{-1}$ and  ($*$) is used to deduce that the element $\alpha(\tilde{x})y^{-1}\alpha(\tilde{x})^{-1}$ lies in $ \mathcal{K}_{1}(k) $. We write 
	\begin{align*}
		d({\tilde{x}}, p_{+}):=\alpha(\tilde{x})y\alpha(\tilde{x})^{-1}=\Phi_0(\varphi, \tilde{x})\sigma(u_{+})\Phi_0(\varphi, \tilde{x})^{-1}.
	\end{align*}
	Finally, we may take 
	\begin{align}\label{Formulafor c} 
		c= \big(d({\tilde{x}}, p_{+}), \ p_{+}\big)\in \mathcal{K}^{\diamond} (k).
	\end{align}
\end{proof}

\begin{theorem}\label{MainTheorem2}
	The composition of morphisms of fpqc sheaves $ \mathrm{I}_{+}\to \mathcal{D}_{1}\to \mathcal{D}_{1}/\mathcal{K}^{\diamond} $ is $ P_{+} $-invariant, and hence induces a morphism of fpqc sheaves, $ \eta: S\to \mathcal{D}_1/\mathcal{K}^{\diamond} $, fitting the following commutative diagram
	\begin{align*}
		\xymatrixcolsep{5pc}\xymatrix{	\mathrm{I}_{+}\ar[r]^{\theta}\ar[d]&\mathcal{D}_{1}\ar[d]\\
			S\ar[r]^{\eta}&\mathcal{D}_{1}/\mathcal{K}^{\diamond}.}
	\end{align*}
\end{theorem}
\begin{proof}
	To keep coherence with notations we write $x: \mathrm{I}_{+}\times_{\kappa}P_{+}\to \mathrm{I}_{+} $ for the first projection and $ p_{+}: \mathrm{I}_{+}\times_{\kappa}P_{+}\to P_{+}  $ for the second projection and $ x\cdot p_{+}: \mathrm{I}_{+}\times_{\kappa}P_{+}\to \mathrm{I}_{+} $ for the $ P_{+} $-action on $ \mathrm{I}_{+} $. We need to show the commutativity of the diagram
	\begin{align}\label{P_+invariantcommutativediagram1}
		\xymatrixcolsep{5pc}\xymatrix{	\mathrm{I}_{+}\times_{\kappa}P_{+}\ar[r]^{\theta\circ (x\cdot p_{+})}\ar[d]^{\theta \circ x}&\mathcal{D}_{1}\ar[d]\\
			\mathcal{D}_{1}\ar[r]&\mathcal{D}_{1}/\mathcal{K}^{\diamond}.}
	\end{align}
	Since each $ \Spec A_{i,0} $ in the covering $ \{\Spec A_{i,0}\to \mathrm{I}_{+}\}_{i\in I} $ is $ P_{+} $-invariant, we may assume that $ \mathrm{I}_{+}=U $, where  $ U$ runs over all $\Spec A_{i,0} $ in the covering.  For the discussion below we may suppress the subscript $ i $ in $ A_{i,0}, \theta_{i} $, etc.  We fix a Frobenius lift $ \varphi_{A}: A\to A $ of $ A $ and write $ \mathbf{\tilde{x}} : \Spec A\to \mathbb{I}_{+}$ for the structure morphism.
	
	We define an element in $ \mathcal{L}G(U\times_{\kappa}P_{+}) $ as follows
	\begin{align*}
		d=\Phi_0(\varphi_{A}, \mathbf{\tilde{x}})\sigma(u_{+})\Phi_0(\varphi_{A}, \mathbf{\tilde{x}})^{-1}.
	\end{align*}
	where $ \sigma(u_{+}) $ is the morphism 
	\[ U\times_{\kappa}P_{+}\xrightarrow{p_{+}} P_{+}\to U_{+}\to U_{+}^{(p)}, \]
	and where $ \Phi_0(\varphi_{A}, \mathbf{\tilde{x}}) \in \mathcal{D}(U)$ is defined in 
	\eqref{longformula}, but here we see it as an element in $\mathcal{D}(U\times_{\kappa}P_{+})$ by precomposing the first projection $ x: U\times_{\kappa}P_{+}\to U $. 
	
	\textbf{Claim}: $ d $ lies in $ \mathcal{K}_{1}(U\times_{\kappa}P_{+}) $. 
	
	Indeed, since $ U\times_{\kappa}P_{+}$ is reduced and $ \mathcal{K}_{1} $ is a closed subscheme of $ \mathcal{L}G $, the morphism $ d: U\times_{\kappa}P_{+}\to \mathcal{L}G $ lies in $ \mathcal{K}_{1} $ if and only if its set theoretic image lies in $ \mathcal{K}_{1} $. But we have seen from the proof of Lemma \ref{Computationfor c} that the image of $ d $ on geometric points do lie in $ \mathcal{K}_{1} $. Hence the claim is proved.
	
	Denote by 
	$$ \Gamma_{\theta\circ (x\cdot p_{+})}, \ \Gamma_{\theta\circ x}: U\times_{\kappa}P_{+}\to \mathcal{D}_{1}\times_{\kappa}(U\times_{\kappa}P_{+}) $$ 
	the graph of $ \theta\circ (x\cdot p_{+}),\  \theta\circ x:  U\times_{\kappa}P_{+}\to \mathcal{D}_{1} $ respectively.  We shall show the commutativity of the following diagram
	\begin{align}\label{P_+invariantcommutativediagram1}
		\xymatrixcolsep{5pc}\xymatrix{& \mathcal{D}_{1}\times_{\kappa}(U\times_{\kappa}P_{+}) \ar[dd]^{\cdot c}\\
			U\times_{\kappa}P_{+}\ar[ur]^{\Gamma_{\theta\circ (x\cdot p_{+})}}\ar[dr]^{\Gamma_{\theta\circ x}}&\\
			& \mathcal{D}_{1}\times_{\kappa}(U\times_{\kappa}P_{+}), }
	\end{align}
	where  the right vertical morphism $ \cdot c $ is the translation by
	$$ c=(d, p_{+})\in \mathcal{K}^{\diamond}(U\times_{\kappa}P_{+}), $$ 
	which is an automorphism of $ \mathcal{D}_{1}\times_{\kappa}(U\times_{\kappa}P_{+}) $ over  $ U\times_{\kappa}P_{+} $. Note that the commutativity of \eqref{P_+invariantcommutativediagram1} implies that $ \theta\circ (x\cdot p_+) = \theta \circ x$ since we also have the commutative diagram
	\begin{align*}
		\xymatrixcolsep{4pc}\xymatrix{\mathcal{D}_{1}\times_{\kappa}(U\times_{\kappa}P_{+})\ar[r]^{\mathrm{pr}_1}\ar[dd]^{\cdot c}&\mathcal{D}_{1}\ar[dr]&\\
			&&\mathcal{D}_{1}/\mathcal{K}^{\diamond}.\\
			\mathcal{D}_{1}\times_{\kappa}(U\times_{\kappa}P_{+})\ar[r]^{\mathrm{pr}_1}&	\mathcal{D}_{1}\ar[ur]
		}
	\end{align*}
	But as  $ U\times_{\kappa} P_{+} $ and $ \mathcal{D}_{1}\times_{\kappa}(U\times_{\kappa}P_{+}) $ are of finite type over $ \kappa $ and $U\times_{\kappa} P_{+}$ is smooth, hence geometrically reduced, it is enough to check the commutativity of \eqref{P_+invariantcommutativediagram1} on geometric points (see for example Exercise 14.15 of \cite{GoertzWedhornBook1}). Now the conclusion follows from Lemma \ref{Computationfor c} and the formula for the $ c $ in \eqref{Formulafor c} 
\end{proof}

\section{Ekedahl-Oort stratification}

The definition of the Ekedahl-Oort stratification for a general Shimura variety of Hodge type is based on the theory of $ G $-zips developed in \cite{PinkWedhornZigler1} and \cite{PinkWedhornZiegler2}. We review first the definition and basic properties of $G$-zips. Notations are as in Section \ref{GoodreductionofShimuravariety} and Section \ref{Sectionofcocharacters}. 

\subsection{$G$-zips of a certain type $\chi$} 
\begin{definition}[{\cite[3.1]{PinkWedhornZiegler2}}] Let $ T $ be a scheme over $ \kappa $. 
	\begin{enumerate}[(1)] 
		\item 	 \textbf{A $G$-zip of type} $\chi$ over $T$ is a quadruple $\underline{\mathrm{I}}= (\mathrm{I}, \mathrm{I}_+, \mathrm{I}_-, \iota)$ consisting of a right $G$-torsor $\mathrm{I}$ over $T$, a $P_+$-torsor $\mathrm{I}_+\subset \mathrm{I}$, and $P_-^{(p)}$-torsor $\mathrm{I}_-\subset \mathrm{I}$, and an isomorphism  of $M^{(p)}$-torsors: 
		$$\iota: \mathrm{I}_+^{(p)}/U_{+}^{(p)}\cong \mathrm{I}_-/U_-^{(p)}.$$
		\item A morphism $\underline{\mathrm{I}} \to \underline{\mathrm{I}}'=(\mathrm{I}', \mathrm{I}_+', \mathrm{I}_-', \iota')$ of $G$-zips of type $\chi$
		over $T$ consists of an equivariant morphism $\mathrm{I}\to \mathrm{I}'$ which sends $\mathrm{I}_+$ to $\mathrm{I}_+'$ and $\mathrm{I}_-$ to $\mathrm{I}_-'$,  and which is compatible with the isomorphisms $\iota $ and $\iota'$. 
	\end{enumerate}
	
\end{definition}

With the natural notion of pullback, the $G$-zips of type $\chi$, as $T$ varies, form a category fibred in groupoids  over the category of ${\kappa}$-schemes. We denote this fibred category by $G\textsf{-Zip}^{\chi}$.  It is shown in \cite[Proposition 3.1]{PinkWedhornZiegler2} that $G\textsf{-Zip}^{\chi}$ is a stack over $ {\kappa} $.   

The stack $G\textsf{-Zip}^{\chi}$ has an interpretation as an algebraic stack as below. To $G$ and $\chi$ one can also associate a natural \textbf{algebraic zip datum} which we shall not specify (see \cite[Definition ]{PinkWedhornZiegler2}), and to such an algebraic datum there is an associated \textbf{zip group} $E_{\chi}:=E_{G, \chi}$, given on points of a $\kappa$-scheme $T$ by
\begin{align*}
	E_{\chi}(T):=\{(mu_+, m^{(p)}u_-)\big |m\in M(T),  u_+\in U_+(T), u_-\in U_-^{(p)}(T)\} \subset P_+(T)\times P_-^{(p)}(T).
\end{align*}
The group $E_{\chi}$ is a linear algebraic group over $\kappa$ and it has a left action  on $G_{\kappa}$ by $(p_+, p_-)\cdot g: = p_+g p_-^{-1}$. With respect to this action one can form the quotient stack $[E_{\chi}\backslash G_{\kappa}]$. 

\begin{theorem}[{\cite[Proposition 3.11, Corollary 3.12]{PinkWedhornZiegler2}}] The stacks $G\textsf{-Zip}^{\chi}$ and $[E_{\chi}\backslash G_{\kappa}]$ are naturally equivalent. They are smooth algebraic stacks of dimension $0$ over $\kappa$. In particular, the isomorphism class of $G$-zips over any algebraically closed field extension $k$ of $\kappa$ are in bijection with the $E_{\chi}(k)$-orbits in $G(k)$. 
\end{theorem}

\begin{remark}
	This realizes $G\textsf{-Zip}^{\chi}$ as an algebraic quotient stacks. While the language of $G$-zips is more ``motivic'' (see \cite[Chapter 0, 0.3]{ChaoZhangEOStratification}), the realization as a quotient stack will help to give a combinatory description of the  topological space of $G\textsf{-Zip}^{\chi}$ (see the next section).
\end{remark} 

\subsection{Combinatory description of the topological space $G$-$\mathsf{Zip}^{\chi}$}\label{CombinatoryWeylgroups}
Let $ k	 $ be an algebraically closed field extension of $ \kappa $. We shall see that the topological space of $ [E_{\chi}\backslash G_{\kappa}] $ is identified with a subset 
$ {}^JW $ of the Weyl group $W$, which will later play the part of ``index set'' of the Ekedahl-Oort stratification. 

As a preparation for what to follow, we let $T\subset B$ be a maximal torus, respectively a Borel subgroup of $ G_k $. Let $ W=N_{G}T(k)/T(k) $ be the finite Weyl group of $ G_k $ with respect to $ T_k $, and let $ I:=I(B, T) \subset W$ be the set of simple reflections associated to the pair $ (B, T) $. As the pair is unique up to a conjugation by $ G(k) $, the Coxeter system $ (W, I) $ is, up to unique isomorphism, independent of the choice of the pair $ (B, T) $ (see \cite[Section 2.3]{PinkWedhornZigler1} for a detailed explanation). Moreover, for any field extension $ k\to k' $, the canonical map $$ \big(W(G_k, B, T), I(G_k, B, T)\big)\to \big(W(G_{k'}, B_{k'}, T_{k'}), I(G_{k'}, B_{k'}, T_{k'})\big)$$ is an isomorphism. 

Recall that the \textbf{length} of an element $ w\in W $ is the smallest number $ l(w) $ such that $ w $ can be written as a product of $ l(w) $ simple reflections. By definition of the \textbf{Bruhat order }$ \leq $ on $ W $, we have $ w'\leq w $ if and only if for some (equivalently, for any) reduced expression of $ w $ as a product of simple reflections, by leaving out certain factors one can get an expression of $ w' $.

For any subset $ J$ $\subset I$ we denote by $ W_J $ the subgroup of $ W $ generated by $ J $. For any $ w\in W $, there is a unique element in the left coset $ W_Jw $ (resp. right coset $wW_J$) which is of minimal length; if $ K\subset I $ is another subset, then there is a unique element in the double coset $ W_JwW_K $ which is of minimal length. We denote by $ {}^JW$ (resp. $W^K$, resp. ${}^JW^K$) the subset of $ W $ consisting of $ w\in W $ which are of minimal length in $ W_Jw $ (resp. $ wW_K $, resp. $ W_JwW_K $). Then we have $ {}^JW^K= {}^JW\cap W^K$ and $ {}^JW $ (resp. $W^K$, resp. ${}^JW^K$) is a set of representatives of $W_J \backslash W $ (resp. $W/W^K$, resp. $W_J\backslash W/W_K$). Denote by $ w_0, w_{0, J} $ the unique element of maximal length in $ W $ and in $ W_J $ respectively, and let $ x_J:=w_0w_{0, J} $. For any $ w\in W $, we denote by $ \dot w$ a lift of $ w $ in $ N_GT(k) $. 

\begin{definition}
	For any two elements $ w, w' \in {}^JW $, we write $w'\preceq w $ if there exists a $y\in W_J $, such that $y^{-1}w'\sigma (x_Jyx_J^{-1})\leq w$. 
\end{definition}The relation $ \preceq $ defines a partial order on $ {}^JW $  (\cite[Corollary 6.3]{PinkWedhornZigler1}). This partial order ``$ \preceq $'' defined above induces a unique topology on $ {}^{J}W $.  

For any standard parabolic subgroup $ P $ of $ G_k $ (the word ``standard'' means $ P$ contains $ B $),  there is an associated subset $ J:=J(P) \subset I$, defined as $J=\{g\in I\big| gT(k)\subset P(k)\}$, called the \textbf{type} of $ P $.  If we view $ I $ as the set of simple roots of $ G $ with respect to the pair $ (B, T) $, then the type $ J $ of $ P $ is simply the set of simple roots whose inverse are roots of $ P $. It is a basic fact that two parabolic subgroups of $ G_k $ are conjugate if and only if they have the same type. For a cocharacter $ \lambda: \mathbb{G}_{m,k}\to G_k $ over $ k $, the type of $ \lambda $ is defined to be the type of the parabolic subgroup $ P=P_{+}(\lambda) $ (see Section \ref{Parabolics induced by a cocharacter} for the formation of $ P_{+}(\lambda) $). It only depends on the $ G(k) $-conjugacy class of $ \lambda $. 

We let $J:=J(\chi_k) $ be the type of the cocharacter $ \chi_k: \mathbb{G}_{m,k}\to G_{k} $. Recall that the finite set $E_{\chi}(k)\backslash G(k)$ as orbit spaces is naturally equipped with the quotient topology of $ G(k) $. Following \cite[Section 3.5]{PinkWedhornZiegler2} (see also \cite[Section 6.4]{Wortmann} for a clear statement) one can define a map
\begin{align} \label{MapfromJWtoGzips}
	\pi: {}^{J}W\longrightarrow E_{\chi}(k)\backslash G(k)\cong \big|[E_{\chi}\backslash G_{\kappa}]_k\big|, \ \ \  w\longmapsto E_{\chi}(k)\cdot \dot w\dot w_0\sigma(\dot w_{0, J}).
\end{align}
\begin{lemma}\label{TopologicalIso}
	The map $ \pi $ in \eqref{MapfromJWtoGzips} is a bijection preserving partial order relations and hence is a homeomorphism of topological spaces. Hence there is a homeomorphism between $ {}^{J}W $ and $ \big|G\textsf{-Zip}^{\chi}_{k}\big|. $
\end{lemma}

\subsection{Definition of Ekedahl-Oort stratification}\label{SectionDefinitonofEO}

Recall that to give a morphism of stacks over $ \kappa $ from $S$ to $G\textsf{-Zip}^{\chi}$ is equivalent to  give an object in $G\textsf{-Zip}^{\chi}(S)$, namely to give a $G$-zip of type $\chi$ over $S$. C. Zhang defines a classification map of Ekedahl-Oort strata $\zeta: S\to G\textsf{-Zip}^{\chi}$ by constructing a universal $G$-zip over $S$, from which he gives the definition of Ekedahl-Oort stratification of $S$ by taking geometric fibres of $\zeta$ (see Definition \ref{Definition of Ekedahl-Oort Stratification} below). Wortmann gives in \cite{Wortmann} a slightly different construction (replacing the $\mathbb{Z}_p$ lattice $\Lambda$ by its dual $\Lambda^*$).  We recall Wortmann's construction in the following.

Recall that the de Rham cohomology $ \mathcal{V}_0 $ comes with  a Hodge filtration $ \mathcal{C}_0 $. Denote by $\mathcal{D}_{0}\subset \mathcal{V}_0$ the conjugate filtration of $ \mathcal{V}_0 $. Recall that we have fixed an embedding (cf. Section \ref{GoodreductionofShimuravariety})
$$\iota: G\hookrightarrow \GL(\Lambda_{\mathbb{F}_p})\cong \GL(\Lambda_{\mathbb{F}_p}^*).$$

The cocharacters $()^{\vee}\circ \chi$ and $()^{\vee}\circ \chi^{(p)}$ induce $\mathbb{Z}$-gradings (cf. Section \ref{Sectionofcocharacters})
$$\Lambda_{\kappa}^*=(\Lambda_{\kappa}^*)_{\chi}^0\oplus(\Lambda_{\kappa}^*)^1_{\chi}, \ \ \ \ 
\ \ \Lambda_{\kappa}^*=(\Lambda_{\kappa}^*)_{\chi^{(p)}}^0\oplus(\Lambda_{\kappa}^*)^1_{\chi^{(p)}}.$$ From this we get a descending and an ascending filtration
\begin{align*}
	&\Fil_{\chi}^0:=\Lambda_{\kappa}^*\supset \Fil_{\chi}^1:=(\Lambda_{\kappa}^*)^1_{\chi}\supset\Fil_{\chi}^2:=0,\\
	&\Fil_{-1}^{\chi^{(p)}}:=0\subset\Fil_{0}^{\chi^{(p)}}:=(\Lambda_{\kappa}^*)^0\subset \Fil_1^{\chi^{(p)}}:=\Lambda_{\kappa}^*.
\end{align*}
Then $P_+$ (resp. $P_{-}^{(p)}$) is the stabilizer of $\Fil_{\chi}^{\bullet}$ (resp. $\Fil^{\chi^{(p)}}_{\bullet}$) in $G_{\kappa}$. Denote by $\bar{s}_{\dR}\subset \bar{\mathcal{V}}^{\otimes}$ the reduction of the set of tensors $s_{\dR}\subset \mathcal{V}^{\otimes}$, and by $\bar{s}$ the base change to $(\Lambda^*_{\chi})^{\otimes}$ of $s\subset (\Lambda^*)^{\otimes}$. Now we define
\begin{align*}
	&\mathrm{I}:=\Isom_S\big([\Lambda_{\kappa}^*, \bar{s}]\otimes O_{S}, [\mathcal{V}_{0}, \bar{s}_{\dR}]\big),\\
	&\mathrm{I}_+:=\Isom_S\big([\Lambda_{\kappa}^*, \bar{s}, \Fil_{\chi}^{\bullet} ]\otimes O_{S}, [\mathcal{V}_{0}, \bar{s}_{\dR}, \mathcal{V}_{0}\supset\mathcal{C}_{0}]\big),\\
	&\mathrm{I}_-: =\Isom_S\big([\Lambda_{\kappa}^*, \bar{s}, \Fil^{\chi^{(p)}}_{\bullet} ]\otimes O_{S}, [\mathcal{V},\bar{s}_{\dR} , \mathcal{D}_{0}\subset \mathcal{V}_{0}]\big),
\end{align*}
where $ \mathrm{I} $ and $ \mathrm{I}_{+} $ were defined already in Section \ref{TorsorsoverSandS} and here we just repeat the definitions for the convenience of readers.
The group $G_{\kappa}$ acts on $\mathrm{I}$ from the right by $t\cdot g:= t\circ g^{\vee}$ for any $S$-scheme $T$ and any sections $g\in G(T)$ and $t\in \mathrm{I}(T)$. This action induces the action of $P_{+}$ (resp. $P_{-}^{(p)}$) on $\mathrm{I}_{+}$ (resp. $\mathrm{I}_{-}$). The Cartier isomorphism on $\mathcal{V}_0$ induces an isomorphism $\iota: \mathrm{I}_+^{(p)}/U_+^{(p)}\cong \mathrm{I}_{-}/U^{(p)}_{-}$ and it is shown in \cite[Theorem 2.4.1]{ChaoZhangEOStratification} (see \cite[5.14]{Wortmann} for necessary changes of the proof) that $\underline{\mathrm{I}}:=(\mathrm{I}, \mathrm{I}_{+}, \mathrm{I}_{-}, \iota)$ is a $G$-zip of type $\chi$ over $S$, which induces a morphism of stacks $$\zeta: S\longrightarrow G\textsf{-Zip}^{\chi}\cong [E_{\chi}\backslash G_{\kappa}].$$
Moreover, Zhang shows that the morphism $\zeta$ is a smooth morphism of stacks over $ \kappa $ (\cite[Theorem 3.1.2]{ChaoZhangEOStratification}). 

\begin{definition}\label{Definition of Ekedahl-Oort Stratification}
	For any geometric point $ w: \Spec k\to G\textsf{-Zip}^{\chi}$, the \textbf{Ekedahl-Oort stratum} of $S_k$ associated to $ w $, denoted by $ S_k^{w} $, is defined to be the fibre of $ w $ under the classifying morphism $\zeta_k: S_k\to [E_{\chi}\backslash G_{\kappa}]$.
	
	The \textbf{Ekedahl-Oort stratification} of $ S_k $ is by definition the disjoint union 
	\begin{align}\label{Ekedahl-Oort Stratification}
		S_{k}=\bigsqcup_{w\in {}^{J}W} S_{k}^{w} 
	\end{align}
	by taking the fibre of the stratification $$ [E_\chi\backslash G_{\kappa}]_k= \bigsqcup_{w\in {}^{J}W}[E_{\chi}\backslash O^{w}],$$ where each $ O^{w} $ is the $ E_{\chi} $-orbit in $ S_k $ corresponding to $ w\in W $ (cf. Lemma \ref{TopologicalIso}), and where each $ S_k^{w} $ is the fibre  along $ \zeta_k$ of $[E_{\chi}\backslash O^{w}]$. Each $ S_{k}^{w} $ is a locally closed subscheme of $ S_{k} $. 
\end{definition}

\begin{remark} 
	As remarked in \cite[Remark 5.16]{Wortmann} (the proof is actually already given there), though the definition of $\zeta$ depends on the choice of a cocharacter $\chi\in [\chi]_{\kappa}$, the resulting map $\zeta(k): S (k)\to {}^JW $ is in fact independent of the choice of $\chi$. 
\end{remark}
\begin{remark}  \label{basic properties of E-O strata}
	We list here some properties of Ekedahl-Oort stratification following \cite{Wortmann} and \cite{WedhornGeneralizedHass}.
	\begin{enumerate}[(1)]
		\item Each Ekedahl-Oort stratum $S_k^w$ is smooth and hence reduced, since the classifying map $\zeta: S\to [E_{\chi}\backslash G_{\kappa}]$ is smooth. 
		\item Each stratum $S_k^{w}$ is either empty or equidimensional of dimension $l(w)$ (\cite[Proposition 3.1.6]{ChaoZhangEOStratification}).

		\item The closure relationship between  Ekedahl-Oort strata is given by $$\overline{S_k^{w}}= \bigcup_{w'\preceq w}S_k^{w'},$$ where $ \preceq$ is the Bruhat order (\cite[Proposition 3.1.6]{ChaoZhangEOStratification}). This follows from the topological structure of ${}^JW$ and the fact that taking inverse under $\big|\zeta\big|: |S_k|\to \big|[E_{\chi}\backslash G_{\kappa}]_k\big|$ commutes with taking closure as the map $\zeta: S\to [E_{\chi}\backslash G_{\kappa}]$ is smooth and hence is open. 
		
		\item The set $^JW$ contains, with respect to the partial order $\preceq$, a unique maximal element $w_{\mathrm{max}}:=w_{0,J}w_0$, and a unique minimal element $w_{\mathrm{min}}:=1$. By (3), the associated stratum $S_k^{w_{\mathrm{max}}}$ is the unique open Ekedahl-Oort stratum and is dense in $S_k$, and $S_k^{w_{\mathrm{min}}}$ is closed and contained in the closure of each stratum $S_k^{w}$. 
		
		\item For Shimura variety of PEL type  each Ekedahl-Oort stratum is nonempty (\cite[Theorem 10.1]{ViehmannWedhornEOPELtype}). For general Hodge type Shimura varieties, the non-emptiness is claimed by Chia-Fu Yu (the preprint is not available yet). 
	\end{enumerate}
\end{remark}

\subsection{Comparison of $ \mathcal{D}_{1}(k)/\mathcal{K}^{\diamond}(k) $ and $ E_{\chi}(k)\backslash G_{\kappa}(k) $}\label{ComparisonofpointsofTwostacks}

We let $ k $ be a perfect field extension of $ \kappa $.  Denote by $[[g]]$ the $\mathcal{K}^+(k)$-orbit of $g$; that is, we have 
\begin{align*}
	[[g]]&= \{\mathcal{K}_1(k)h^{-1}g\varphi(h) \mathcal{K}_1(k)| h\in \mathcal{K}(k)\}\\
	&=\{\mathcal{K}_1(k)h^{-1}g\sigma(h) \mathcal{K}_1(k)| h\in \mathcal{K}(k)\},
\end{align*}
where the second ``='' is due to Lemma \ref{TwoFrobenii}.

Recall that the functor $\mathcal{C}(G, \mu(u))$, following the notations of Section \ref{Sectionsubfunctor}, is the subfunctor of $\mathcal{L}G$ sending a $\kappa$-algebra $ R $ to the set $\mathcal{K}(R)\mu(u)\mathcal{K}(R)$. For simplicity of notations, we write
$$\mathcal{C}:=\mathcal{C}(G, \mu(u)).$$
We want to define such a map of sets
\begin{equation}\label{MapofEO}
\begin{array}{rcl}
\omega: \mathcal{C}(k)& \longrightarrow & E_{\chi}(k)\backslash G(k)\\
h_1\mu(u)h_2& \longmapsto & E_{\chi}\cdot (\sigma^{-1}(\bar{h}_2)\bar{h}_1),
\end{array}
\end{equation}
where for an element $h\in \mathcal{K}(k)$, we write $\bar{h}= h\mod u$.

The following proposition is an equicharacteristic analogue of \cite[Proposition 6.7]{Wortmann}. We follow his strategy but use a slightly different method to prove Proposition \ref{PropositionofEO}. (1). 

\begin{proposition} \label{PropositionofEO} For any perfect field extension  $ k $ of $ \kappa $, we have the following:
	\begin{enumerate}[(1)]
		\item The map $\omega$ in \eqref{MapofEO} is a well-defined map.
		\item For any $g_1, g_2\in \mathcal{C}(k)$, we have $\rho(g_1)= \rho (g_2 )$ if and only if $[[g_1]]=[[g_2]]$; in other words, the map $ \omega $ induces a bijection 
		\begin{align}\label{Comparision of 2maps on points}
			\omega: \mathcal{C}(k)/\mathcal{K}^{+}(k)= \mathcal{D}_{1}(k)/\mathcal{K}^{\diamond}(k)\to E_{\chi}(k) \backslash G(k). 
		\end{align} 
		\item We have the following commutative diagram 
		\begin{align}\label{TheFinalCommutativediagram}
			\xymatrix{&&\mathcal{D}_{1}(k)/\mathcal{K}^{\diamond}(k)\ar[dd]^{\omega}\\
				S(k)\ar[urr]^{\eta(k)}\ar[drr]_{\zeta(k)}&&\\
				&&E_{\chi}(k)\backslash G(k)}
		\end{align}
		
	\end{enumerate}
\end{proposition}
\begin{proof} 
	
	\begin{enumerate}[(1)]
		
		\item  To show $\rho$ is well defined, it is enough to show that the orbit $E_{\chi}\cdot (\sigma^{-1}(\bar{h}_2)\bar{h}_1)$ is independent of the choice of $h_1$ and $h_2$. Suppose that $$h_1\mu(u)h_2=h_1'\mu(u)h_2', \ \ \ \ h_1, h_2, h_1', h_2'\in \mathcal{K}(k).$$ Write
		$c_1=h_1^{-1}h_1'$ and $c_2=\sigma^{-1}(h_2h_2'^{-1})$, 
		then we have \begin{equation}\label{Eqrelationc1c2}
		c_2(\sigma^{-1}(h_2')h_1')c_1^{-1}= \sigma^{-1}(h_2)h_1, \ \ \ \  \  c_2=y^{-1}\sigma^{-1}(c_1)y,
		\end{equation} where we set $y:=\sigma^{-1}(\mu(u))=\chi(u^{p})$ (the latter equality here follows from Lemma \ref{TwoFrobniiLemma1}). The first part of \eqref{Eqrelationc1c2} implies that we need only to show that $(\bar{c}_2, \bar{c}_1)\in E_{\chi}(k)$.
		But since $E_{\chi}$ is a subscheme of 
		$G\times_{\kappa}G$ and the formation of $P_{\pm}, U_{\pm}$ and $M$ commute with base change of $\chi$, it suffices to show that \'etale locally the point $(\bar{c}_2, \bar{c}_1)\in (G\times_{\kappa}G)(k)$ lies in $E_{\chi}(k)$. Hence by passing to a finite (separable) extension of $\kappa$ we may suppose that the cocharacter $\chi$ factors through a split maximal torus $T$ of $G$. Then a similar discussion as in the beginning of Section \ref{ConstructionofMorphismeta} implies the following inclusions of schemes (cf. \eqref{Eqxiaoyue}) 
		\begin{equation} \label{Keyrelation}
		y\mathcal{L}^{+}U_{+}y^{-1}\subset \mathcal{K}_1, \ \ \ \ \ \ \ \ 	y^{-1}\mathcal{L}^{+}U_{-}y\subset \mathcal{K}_1.
		\end{equation}
		
		By \cite[Expos\'e XXVI]{SGA3} Theorem 5.1, as a scheme over $ \kappa $, $ G_{\kappa} $ is a union of open subschemes $ s\cdot P_{-}P_{+}=s\cdot U_{-}U_{+}M $, where $ s $ runs over $ s\in U_+(k) $.  In particular $c_2\in \mathcal{L}^+G(k)=G(k[[u]])$ as a $ k[[u]] $-point of $ G $ must lie in some $s\cdot U_{-}U_{+}M$. And hence $ c_2 $ is of the form $c_2=u_1vu_2m$, with $$u_1\in U_+(k)\subset \mathcal{L}^{+}U_+(k),\  u_2\in\mathcal{L}^+U_{+}(k), \ v\in \mathcal{L}^+U_{-}(k),\  m\in \mathcal{L}^+M(k).$$ Given such a decomposition, by the second part of \eqref{Eqrelationc1c2} we have 
		\begin{equation}\label{Valuec1}
		yu_1vu_2y^{-1}m=yu_1vu_2my^{-1}=\sigma^{-1}(c_1)\in \mathcal{L}^+G(k)
		\end{equation}
		
		By the first part of \eqref{Keyrelation}, $yu_1y^{-1}, yu_2y^{-1}$ lies in $\mathcal{K}_1(k)$, and hence we have $ yvy^{-1}\in \mathcal{L}^+G(k) $, which implies that the element $ yvy^{-1}\in \mathcal{L}U_{-}(k) $ actually lies in $\mathcal{L}^{+}U_{-}(k)$ since $U_{-}$ is a closed subscheme of $G$. Then we have $v\in \mathcal{K}_1(k)$ by the second part of \eqref{Keyrelation}. Now it is clear that $\bar{c}_2=\bar{u}_1\bar{u}_2\bar{m}$ lies in $P_{+}(k)$ with Levi component $\bar{m}\in M(k)$ and 
		$$\bar{c}_1=\sigma(\overline{yvy^{-1}})\sigma(m) $$ lies in $ \sigma\big(U_{-}(k)M(k)\big)=P_{-}^{(p)}(k) $, with Levi component $\sigma(m)$. 
		
		\item  The ``only if" part is clear from (1). For the ``if'' part,  suppose $\rho(g_1)= \rho (g_2 )$ and we write $x=\mu(u)$. Thanks to (1) and Lemma \ref{TwoFrobenii} we may assume $g_1=h_1x$ and $g_2=h_2x$. Choose $ p_{+}= u_{+}m,\  p_{-}=u_{-}\sigma(m)$ with
		$m\in M(k),\;  u_{+}\in U_{+}(k), \; u_{-}\in U^{(p)}_{-}(k)$
		such that $  \bar{h}_2=p_{+}\bar{h}_1p_{-}^{-1}$, i.e., $ \mathcal{K}_1(k) h_2=\mathcal{K}_1(k)p_{+}h_1p_{-}^{-1} $.
		We have seen from \eqref{Keyrelation} that \[ x^{-1}u_{-}^{-1}x=\sigma(x_1\sigma^{-1}(u_-^{-1})x_1^{-1})\in \mathcal{K}_{1}(k), \;\;\; x\sigma(u_{+})x^{-1}=\sigma(x_1u_+x_1^{-1})\in \mathcal{K}_{1}(k),  \]
		where $ x_1:=\chi(u^{p}). $
		Hence, together with the fact that $\sigma(m)^{-1}$ commutes with $ x $ and $\mathcal{K}_1\subset \mathcal{K}$ is normal, we find 
		\[ \begin{array}{cclcl} 
		[[h_2x]]&=& [[u_{+}mh_1\sigma(m)^{-1}u^{-1}_{-}x]]&=&[[u_{+}mh_1\sigma(m)^{-1}x]]\\
		&=&[[mh_1\sigma(m)^{-1}x\sigma(u_{+})]]&=&[[mh_1\sigma(m)^{-1}x]]\\
		&=&[[mh_1x\sigma(m)^{-1}]]&=& [[h_1x]].
		\end{array}\]

		\item The commutativity of the diagram \eqref{TheFinalCommutativediagram} follows from the construction of $ \eta: S\to \mathcal{D}_{1}/\mathcal{K}^{\diamond} $ and $ \zeta: S\to |E_{\chi}\backslash G_{\kappa}| $.
	\end{enumerate}

\end{proof}

\begin{remark}\label{Comparisionofmus}
	Recall that the cocharacter ``$ \mu $'' in Wortmann is our $ \chi^{(p)} $.  Let $ k $ be an algebraically closed field extension of $ \kappa $ and let $ (B, T) $ be a Borel pair of $ G_{k} $ such that $ \chi_{k} $ factors through $ T_k $ and is dominant. Since  $ T_k $ is necessarily split, the action  
	$$\sigma: X_*(G_k):=\Hom_k(\mathbb{G}_{m,k}, G_{k})\to X_*(G_k) $$
\end{remark}
sending $ \lambda \in X_*(G_k)$ to $ \lambda^{(p)} $ (see \eqref{Actionofsigmaoncocharacters}) is trivial, and hence we  have 
\begin{align}
	\sigma^{-1}(\mu_k)=\mu_k=(\Frob_{G_{\kappa}/\kappa})_k\circ\chi_{k}=\Frob_{G_{k}/k}\circ\chi_{k}=p\chi_{k},
\end{align}
where the last equation is easy to see since $ T_k $ is split over $ k $.    
It follows that  $\sigma^{-1}(\mu_k)$ and $\sigma^{-1}(\chi^{(p)}_k)=\chi_k$ have the same type $ J $ (the type of a cocharacter is insensitive to multiplications).  Recall that $\mathcal{D}_{1}(k)/\mathcal{K}^{\diamond}(k)$ is by definition the $C(G, \mu_k)$ defined in the Introduction.  Now by \cite[Theorem 1.1]{ViehmannTrucation1}, both $C(G, \mu_k)$ and $C(G, \chi^{(p)}_k)$ can be identified with the subset $ {}^JW $ of $ W $. And hence we can also identify $C(G, \chi^{(p)}_k)$ with $C(G, \mu_k)$. This shows that our variation of $ \mu $ is essentially harmless.


\phantomsection

\end{document}